\documentclass[10pt]{amsart}

\usepackage{times,amsfonts,amsmath,amstext,amsbsy,amssymb,
amsopn,amsthm,upref,eucal, amscd}

\usepackage[colorlinks=true]{hyperref}
\usepackage[T1]{fontenc}
\usepackage[backgroundcolor=white]{todonotes}
\usepackage{multirow}
\usepackage{color,tcolorbox,colortbl}
\usepackage[displaymath, mathlines, pagewise]{lineno}

\usepackage{graphicx}
\usepackage{algpseudocode,algorithm}
\usepackage{caption}

\newtheorem{theorem}{Theorem}[section]
\newtheorem{lemma}[theorem]{Lemma}

\newtheorem{proposition}[theorem]{Proposition}

\numberwithin{equation}{section}

\theoremstyle{definition}
\newtheorem{definition}[theorem]{Definition}
\newtheorem{example}[theorem]{Example}

\newtheorem{remark}[theorem]{Remark}

\newtheorem{manualdefinitioninner}{Definition}
\newenvironment{manualdefinition}[1]{%
  \renewcommand\themanualdefinitioninner{#1}%
  \manualdefinitioninner
}{\endmanualdefinitioninner}

\def\leq{\leqslant }
\def\geq{\geqslant}

\def\liftl{\vrule depth0pt height11pt width0pt }

\makeatletter
\newcommand{\@minipagerestore}{\setlength{\parskip}{\medskipamount}
\setlength{\parindent}{12pt}}
\makeatother

\usepackage{accents}

\begin{document}

\title[Phase transition and difference of classical spectra]{Hausdorff dimension
of Gauss--Cantor sets and two applications to classical Lagrange and Markov
spectra} 

\author[C. Matheus, C. G. Moreira, M. Pollicott and P. Vytnova]{Carlos Matheus,
Carlos Gustavo Moreira, Mark Pollicott and Polina Vytnova} 

\address{C. Matheus, CNRS \& \'Ecole Polytechnique, CNRS (UMR 7640), 91128, Palaiseau, France.}
\email{matheus.cmss@gmail.com}

\address{C. G. Moreira, School of Mathematical Sciences, Nankai University, Tianjin 300071, P. R.China, and IMPA, Estrada Dona Castorina 110, CEP 22460-320, Rio de Janeiro, Brazil.}
\email{gugu@impa.br}

\address{M. Pollicott, Department of Mathematics, Warwick University, Coventry,
   CV4 7AL, UK.}
\email{masdbl@warwick.ac.uk}

 \address{P. Vytnova, Department of Mathematics, Warwick University, Coventry,
   CV4 7AL, UK.}
\email{P.Vytnova@warwick.ac.uk}

\thanks{The second author is partly supported by CNPq and FAPERJ. The third author is partly supported by ERC-Advanced Grant 833802-Resonances and EPSRC grant
EP/T001674/1. The fourth author is partly supported by EPSRC grant EP/T001674/1. }

\date{\today}

\begin{abstract}
    This paper is dedicated to the study of two famous subsets of
    the real line, namely Lagrange spectrum~$L$ and Markov spectrum~$M$.
    Our first result, Theorem~\ref{t.A},
    provides a rigorous estimate on the smallest value~$t_1$ 
    such that the portion of the Markov spectrum~$(-\infty,t_1)\cap M$ has Hausdorff
    dimension~$1$. Our second result, Theorem~\ref{t.B}, gives a new upper bound on the Hausdorff
    dimension of the set difference~$M\setminus L$.
    In addition, we also give a plot of the dimension function, which
    hasn't appeared previously in the literature to our knowledge. 

 Our method combines new facts about the structure of the
 classical spectra together with  
finer estimates on the Hausdorff dimension of Gauss--Cantor sets of continued
fraction expansions whose entries satisfy appropriate restrictions. 
\end{abstract}\maketitle


\section{Introduction}
\label{s:intro}
The theory of Diophantine approximations begun with the search
for rational approximations to the solutions of certain algebraic equations
(e.g., $x^2-2=0$, $x^2-x-1=0$, etc.) and well-known mathematical constants
(e.g., $\pi=3.14159265\dots$). Besides its intrinsic beauty, this topic
attracted the attention of several generations of mathematicians thanks to its
deep connections with many other areas including Kolmogorov--Arnold--Moser (KAM)
theory of quasi-periodic motions for Hamiltonian systems (cf. Siegel and Moser
books~\cite{SiegelMoser},~\cite{Moser}) and the spectral theory of certain
quasi-periodic Schr\"odinger operators (cf. Avila--Jitomirskaya's
solution~\cite{AJ} to the ``ten martini problem''). 

The investigation of Diophantine approximations often leads to
the study of the smallest values of quadratic forms on lattices (i.e., the
so-called \emph{geometry of numbers}): for instance, if $\alpha\in\mathbb{R}$
and $p/q\in\mathbb{Q}$, then $|\alpha-p/q| = q^{-2} |h_{\alpha}(p,q)|$ where
$h_{\alpha}(p,q):=q^2\alpha- pq\in h_{\alpha}(\mathbb{Z}^2)$. 
In 1841, Dirichlet used his famous pigeonhole principle to show
that if $\alpha\in\mathbb{R}\setminus\mathbb{Q}$, then
$\#\{p/q\in\mathbb{Q}:|q^2\alpha-pq|<1\}=\infty$, and, subsequently, Hurwitz
improved upon Dirichlet's theorem by proving that if
$\alpha\in\mathbb{R}\setminus\mathbb{Q}$, then
$\#\{p/q\in\mathbb{Q}:|q^2\alpha-pq|<1/\sqrt{5}\}=\infty$, but
$\#\{p/q\in\mathbb{Q}:|q^2\left(\frac{1+\sqrt{5}}{2}\right)-pq|<1/(\sqrt{5}+\varepsilon)\}<\infty$
for all $\varepsilon>0$. More generally, 
it was realised that the finite ``best constants'' of
Diophantine approximations for real numbers and real indefinite binary quadratic
forms are encoded by the \emph{Lagrange spectrum}  
$$L:=\left\{\limsup\limits_{\substack{p,q\to\infty \\ p\in\mathbb{Z},
q\in\mathbb{N}}}\frac{1}{|q^2\alpha-pq|}<\infty: \alpha\in\mathbb{R}\right\}$$ 
and the \emph{Markov spectrum} 
$$M:=\left\{\sup\limits_{\substack{(p,q)\in\mathbb{Z}^2 \\
(p,q)\neq(0,0)}}\frac{1}{|ap^2+bpq+cq^2|}<\infty: ax^2+bxy+cy^2 \textrm{ real
indefinite, }b^2-4ac=1\right\}.$$  

In 1879--1880, Markov~\cite{Mar79},~\cite{Mar80} performed the
first systematic study of the Lagrange and Markov spectra: in particular, he
showed that 
$$
L\cap[\sqrt{5},3) = M\cap[\sqrt{5},3)=\left\{\sqrt{9-4/z_n^2}:n\in\mathbb{N}\right\}
$$ 
where $z_n\in\mathbb{N}$ are \emph{Markov numbers}, i.e., the largest
coordinates of a triple $(x_n,y_n,z_n)\in\mathbb{N}^3$ satisfying the
Markov--Hurwitz cubic equation  
$$
x_n^2+y_n^2+z_n^2=3x_ny_nz_n.
$$ 
\begin{remark} Zagier~\cite{Zagier} showed in 1982 that the set
    of Markov numbers is \emph{sparse} (e.g., $\#\{z_n\leq x:n\in\mathbb{N}\} =
    c(\log x)^2+O(\log x(\log\log x)^2)$), 
    Goldman~\cite{Gol03} observed that the Markov--Hurwitz
    cubic surface is a special example of \emph{character variety}, the Markov
    numbers are known to describe \emph{hyperbolic lengths} of closed geodesics
    on an once-punctured torus~\cite{Mirzakhani}, and the
    reduction modulo $p$ of the Markov--Hurwitz cubic equation leads to an
    interesting family of graphs~\cite{BourgainGamburdSarnak},~\cite{Chen} which
    are conjectured by Bourgain--Gamburd--Sarnak to form an \emph{expander
    family}. 
\end{remark}
In 1921 Perron~\cite{Per21} gave a simple (dynamical)
characterization of the classical spectra in terms of continued fractions. 
Following Perron, see also~\cite{Mor18}, consider the set $(\mathbb N^*)^\mathbb Z$ of bi-infinite sequences
 of elements of $\mathbb{N}^*=\mathbb{N}\setminus\{0\}$. To any element
 $\underline{\alpha} = (\alpha_n)_{n\in\mathbb{Z}} \in (\mathbb N^*)^\mathbb Z$ and each $k\in\mathbb{Z}$, 
we associate a pair of real numbers defined in terms of
continued fraction expansions\footnote{For more details about
the standard relationship between continued fractions, Bernoulli shift, and
Gauss map, see the book of Einsiedler and Ward~\cite{EW}.} 
$$ 
[\alpha_k;\alpha_{k+1},\alpha_{k+2},\ldots]:=\alpha_k+\cfrac{1}{\alpha_{k+1}+\cfrac{1}{\alpha_{k+2}+\cfrac{1}{\ddots}}} 
\qquad \mbox{ and } \qquad 
[0;\alpha_{k-1},\alpha_{k-2},\ldots] = \cfrac{1}{\alpha_{k-1}+\cfrac{1}{\alpha_{k-2}+\cfrac{1}{\ddots}}} ,
$$ 
and consider the map $\lambda_0(\underline{\alpha}): = [\alpha_0;\alpha_{1},\alpha_{2},\ldots] +
[0;\alpha_{-1},\alpha_{-2},\ldots]$. Let us denote by $\sigma$ the Bernoulli
shift on $(\mathbb N^*)^\mathbb Z$ given by $\sigma( (\alpha_n)_{n \in \mathbb
Z} ) = (\alpha_{n+1})_{n\in \mathbb Z}$. The Lagrange value of $\underline{a}$ is the
limit superior of values of $\lambda_0$ along the $\sigma$-orbit of $\underline{a}$: 
$$
\ell(\underline{\alpha}) := \limsup_{n\to \infty} \lambda_0(\sigma^n \underline{a})
=\limsup_{n\to \infty} ([\alpha_n;\alpha_{n+1},\alpha_{n+2},\ldots] +  
[0;\alpha_{n-1},\alpha_{n-2},\ldots])
$$
and the Markov value of~$\underline{\alpha}$ is the supremum of values of $\lambda_0$ along
the $\sigma$-orbit of $\underline{\alpha}$:  
$$
m(\underline{\alpha}) : =  \sup_{n\in \mathbb Z} \lambda_0(\sigma^n \underline{a}) =
\sup_{n\in \mathbb Z}([\alpha_n;\alpha_{n+1},\alpha_{n+2},\ldots] + 
[0;\alpha_{n-1},\alpha_{n-2},\ldots]).
$$ 
In this setting, Perron established that the Lagrange and Markov spectra are the
collections of (finite) Lagrange and Markov values: 
\begin{equation}
    \label{eq:LMdef}
    L := \left\{ \ell(\alpha)\in\mathbb{R} \mid \alpha \in
    (\mathbb N^*)^\mathbb Z \right\} \quad \mbox{ and } \quad
    M := \left\{  m(\alpha)\in\mathbb{R} \mid \alpha \in (\mathbb N^*)^\mathbb Z \right\}. 
\end{equation}
\begin{remark} 
    The values of $n$-ary quadratic forms, $n\geq 3$, can also be studied using
    dynamical ideas: for instance, Margulis famously solved Oppenheim's
    conjecture using higher rank actions on homogenous spaces
    (see~\cite{Margulis} for a nice survey). However, we shall not make further
    comments about this because the techniques in the present paper (inspired
    from \emph{rank one} systems such as the Gauss map and the geodesic flow on
    the modular surface) are fundamentally distinct from the results concerning
    higher rank systems. 
\end{remark}
 
The dynamical result of Perron gives access to many basic
properties of the classical spectra (see, e.g.,~\cite{CF}): for example, it is
known that $L\subset M$ are closed subsets of the real line such that
$\sqrt{12}, \sqrt{13}\in L$ and 
$$
L\cap(\sqrt{12},\sqrt{13})=M\cap(\sqrt{12},\sqrt{13}) = \varnothing.
$$
Moreover, Hall~\cite{Hall47} observed in 1947 that certain portions of $L$ and
$M$ are controlled in terms of the arithmetic sums of Cantor sets of real
numbers whose continued fraction expansions satisfy certain restrictions: for
instance, if $E_4=\{[0;\alpha_1,\alpha_2,\dots]: \alpha_n\in\{1,2,3,4\} \,\,\forall\,
n\in\mathbb{N}\}$, then Hall proved that the arithmetic sum  
$$
E_4+E_4=\{x+y \colon x, y\in E_4\}
$$ 
contains an interval of length $>1$, and this fact was exploited to show that
$[6,\infty)\subset L\subset M$. Actually, since the classical spectra are closed
subsets of the real line, there exists the smallest number~$c_F$ such that $ [c_F, +\infty)\cap L =
[c_F, +\infty)\cap M =  [c_F, +\infty)$: this half-line is usually called
\emph{Hall's ray} in the literature. As it turns out, the value of $c_F$ was
computed explicitly by Freiman in~\cite{Fr75} to be\footnote{All numbers are truncated, not
rounded.}~$c_F=4.5278\dots$ and is called \emph{Freiman's constant}.
\begin{remark} The arithmetic sum $A+B=\{a+b \colon a\in A, b\in B\}$ of $A,
    B\subset \mathbb{R}$ is the image $A+B=\pi(A\times B)$ of the cartesian
    product $A\times B\subset \mathbb{R}^2$ under
    $\pi:\mathbb{R}^2\to\mathbb{R}$, $\pi(x,y)=x+y$. Thus, the results of Hall
    mentioned above point towards a connection between the classical spectra and
    the projections of fractal sets which is going to be relevant in the present
    paper.  
\end{remark}

In contrast to this, the sets $L\cap [3, c]$ and  $M\cap [3, c]$ have a complicated 
and mysterious structure. Nevertheless, some facts have been established. In
particular, it was shown by Hall~\cite{Hall} that  $M\cap [0,
\sqrt{10}]$ has zero Lebesgue measure. A few years later this result was
improved by Pavlova and Freiman~\cite{FP73} (cf.~\cite[Theorem 2, Chapter
6]{CF}), 
when they showed that $M\cap [0, \sqrt{689}/8]$ 
has zero Lebesgue measure\footnote{Note that $\sqrt{689}/8=
3.2811\ldots$ and $\sqrt{10} =  3.162277\dots$}.

More recently, it was shown by the second author in~\cite{Mor18} that for any $t>0$  the sets  
$(-\infty , t]\cap M$ and  $ (-\infty , t] \cap L$ have the same Hausdorff
dimension: 
\begin{equation*}
    \dim_H \left((-\infty , t]\cap M\right)  =\dim_H \left( (-\infty , t] \cap
    L\right)
\end{equation*}
\noindent\begin{minipage}{100mm}
and, moreover, the function 
$$
f(t): = \dim_H \left( (-\infty,t] \cap M)\right)
$$ 
is a continuous non-decreasing function on the real line. 
We now introduce the number which is the subject of our investigations. 

\begin{definition} 
\begin{equation}
    \label{eq:t1def}
    t_1 := \inf\left\{ t\in \mathbb R \mid f(t) = 1 
    \right\}.
\end{equation}
\end{definition}
In view of monotonicity of~$f$ the value~$t_1$
is usually referred to as \emph{the first transition point} of the
classical Lagrange and Markov spectra.
In 1982 Bumby~\cite{Bumby} gave a heuristic estimate  
\begin{equation}\label{e.Bumby-heuristic} 
3.33437 < t_1 < 3.33440,
\end{equation} 
while the results by Hall~\cite{Hall} and the second author~\cite{Mor18} give the best
\emph{rigorous} lower and upper bounds on~$t_1$ to date: 
\begin{equation}
    \label{eq:10-12}
\sqrt{10} = 3.162277\ldots <t_1<\sqrt{12} = 3.464101\dots. 
\end{equation}

Our first result, Theorem~\ref{t.A} confirms Bumby's claim and gives a rigourous estimate of $t_1=3.334384\dots$
The proof is built on ideas developed by Bumby and uses a connection between
Markov values and Gauss--Cantor sets defined in terms of continued fractions of
their elements. 
The argument is computer---assisted and the result could be refined further with 
the method we present, subject to more computer time and resources.  

Using a similar approach, we can also compute a good approximation to
the function~$f$ by solving equations $f(x) = \frac{k}n$ for $ k =
1,2,\ldots n-1$ and $n$ sufficiently large. 
The plot of the resulting function is shown on the right. 
\end{minipage}%
\begin{minipage}{80mm}
  \qquad   \includegraphics{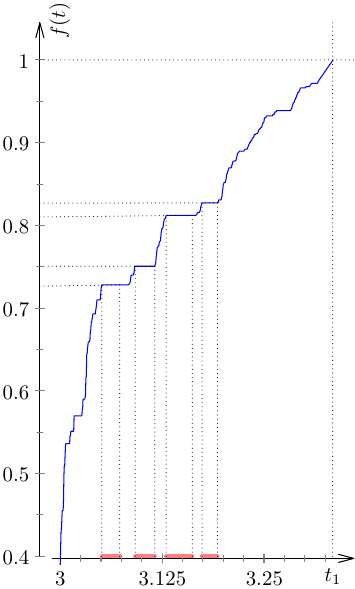}
    \captionof{figure}{A plot of the function~$f$. Pink intervals have been
    identified in~\cite{CF} as the gaps in the Markov spectrum.}
    \label{fig:functionf}
\end{minipage}

Lima and the second author in~\cite{LM} recently
conjectured\footnote{This is motivated by the main results
of~\cite{LM} saying that this conclusion holds in many ``similar'' examples of
dynamical Lagrange and Markov spectra.} that $(t_1, 
t_1+\delta)\cap L $ has non-empty interior for all  $\delta > 0$. Together with
our new result $t_1 = 3.334384\dots$ this would imply, in particular, that
$(3.334384,3.334385)\cap L$ has non-empty interior and thus
would prove an open 
folklore conjecture that the interior of $(-\infty, \sqrt{12})\cap L $ is non-empty. 

The second part of our paper concerns the set difference of the Markov and
Lagrange spectra. It is known that $M\setminus L$ has zero Lebesgue measure.
Furthermore, it was proved
in~\cite{MatheusMoreira} and~\cite{PV20} that the Hausdorff dimension of~$M\setminus L$ satisfies 
\begin{equation*}
0.5312 < \dim_H(M\setminus L) < 0.8823.
\end{equation*}
Our second result, Theorem~\ref{t.B}, shows that the Hausdorff dimension of~$M\setminus L$ has sharper  bounds 
\begin{equation*}
0.537152 < \dim_H(M\setminus L) < 0.796445.
\end{equation*}
The proof is also computer---assisted. Following the approach developed by the first
two authors~\cite{MatheusMoreira}, we use fine-grained combinatorial analysis of
continued fractions to construct a cover 
$M\setminus L$ by arithmetic sums of Gauss--Cantor sets and the so-called
``Cantor sets of the gaps''. We then apply the new method for computing the
Hausdorff dimension 
recently developed by the last two authors~\cite{PV20} to several Gauss--Cantor sets
to obtain sharper upper bounds on~$\dim_H (M \setminus L)$. 

\begin{figure}[htb!]
\includegraphics[width=\textwidth]{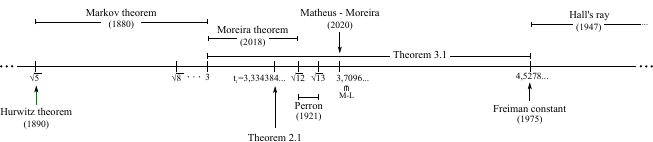}
\caption{Description of several results about the classical spectra (including
our main results, namely Theorems~\ref{t.A} and~\ref{t.B}). }
\label{fig:mainresults}
\end{figure}

 We organize this article as follows. In \S\ref{s.Bumby-improved}, we
 reduce the problem of computing~$t_1$ to the problem of constructing two
 Gauss--Cantor sets~$X$ and~$Y$ such that 
 \begin{equation}
     \label{eq:XY}
 \dim_H X < 0.5 < \dim_H Y, 
 \end{equation}
 and the substrings of $\alpha\in(\mathbb{N}^*)^{\mathbb{Z}}$ with $m(\alpha)$ close to $t_1$ are ``controlled'' by $X$ and $Y$. The conditions that~$X$,~$Y$ and~$t_1$ should jointly satisfy are slightly more subtle and we
 describe them in detail in~\S\ref{ss:aub}.
 Next \S\ref{s.M-L-new} is dedicated to the construction and analysis of the
 arithmetic sums of Gauss--Cantor sets and ``Cantor sets of the gaps'' which
 cover $M\setminus L$, and subsequently allows us to obtain an upper bound on
 $\dim_H (M\setminus L)$. The intricate character of the Gauss--Cantor sets involved in estimates in
 \S\ref{s.Bumby-improved} and~\S\ref{s.M-L-new} means that the algorithm for
 computing the Hausdorff dimension developed in~\cite{PV20} has to be considerably
 adapted and improved. For completeness, in \S\ref{s.PV-new} we explain how the
 Hausdorff dimension of the complicated Gauss--Cantor sets can be
 computed and give some details of the numerical implementation; in
 addition, we also provide pseudocode in the Appendix (with the codes available at https://github.com/Polevita/Gauss\_Cantor\_sets). 
Finally, we discuss in~\S\ref{s.open-problems} a few lines of
future research partly motivated the our main results. 

\begin{remark} On our way to establishing  the results mentioned in the previous
    paragraphs, we encounter some other interesting facts about the structure of
    the classical spectra. For example, Lemma~\ref{l.rational-Lagrange} below 
    says that~$4.5$ is a non-trivial rational point in~$L$ in the sense that
    it occurs after~$3$ and before the beginning $c=4.5278\dots$ of Hall's ray.  
    Hence the value~$4.5$ is realised as the Markov value of two sequences 
    which arise in our study of~$(4.4984,\sqrt{21})\cap(M\setminus L)$. 
\end{remark}

\subsection*{Acknowledgments} 
The authors are thankful to the referee for the valuable comments and
suggestions leading to the current version of this paper. 

\section{Phase transition in classical spectra}\label{s.Bumby-improved}

In this section, we give the theoretical basis for the proof of our first main
result which provides rigorous bounds
on the first transition point~$t_1$ and construct explicitly the relevant
Gauss--Cantor sets. The theoretical background for computer-assisted
calculations which are used to obtain estimates
on Hausdorff dimension of those Gauss--Cantor sets which are constructed here can be found in~\S\ref{s.PV-new}.

\begin{theorem}\label{t.A} $t_1=\inf\{t\in\mathbb{R}:\dim_H((-\infty, t]\cap
    M)=1\}=3.334384\dots$, where this value is accurate to 
    the~$6$ decimal places presented. 
\end{theorem}


\subsection{Preliminaries} We begin by describing the basic strategy to deduce bounds
on~$t_1$ which generalizes the approach of Hall. The first Cantor set we
introduce is relatively famous  and consists of all real numbers whose continued
fraction expansion has only digits~$1$ and~$2$: 
$$
E_2 : = \left\{a = [0;\alpha_1,\alpha_2,\dots] \mid \alpha_j \in\{1,2\}, \ j \ge 1
\right\}.
$$
Its Hausdorff dimension has been computed to high precision (see, for
instance~\cite{JP18}, and references therein), and for our purposes it is
sufficient to know that 
\begin{equation}
\label{eq:dimE2}
\dim_H E_2 > 0.53128. 
\end{equation}
In what follows, we identify a subset $A \subseteq E_2$ with a set of one-sided
sequences corresponding to the continued fraction expansions of its elements. 

In the sequel we use a simple observation that the constant bi-infinite sequence $\beta_n \equiv 1$, $n \in
\mathbb Z$, has the minimal Markov value among all bi-infinite sequences $\alpha
\in \{1,2\}^{\mathbb Z}$, and, moreover, for any $\alpha \ne \beta$ we have
$m(\beta) < m(\alpha)$. 
Straightforward computation gives 
$$
m\left(\beta \right) = \sqrt 5 \le m (\alpha) \quad \mbox{ for any } \alpha \in
\{1,2\}^{\mathbb Z}.
$$

\subsubsection{Approach to lower bound.}
\label{ss:alb}
We fix some threshold $T$ and attempt to construct a finite set of finite ``forbidden'' 
strings $\beta_{-k} \ldots \beta_{-1} \beta_0 \beta_1 \ldots \beta_n$ so
that \emph{all} infinite extensions 
$
\left\{ \alpha \in \{1,2\}^{\mathbb{Z}} \mid \beta_j = \alpha_j, \, -k \le j \le
n \right\}
$ 
of these strings have Markov values $m(\alpha) > T$.   

By definition, after excluding from $E_2$ all irrational numbers 
whose continued fraction expansion contains a ``forbidden'' string, we obtain
a Cantor set~$K\subset E_2$ such that 
\begin{equation}
    \label{eq:Klow}
M \cap \bigl(\sqrt 5,T\bigr) 
\subset 2+ K + K.
\end{equation}
Recall that $\dim_H (K+K) \leq \dim_H K+ \overline{\dim_B} K$,
where $\overline{\dim_B} K$ denotes the upper box dimension. 
It is known that  $\overline{\dim_B} K = \dim_H K$ for these
types of sets (cf. Chapter 4 of Palis--Takens book~\cite{PT}) and hence $\dim_H K < 0.5$ implies that $t_1 \geq T$. 

\subsubsection{Approach to upper bound.}
\label{ss:aub}
Now let~$S$ be the maximal Markov value of strings which do not contain a forbidden
string as a substring and let $K\subseteq E_2$ be as above. It was shown in~\cite[proof of Lemma 3]{Mor18}
that 
\begin{equation}
    \label{eq:Ktop}
\min\{2\cdot\dim_H K,1\}\leq \dim_H( (\sqrt5,S) \cap M),
\end{equation}
Therefore we deduce that $\dim_H K \ge 0.5$ implies $t_1 \leq S$.  

In order to illustrate this methodology we shall show the double
inequality~\eqref{eq:10-12}. 
\begin{example}
To establish the lower bound in \eqref{eq:10-12}, we can use a
result by Hall~\cite{Hall} stating that 
if $\alpha\in\{1,2\}^{\mathbb{Z}}$ doesn't contain the string $121$, 
then $m(\alpha)<\sqrt{10}$. So we choose 
$$
K : = \left\{[0;\alpha_1,\alpha_2,\dots] \mid \alpha_j\in\{1,2\}, \mbox{ and }
(\alpha_j \alpha_{j+1} \alpha_{j+2})\neq (121) \, \mbox{ for all } 
\, j \ge 1 \right\},
$$ 
and apply the algorithm from~\S\ref{s.PV-new} to show that $\dim_H K < 0.45$. 
Then~\eqref{eq:Klow} gives
$$
\dim_H\left(\left\{\alpha \in\{1,2\}^{\mathbb{Z}}: \sqrt{5} < m(\alpha) \leq
\sqrt{10} \right\}\right) \le 2\dim_H K \le 0.9
$$
and the lower bound $t_1 \ge \sqrt{10}$ follows. 

To establish the upper bound in~\eqref{eq:10-12}, we recall a result by Perron~\cite{Per21} which
    states that $m(\alpha)\leq\sqrt{12}$ if and only if
$\alpha\in\{1,2\}^{\mathbb{Z}}$. Therefore we may choose the empty set of forbidden
strings and $K = E_2$. Combining~\eqref{eq:dimE2} with~\eqref{eq:Ktop} we
get 
$$
\dim_H( (\sqrt5,\sqrt{12}) \cap M) \ge \min\{2\cdot \dim_H E_2, 1\} =
\min(2\cdot 0.54318,1) = 1,
$$ 
and conclude that $t_1 \le \sqrt{12}$. 
\end{example}
This approach has been used by Bumby to obtain heuristic estimates and we will
review it in detail below in~\S\ref{ss:abumby}. Since we already know that $m(\alpha) \le \sqrt{12}$
if and only if $\alpha \in \{1,2\}^{\mathbb Z}$, until the end of \S\ref{s.Bumby-improved} 
we study only sequences of~$1$s and~$2$s. 

\subsection{Bumby's method for building the set of
forbidden strings}
\label{ss:abumby}
In this section we explain how to find a suitable set of forbidden strings which
can be employed to define a set~$K$ to use in~\eqref{eq:Klow} or~\eqref{eq:Ktop}. 

Recall the map $\lambda$ introduced in~\S\ref{s:intro}
$$
\lambda_0 \colon \mathbb N^{\mathbb Z} \to \mathbb R, \qquad \lambda(\alpha) =
[\alpha_0; \alpha_{1}, \alpha_{2},\dots] + [0;  \alpha_{-1}, \alpha_{-2},\dots].
$$
On the one hand, it is clear from definition that $m(\alpha) \ge
\lambda_0(\alpha)$. On the other hand, it is a well known fact that for any
Markov value $m \in M$, there exists a sequence $\alpha$ such that
$\lambda_0(\alpha) = m$; see for instance~\cite[Lemma 6, Chapter 1]{CF}. 
Therefore, one can attempt to construct a suitable set of forbidden strings by
studying the function~$\lambda_0$. This brings us to introducing a function~$J$,
which associates to a finite string a closed interval.

\begin{definition}
We denote by $J(\alpha_{-k,j})$ the interval given by the convex hull of the set
of values $\lambda_0(\beta)$ for strings $\beta \in\{1,2\}^{\mathbb{Z}}$ such that
$\beta_n=\alpha_n$ for all $-k\leq n\leq j$.
\end{definition}
In the sequel we will use the following shorthand notation for certain finite substrings of a
string $\alpha$: $\alpha_{-k,j} = \alpha_{-k}\dots
\alpha_{-1} \alpha_0 \alpha_1\dots \alpha_j$, where $j,k \ge 0$. 

Let us denote by $\overline{\alpha}$ the periodic
sequence obtained by infinite repetition of a given finite string $\alpha$.   
The following technical Lemma allows one to compute the
interval $J(\alpha_{-k,j})$ explicitly. 
\begin{lemma}
    \label{lem:intlim}
    For any sequence $\alpha \in \{1,2\}^{\mathbb{Z}}$ we have an upper bound
    $$
    \lambda_0(\alpha) \le 
    \begin{cases}
    [\alpha_0; \alpha_1 \dots \alpha_j \overline{12}] + 
    [0;\alpha_{-1} \dots \alpha_{-k} \overline{12}], &\mbox{ if } k \mbox{ and }
    j \mbox{ are even,} \\
    [\alpha_0; \alpha_1 \dots \alpha_j \overline{21}] + 
    [0;\alpha_{-1} \dots \alpha_{-k} \overline{12}], &\mbox{ if } k \mbox{ is
    even and } j \mbox{ is odd,} \\ 
    [\alpha_0; \alpha_1 \dots \alpha_j \overline{12}] + 
    [0;\alpha_{-1} \dots \alpha_{-k} \overline{21}], &\mbox{ if } k \mbox{ is
    odd and } j \mbox{ is even,} \\ 
    [\alpha_0; \alpha_1 \dots \alpha_j \overline{21}] + 
    [0;\alpha_{-1} \dots \alpha_{-k} \overline{21}], &\mbox{ if } k \mbox{ and }
    j \mbox{ are odd;} 
    \end{cases}
    $$
    and a lower bound 
    $$
    \lambda_0(\alpha) \ge 
    \begin{cases}
    [\alpha_0; \alpha_1 \dots \alpha_j \overline{21}] + 
    [0;\alpha_{-1} \dots \alpha_{-k} \overline{21}], &\mbox{ if } k \mbox{ and }
    j \mbox{ are even,} \\
    [\alpha_0; \alpha_1 \dots \alpha_j \overline{12}] + 
    [0;\alpha_{-1} \dots \alpha_{-k} \overline{21}], &\mbox{ if } k \mbox{ is
    even and } j \mbox{ is odd,} \\ 
    [\alpha_0; \alpha_1 \dots \alpha_j \overline{21}] + 
    [0;\alpha_{-1} \dots \alpha_{-k} \overline{12}], &\mbox{ if } k \mbox{ is
    odd and } j \mbox{ is even,} \\ 
    [\alpha_0; \alpha_1 \dots \alpha_j \overline{12}] + 
    [0;\alpha_{-1} \dots \alpha_{-k} \overline{12}], &\mbox{ if } k \mbox{ and }
    j \mbox{ are odd.} 
    \end{cases}
    $$
\end{lemma}
\begin{proof}
    This follows immediately from the fact that $\inf E_2 =
    [0;\overline{21}]=\frac12(\sqrt3-1)$ and $\sup E_2 =
    [0;\overline{12}]=\sqrt3-1$, where $\overline{21}$ and $\overline{12}$
    represent infinite sequences of alternating $1$s and $2$s.
\end{proof}
This Lemma also allows us to establish two more properties of the function~$J$ which will be
useful for our analysis. 
\begin{enumerate}
\item The function $J$ is invariant under reversal of the string (note that
    reversal keeps the $0$'th place unchanged).  
\item Extensions of a string $\alpha_{-k,j}$ correspond to
    subintervals of $J(\alpha_{-k,j})$. 
\begin{align*}
J(\alpha_{-k,j}) & = J(1 \alpha_{-k}\dots \alpha_j) \cup J(
2\alpha_{-k} \dots \alpha_j) \\ 
 \quad & = J (\alpha_{-k}\dots \alpha_j 1) \cup J(
\alpha_{-k}\dots \alpha_j 2); 
\end{align*}
Observe that unions need not to be disjoint. 
\end{enumerate}
The second property allows us to organise the intervals obtained from continuations
of a given string in a binary tree, so that the union of children is equal
to the parent. 

Now we can describe a recursive process for the construction of sets of forbidden strings. 
A basic idea is that we fix a threshold~$T$, close to
a conjectured lower bound on~$t_1$ and look for finite strings~$\alpha_{-k,j}$ such that the
corresponding intervals~$J(\alpha_{-k,j})$ lie to the \emph{right} of~$T$.
We call these finite strings ``forbidden'' and obtain the Cantor set~$K$ by
removing from~$E_2$ all numbers whose continued fraction expansion contains a forbidden
substring. If an interval $J(\alpha_{-k,j})$ lies to the \emph{left} of~$T$, 
we make a record of its right end point as a possible upper bound on $t_1$.
If $T \in J(\alpha_{-k,j})$ then we subdivide the interval into two by adding an extra
symbol to $\alpha_{-k,j}$ either in the beginning or at the end and study these
two new intervals at the next step of the recursive process. When we find a new
forbidden substring, we recompute the Hausdorff dimension of the updated set~$K$. We may
need to lower the original threshold, if $\dim_H K > 0.5$ \emph{and} the right end points of
the intervals which lie to the left of $T_1$ is too large; on the other hand, we
may need to increase the original threshold if $\dim_H K < 0.5$ and we
\emph{look to improve} an existing lower bound. We terminate the recursion when 
we find two sets which are suitable to confirm lower and upper bounds on~$t_1$
using~\eqref{eq:Klow} and~\eqref{eq:Ktop} respectively.


\begin{remark}
    A similar approach can be used to solve the equation $f(x) = t$ for other
    values $0 < t < 1$. Namely, Figure~\ref{fig:functionf} was obtained by
    computing the lower and upper bounds on~$x$ with the property that $\dim_H
    \left( (-\infty,x) \cap M \right) = t$ using~\eqref{eq:Klow}
    and~\eqref{eq:Ktop} for $t = 3 + t_1 \cdot \frac{k}{700}$, $k = 1, 2,
    \ldots, 700$. 
\end{remark}

\subsection{Rigorous verification of the Bumby's
estimate~\eqref{e.Bumby-heuristic}. }
In preparation for our estimate for $t_1$ we will  first rigorously confirm bounds 
close to the heuristic values of Bumby~\cite{Bumby}. This analysis will be an
integral part of our subsequent improved estimates.

Following Bumby, keeping in mind the heuristic estimate $3.33437 < t_1 <
3.33440$ which we would like to confirm rigorously, let us fix the threshold 
$$
T_1=3.334369.
$$ 
We are ready to start the recursive process of computing the set of
forbidden strings. 
We use the asterisk to mark the zeroth place in the string, i.e. $22^*1$
corresponds to $\alpha_{-1}=2$, $\alpha_0=2$, $\alpha_1=1$. 

\begin{figure}
\includegraphics[scale=0.85]{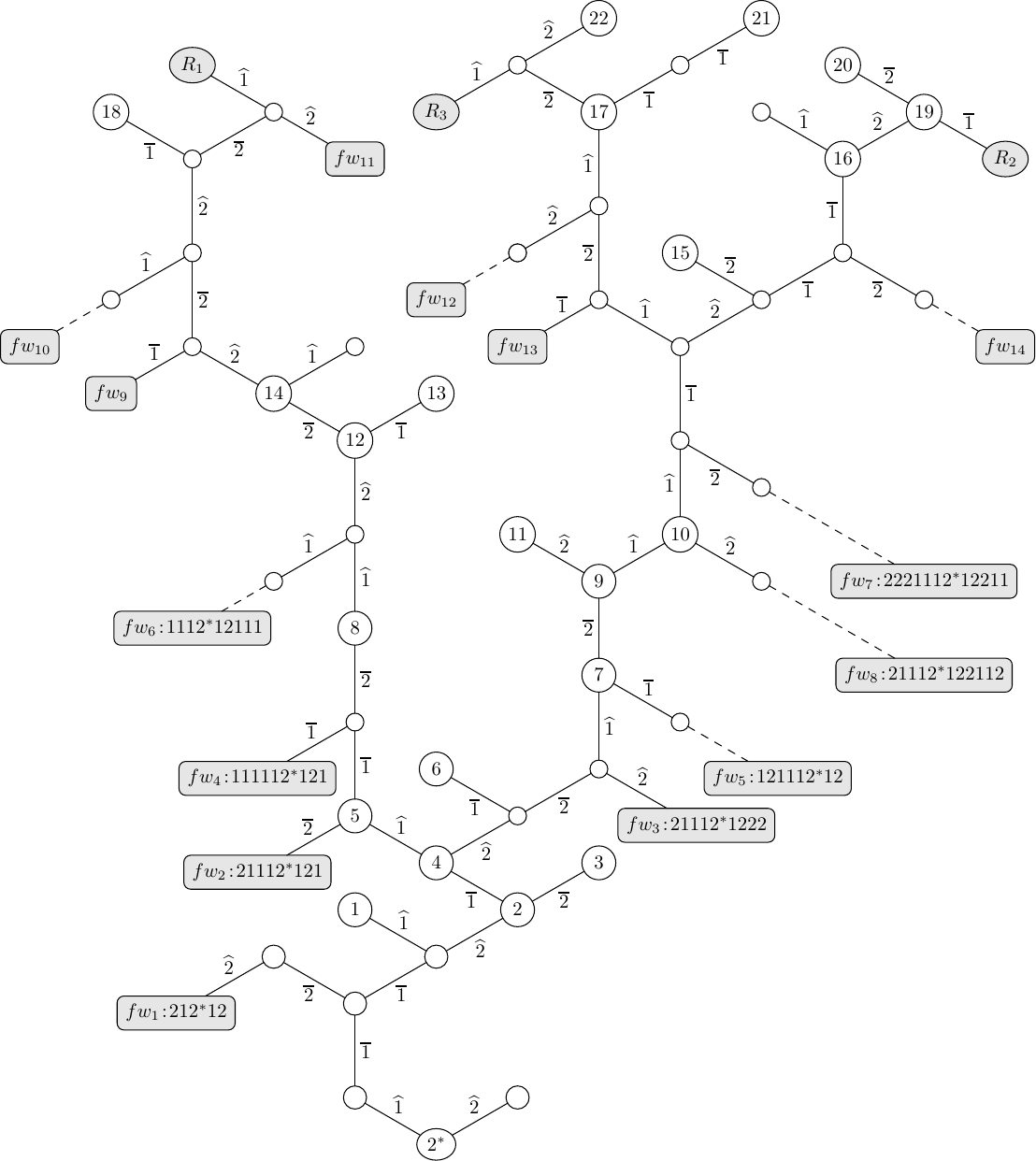}
\caption{The tree depicting the process of constructing forbidden strings. Each
vertex corresponds to a finite string of $1$s and $2$s, which can be recovered
by going down to the root writing the labels along the path marked by bars and going up to the
vertex writing the labels marked by hats. The
short forbidden words are written explicitly, the longer ones abbreviated as
$fw_{j}$, $j = 9,\dots,14$. A dashed edge without a label leading to a forbidden
word means that the forbidden word corresponds to a longer interval than the 
one which corresponds to the vertex it is connected to. The vertices $R_1$, $R_2$,
$R_3$ will be used as the roots for new trees which we build in Section~\ref{ss:t.A} in
order to prove our first main result, Theorem~\ref{t.A}. }
\label{fig:bumbytree}
\end{figure}
We begin with a simple observation that
a sequence $\alpha\in\{1,2\}^{\mathbb{Z}}$ with $3<\lambda_0(\alpha)<\sqrt{12}$
satisfies $\alpha_0=2$. (Since by Lemma~\ref{lem:intlim} we get
$J(1^*)=[\sqrt3,\sqrt{12}-1]$ and $J(2^*) = [ \sqrt3+1,\sqrt{12} ]$.) We may now
consider two continuations $2^*2$ and $2^*1$.  
Applying Lemma~\ref{lem:intlim} again we compute
$$ 
1+\sqrt3 = [2;2\overline{12}] + [0;\overline{21}] \leq
\lambda_0(2^*2)\leq[2;2\overline{21}]+[0;\overline{12}] = 2+\frac{2}{\sqrt3}.
$$
We conclude that $ J(22^*) = J(2^*2) \subset [1+\sqrt3, 3.155] < T_1$ and
proceed to  analyse continuations of $x_{-1}x_0^*x_1=12^*1$. 

The process of constructing the set of forbidden strings is depicted 
in Figure~\ref{fig:bumbytree}. We begin at the root marked $2^*$ and follow two
edges up adding letters marked by the bar symbol in the beginning, as a prefix, and
letters marked by the hat symbol at the end, as a suffix. Thus every vertex
corresponds to a finite string and we compute the interval corresponding to this
string in order to decide how to proceed further. Starting from the root, after two steps 
we arrive at a vertex which corresponds to $12^*1$. Taking two steps
further we obtain $212^*12$ which is our first excluded string because
by Lemma~\ref{lem:intlim},  the 
corresponding interval is $J(212^*12)\subset [3.4,\sqrt{12}] > T_1$.

The intervals corresponding to the vertices of the tree which are crucial to our
analysis are recorded in Table~\ref{tab:bumby}. 
For completeness in \S\S\ref{s:21212}--\ref{s:extra3} we give details of our
analysis. All intervals are computed using Lemma~\ref{lem:intlim}.
\begin{table}
    \begin{tabular}{|c|c|rl|rl|c|}
        \hline
       Set & Vertex & String & $ \kern-9pt \alpha_{-k,j}$ & Interval
        $J_0 \supset $&$  \kern-9pt J(\alpha_{-k,j})$ & Action \\
        \hline
     \multirow{33}{*}{$B_1$}
  &  \cellcolor{black!10}      $fw_1$   &  \cellcolor{black!10} $212^* $  &   \cellcolor{black!10}  $ \kern-13pt 12$                     & $[3.4, $  &   $ \kern-9pt \sqrt{12}]      $ & E \\ 
  &     $1$ & $112^* $  &   $ \kern-11pt 11$                     & $[3.1547, $  &   $ \kern-9pt  3.268]      $ & A \\
  &     $2$ & $112^* $  &   $ \kern-11pt 12$                     & $[3.28, $  &   $ \kern-9pt  3.3661]       $ & S \\
  &     $3$ & $2112^* $  &   $ \kern-11pt 12$                    & $[3.2802, $  &   $ \kern-9pt  3.3193]     $ & A \\
  &     $4$ & $1112^* $  &   $ \kern-11pt 12$                    & $[3.3149, $  &   $ \kern-9pt  3.3661]     $ & S \\
  &     $5$ & $11112^* $  &   $ \kern-11pt 121$                  & $[3.3324, $  &   $ \kern-9pt  3.3524]     $ & S \\
&  \cellcolor{black!10}     $fw_2$    &  \cellcolor{black!10}   $21112^* $  &    \cellcolor{black!10}   $ \kern-13pt 121$                  & $[3.35, $  &   $ \kern-9pt  3.3661]       $ & E \\
  &     $6$ & $11112^* $  &   $ \kern-11pt 122$                  & $[3.3149, $  &   $ \kern-9pt  3.3282]     $ & A \\
&   \cellcolor{black!10}       $fw_3$    &  \cellcolor{black!10}   $21112^* $  &    \cellcolor{black!10}   $ \kern-13pt 1222$                 & $[3.337, $  &   $ \kern-9pt  3.3419]      $ & E \\
  &     $7$ & $21112^* $  &   $ \kern-11pt 1221$                 & $[3.3329, $  &   $ \kern-9pt  3.3389]     $ & S \\
&    \cellcolor{black!10}    $fw_4$    & \cellcolor{black!10}  $111112^* $  & \cellcolor{black!10}    $ \kern-13pt 121$                 & $[3.3376, $  &   $ \kern-9pt  3.3524]     $ & E \\
  &     $8$ & $211112^* $  &   $ \kern-11pt 121$                 & $[3.3324, $  &   $ \kern-9pt  3.3456]     $ & S \\   
   &    \cellcolor{black!10}    $fw_5$  &   \cellcolor{black!10} $121112^* $  &    \cellcolor{black!10}  $ \kern-13pt 12$                  & $[3.3353, $  &   $ \kern-9pt  3.3661]     $ & E \\
  &     $9$ & $221112^* $  &   $ \kern-11pt 1221$                & $[3.3329, $  &   $ \kern-9pt  3.3356]     $ & S \\
  &    \cellcolor{black!10}    $fw_6$    &   \cellcolor{black!10}  $1112^* $  &    \cellcolor{black!10}   $ \kern-13pt 12111$                 & $[3.3351, $  &   $ \kern-9pt  3.3588]     $ & E \\
  &    $10$ & $221112^* $  &   $ \kern-11pt 12211$               & $[3.3341, $  &   $ \kern-9pt  3.3356]     $ & S \\
  &    $11$ & $221112^* $  &   $ \kern-11pt 12212$               & $[3.33294, $  &   $ \kern-9pt  3.33397]   $ & A \\
  &    $12$ & $211112^* $  &   $ \kern-11pt 12112$               & $[3.3324, $  &   $ \kern-9pt  3.3348]     $ & S \\
  &    $13$ & $1211112^* $  &   $ \kern-11pt 12112$              & $[3.3324, $  &   $ \kern-9pt  3.3339]     $ & A \\
 &     \cellcolor{black!10}   $fw_7$     &   \cellcolor{black!10}  $21112^* $  &    \cellcolor{black!10}   $ \kern-13pt 122112$               & $[3.3347, $  &   $ \kern-9pt  3.3389]     $ & E \\
  &    $14$ & $2211112^* $  &   $ \kern-11pt 12112$              & $[3.33369, $  &   $ \kern-9pt  3.33426]   $ & S \\
  &    \cellcolor{black!10}     $fw_8$     &    \cellcolor{black!10}  $2221112^* $  &     \cellcolor{black!10}   $ \kern-13pt 12211$              & $[3.33469, $  &   $ \kern-9pt  3.335541]  $ & E \\
  &     \cellcolor{black!10}    $fw_9$ &    \cellcolor{black!10}  $12211112^* $  &    \cellcolor{black!10}    $ \kern-13pt 121122$          & $[3.33448, $  &   $ \kern-9pt  3.33472]   $ & E \\      
  &   \cellcolor{black!10}      $fw_{10}$  &  \cellcolor{black!10}   $2211112^* $  &    \cellcolor{black!10}    $ \kern-13pt 1211221$       & $[3.33441, $  &   $ \kern-9pt  3.33472]   $ & E \\ 
  &   \cellcolor{black!10}      $fw_{13}$ &    \cellcolor{black!10}  $11221112^* $  &     \cellcolor{black!10}   $ \kern-13pt 1221111$      & $[3.33447, $  &   $ \kern-9pt  3.334684]  $ & E \\ 
  &    $15$  & $21221112^* $  &   $ \kern-11pt 1221112$          & $[3.33414, $  &   $ \kern-9pt  3.33424]   $ & A \\
  &    $16$  & $111221112^* $  &   $ \kern-11pt 1221112$         & $[3.3343, $  &   $ \kern-9pt  3.334393]   $ & S \\
 &      \cellcolor{black!10}   $fw_{14}$ &    \cellcolor{black!10}  $211221112^* $  &     \cellcolor{black!10}   $ \kern-13pt 12211$       & $[3.3343894, $  &   $ \kern-9pt  3.3352]  $ & E \\
  &    $17$  & $21221112^* $  &   $ \kern-11pt 12211111$         & $[3.3343, $  &   $ \kern-9pt  3.334402]   $ & S \\ 
&    \cellcolor{black!10}  $fw_{12}$ & \cellcolor{black!10}  $21112^* $  &   \cellcolor{black!10}  $ \kern-13pt 12211112$        & $[3.3344009, $  &   $ \kern-9pt  3.3384]  $ & E \\
  &    $18$  & $122211112^* $  &   $ \kern-11pt 1211222$         & $[3.334335, $  &   $ \kern-9pt  3.334375] $ & S \\
   &     \cellcolor{black!10}   $R_1$ &   \cellcolor{black!10}  $222211112^* $  &     \cellcolor{black!10}  $ \kern-13pt 12112221$        & $[3.334371, $  &   $ \kern-9pt  3.3343876]$ & E \\
 &     \cellcolor{black!10}   $fw_{11}$ &    \cellcolor{black!10} $222211112^* $  &     \cellcolor{black!10}  $ \kern-13pt 12112222$    & $[3.3343899, $  &   $ \kern-9pt  3.33441] $ & E \\
    \hline
 \multirow{6}{*}{$B_2$} &
     $19$    &       $111221112^* $  &   $ \kern-11pt 12211121$  & $[3.334304, $  &   $ \kern-9pt  3.334363]$ &   A   \\
  &  $20$    &      $2111221112^* $  &   $ \kern-11pt 12211122$ & $[3.33434, $  &   $ \kern-9pt  3.334362]$  &    A  \\
   &  \cellcolor{black!10}  $R_2$   &   \cellcolor{black!10}     $1111221112^* $  &   \cellcolor{black!10}  $ \kern-13pt 12211122$ & $[3.3343695, $  &   $ \kern-9pt  3.334393]$ &   E  \\
  &  $21$    &      $1121221112^* $  &   $ \kern-11pt 12211111$ & $[3.334325, $  &   $ \kern-9pt  3.334373]$ &    S  \\
  &  $22$    &      $221221112^* $  &   $ \kern-11pt 122111112$ & $[3.33435, $  &   $ \kern-9pt  3.3343683]$ &    A  \\
   &   \cellcolor{black!10} $R_3$   &    \cellcolor{black!10}    $221221112^* $  &   \cellcolor{black!10}  $ \kern-13pt 122111111$ & $[3.334378, $  &   $ \kern-9pt  3.33441]$  &    E \\
  \hline
    \end{tabular}
    \medskip
    \caption{Strings and intervals crucial to our analysis. The action column indicates how to proceed
    with the tree construction further: 
    E --- Exclude the string (the case $J_0 > T_1=3.334369$), the corresponding
    vertex is a leaf; A --- Abandon the branch (the case $J_0 < T_1$) the corresponding
    vertex is a leaf; S --- Subdivide the interval into two parts (the
    case  $T_1 \in J_0$), the vertex is a branching point.
    An estimate for $J(211221112^*12211)$ used the fact that $21112121$ is
    excluded, in addition to Lemma~\ref{lem:intlim}. There is only one branch
    out of the $8$th vertex because $2^*1212$ is already excluded, so the only
    possible suffix is $\widehat 1$. 
    }
    \label{tab:bumby}
\end{table}

\subsubsection{Exclusion of $21212$} Note that 
\label{s:21212}
$$J(112^*11)\subset [3.1547, 3.268], \quad J(112^*12)\subset [3.28, 3.3661], \quad J(212^*12)\subset [3.4,\sqrt{12}].$$
Since $3.268 < T_1 < 3.4$, we exclude $212^*12$ and we analyse the continuation
of $212^*11$.  For this purpose, we decompose $J(112^*12)$ into $J(1112^*12)$
and $J(2112^*12)$. 

Note that 
$$J(2112^*12)\subset [3.2802, 3.3193] \quad \textrm{and} \quad J(1112^*12)\subset [3.3149, 3.3661].$$
Thus, it suffices to study the continuations $x=1112^*12$ (as $T_1 > 3.3193$). 

\subsubsection{Exclusion of $21112121$ and $211121222$} Let's consider the
decompositions of the intervals $J(1112^*121)$ and $J(1112^*122)$ where the
string 21212 doesn't appear. For the first interval, it amounts to studying
$J(11112^*121)$, $J(21112^*121)$. Note that 
$$J(11112^*121)\subset [3.3324, 3.3524] \quad \textrm{and} \quad J(21112^*121)\subset [3.35, 3.3661].$$
For the second interval, we have 
$$J(21112^*1222)\subset [3.337, 3.3419], \quad \quad J(11112^*1222)\subset [3.3189, 3.3282]$$
$$J(21112^*1221)\subset [3.3329, 3.3389], \quad \quad J(11112^*1221)\subset [3.3149, 3.3252].$$
Since $3.3282<T_1<3.337$, we exclude $21112121$ and $211121222$, and we shall
consider the decompositions of the intervals $J(11112^*121)$ and
$J(21112^*1221)$. 

\subsubsection{Exclusion of
 $111112121$ and $12111212$}
 Note that 
$$J(111112^*121)\subset [3.3376, 3.3524], \quad J(211112^*121)\subset [3.3324, 3.3456],$$
and 
$$J(121112^*1221)\subset J(121112^*12)\subset [3.3353, 3.3661], \quad J(221112^*1221)\subset [3.3329, 3.3356].$$ 
Because $T_1<3.3353$, we exclude $111112121$ and $12111212$, and we consider the
decompositions of $J(211112^*121)$ and $J(221112^*1221)$. Actually, given that
the string $21212$ is already excluded, our task is to study the decompositions
of $J(211112^*1211)$ and $J(221112^*1221)$. 

\subsubsection{Exclusion of $1111212111$} Observe that 
$$J(211112^*12111)\subset J(1112^*12111)\subset [3.3351, 3.3588], \quad
J(211112^*12112)\subset [3.3324, 3.3348],$$ 
and 
$$J(221112^*12212)\subset [3.33294, 3.33397], \quad J(221112^*12211)\subset [3.3341, 3.3356].$$
Because $3.33397<T_1<3.3351$, we exclude $111212111$ and we decompose
$J(211112^*12112)$ and $J(221112^*12211)$. 

\subsubsection{Exclusion of $21112122112$} Note that $J(211112^*12112)$ decomposes into $J(2211112^*12112)$ and 
$$J(1211112^*12112)\subset [3.3324, 3.3339].$$
Similarly, $J(221112^*12211)$ decomposes into $J(221112^*122111)$ and 
$$J(221112^*122112)\subset J(21112^*122112)\subset [3.3347, 3.3389].$$
Since $3.3339 < T_1 < 3.3347$, we exclude $21112122112$, and we decompose
$J(2211112^*12112)$ and $J(221112^*122111)$.  

\subsubsection{Exclusion of $222111212211$} Note that $J(2211112^*12112)$ breaks
into $J(2211112^*121122)$ and 
$$J(2211112^*12112)\subset [3.33369, 3.33426]$$
Analogously, $J(221112^*122111)$ decomposes into $J(1221112^*122111)$ and 
$$J(2221112^*122111)\subset J(2221112^*12211)\subset [3.33469, 3.335541].$$
Given that $3.33426<T_1<3.33469$, we exclude $222111212211$, and we proceed to
analyse the decompositions of $J(2211112^*121122)$ and $J(1221112^*122111)$. 

\subsubsection{Exclusion of $12211112121122$} We break the previous intervals
into $J(12211112^*121122)$, $J(22211112^*121122)$ and $J(1221112^*1221111)$,
$J(1221112^*1221112)$, and we observe that 
$$J(12211112^*121122)\subset [3.33448,3.33472].$$
Because $T_1<3.33448$, we exclude $12211112121122$, and we decompose
$J(22211112^*121122)$,  $J(1221112^*1221111)$ and $J(1221112^*1221112)$. 

\subsubsection{Exclusion of two extra strings} We decompose the interval
$J(22211112^*121122)$ into $J(22211112^*1211222)$ and 
$$J(22211112^*1211221)\subset J(2211112^*1211221)\subset [3.33441, 3.33472].$$
Similarly, $J(1221112^*1221111)$ subdivides into $J(21221112^*1221111)$ and 
$$J(11221112^*1221111)\subset [3.33447, 3.334684].$$
Analogously, $J(1221112^*1221112)$ breaks into $J(11221112^*1221112)$ and 
$$J(21221112^*1221112)\subset [3.33414, 3.33424].$$
Because $3.33424 < T_1 < 3.33441$, we exclude $112211121221111$ and
$22111121211221$, and we analyse $J(22211112^*1211222)$,
$J(21221112^*1221111)$ and $J(11221112^*1221112)$. 

\subsubsection{Exclusion of three extra strings} 
\label{s:extra3}
We decompose the interval $J(11221112^*1221112)$ into 
$$J(111221112^*1221112)\subset [3.3343, 3.334393],$$ 
$$J(211221112^*1221112)\subset J(211221112^*12211)\subset [3.3343894, 3.3352].$$
(Here, we estimated the second interval using the fact that $21112121$ is excluded.)

Similarly, we break $J(21221112^*1221111)$ into 
$$J(21221112^*12211111)\subset [3.3343, 3.334402],$$ 
$$J(21221112^*12211112)\subset J(21112^*12211112)\subset [3.3344009, 3.3384].$$

Finally, we observe that $J(22211112^*1211222)$ subdivides into $J(122211112^*1211222)$, $J(222211112^*12112221)$, $J(222211112^*12112222)$ with 
$$J(122211112^*1211222)\subset [3.334335, 3.334375],$$
$$J(222211112^*12112221)\subset [3.334371, 3.3343876],$$
and 
$$J(222211112^*12112222)\subset [3.3343899, 3.33441].$$
Since $T_1 < 3.3343894$, we exclude $21122111212211$, $2111212211112$ and $22221111212112222$. 

\subsubsection{Upper bound on $t_1$ revisited} 
\label{sss:t1up}
Using numerical data from the top part of Table~\ref{tab:bumby} we are now in a position to
get an upper bound on $t_1$ in line with the heuristic estimate of $3.33440$
suggested by Bumby. 
Denote by $B_1$ the Cantor set of numbers whose continued fraction expansions in
$\{1,2\}^{\mathbb{N}}$ which do not contain the following fourteen strings (nor
their transposes) taken from the lines of Table~\ref{tab:bumby} marked for
exclusion:
\begin{itemize}
\item $21212$, $21112121$, $211121222$, $111112121$, $12111212$, $111212111$, $21112122112$, 
\item $222111212211$, $12211112121122$, $112211121221111$, $22111121211221$, 
\item $21122111212211$, $2111212211112$ and $22221111212112222$.
\end{itemize} 
The  algorithm described in Section~\ref{s.PV-new} provides us lower and upper
bounds (see Subsection~\ref{B1set} for numerical data and implementation notes)
\begin{equation}
    \label{eq:B1}
0.50001 < \dim B_1 < 0.50005
\end{equation}
which confirms Bumby's heuristics in~\cite{Bumby}. Consequently,
applying~\eqref{eq:Ktop} we get that
$t_1$ is bounded from above by the maximum of the right endpoints of the
non-excluded intervals that appeared in the process of construction of the set
$B_1$ (both abandoned and marked for subdivision). This turns out to be the right end point of the
interval corresponding to the vertex~$17$. In particular, we have that
$$
t_1\leq S_1=3.334402.
$$ 


\subsubsection{Lower bound on~$t_1$ revisited} 
\label{sss:t1low}
With a little more work we can get a lower bound on~$t_1$ which supports Bumby's
lower bound on~$t_1$ of~$3.33437$. 
Continuing to follow Bumby~\cite{Bumby}, let us further analyse the intervals 
$$
J(21221112^*12211111)\ \mbox{ and }\ J(111221112^*1221112),
$$
which correspond to the $16$th and $17$th vertices of the tree and marked for
subdivision in Table~\ref{tab:bumby}. Our computations are presented in
the bottom part of Table~\ref{tab:bumby}.  
In particular, we see that one can also exclude
$111122111212211122$ and $221221112122111111$ 
in order to obtain a smaller Cantor set
$B_2\subsetneq B_1$. Applying the algorithm for computing Hausdorff dimension
described in \S\ref{s.PV-new} we obtain estimates on dimension (see
\S\ref{B2set} for implementation notes): 
\begin{equation}
    \label{eq:B2}
0.499975 < \dim_H B_2 < 0.49999
\end{equation}
This is quite close to Bumby's heuristic claim that $\dim_H(B_2) < 0.499974$ and
we conclude that 
$$t_1\geq T_1=3.334369.$$ 

Summing up, we have rigorously confirmed that the heuristic argument by Bumby in
favour of looking for~$t_1$ inside the interval~$(3.33437,3.33440)$
was correct.  

After the above review (and slight improvement) of Bumby's work~\cite{Bumby}, 
we now turn to the proof of our main result Theorem~\ref{t.A}.

\subsection{Proof of Theorem \ref{t.A}} 
\label{ss:t.A}
Recall that our goal is to show that the first transition point $t_1 =
3.334384\dots$. It is sufficient to prove that 
$$
3.3343840 < t_1 < 3.33438495.
$$
For this purpose, let us fix the thresholds 
\begin{equation}
    \label{eq:T2S2}
T_2 := 3.334384009 \quad \mbox{ and } \quad S_2:= 3.3343849341.
\end{equation}
Our goal now is to modify the Cantor sets~$B_1$ and~$B_2$ defined above to
obtain two Cantor sets~$X$ and~$Y$ such that the intervals corresponding to
forbidden strings used to define~$X$ lie to the right of~$T_2$ and~$S_2$ is the
right end point of the intervals corresponding to the non-excluded strings which
appear in the construction of $Y$. Furthermore, we also require that the double
inequality $\dim_H X < 0.5 < \dim_H Y$ holds. 

In this direction we consider the intervals listed in Table~\ref{tab:bumby}
and choose the smallest (by inclusion) intervals which contain both~$T_2$
and~$S_2$ in order to subdivide them further and to identify forbidden strings
exclusion of which will result in Cantor sets with dimension closer to~$0.5$ than $\dim_H B_1$ and
$\dim_H B_2$. These turn out to be the intervals
corresponding to the vertices $R_1$, $R_2$, and $R_3$. We list the corresponding
strings:
$R_1 = 222211112^*12112221$, $R_2=1111221112^*12211122$, and
$R_3=221221112^*122111111$.  
We subdivide each of the intervals $J(R_1)$, $J(R_2)$, and $J(R_3)$ following
the same process as before, with a separate decision tree in each case. 

\subsubsection{Refinement of $J(222211112^*12112221)$} 
\label{sss:r1}
The tree depicting continuation of the string $R_1=222211112^*12112221$ is shown in
Figure~\ref{fig:r1tree} and the numerical data for the key intervals is
given in Table~\ref{tab:r1tree} (obtained using Lemma~\ref{lem:intlim}).
Three extra strings are marked for exclusion,
namely $1R_112$, $121R_11$, and $21R_1111$. 

\begin{table}
    \begin{tabular}{|rl|rl|c|}
        \hline
       String &$  \kern-9pt  \alpha_{-k,j}$ & Interval
        $J_0 \supset $ & $  \kern-9pt J(\alpha_{-k,j})$ & Action \\
        \hline
   $ $ & $ \kern-9pt R_12      $  &      $[3.334371, $  &   $ \kern-9pt  3.334381]      $ &A \\
  $2 $ & $ \kern-9pt R_11     $  &       $[3.334376, $  &   $ \kern-9pt  3.33438141]     $ &A \\
  $1 $ & $ \kern-9pt R_112    $  &       $[3.334384049, $  &   $ \kern-9pt  3.3343876]   $ &E \\
 $11 $ & $ \kern-9pt R_111   $  &        $[3.334381, $  &   $ \kern-9pt  3.3343837]      $  &A \\
$121 $ & $ \kern-9pt R_11   $  &         $[3.334384009, $  &   $ \kern-9pt  3.3343876]   $   &E \\
$221 $ & $ \kern-9pt R_1112 $  &         $[3.33438368, $  &   $ \kern-9pt  3.33438401]   $   &S \\
 $21 $ & $ \kern-9pt R_1111  $  &        $[3.3343844, $  &   $ \kern-9pt  3.33438551]   $   &E \\
    \hline
    \end{tabular}
    \medskip
    \caption{Numerical data for the subdivision of the interval $J(R_1)=J(222211112^*12112221)$.
    The corresponding tree is shown in Figure~\ref{fig:r1tree}.
    Strings corresponding to the intervals to the right of $T_2 = 3.334384009$ marked
    for exclusion. }
\label{tab:r1tree}
\end{table}

\begin{figure}
\includegraphics{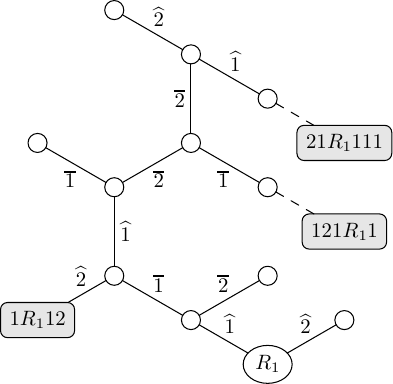}
\caption{Continuation of the string $R_1=222211112^*12112221$.}
\label{fig:r1tree} 
\end{figure}

\subsubsection{Refinement of $J(1111221112^*12211122)$}
\label{sss:r2}
The tree depicting continuation of the string $R_2=1111221112^*12211122$ is shown in
Figure~\ref{fig:r2tree} and the numerical data for the key intervals is
given in Table~\ref{tab:r2tree}. Based on the threshold $T_2$ 
we exclude $2R_21=21111221112^*122111221$ and $12R_222=121111221112^*1221112222$. 
\begin{table}
    \begin{tabular}{|rl|rl|c|}
        \hline
        String & $ \kern-9pt \alpha_{-k,j}$ & Interval
        $J_0 \supset$ & $  \kern-9pt J(\alpha_{-k,j})$ & Action \\
        \hline
    $1 $ & $ \kern-9pt R_2$    &    $[3.334369, $  &   $ \kern-9pt  3.33438361]    $    & A  \\
    $2 $ & $ \kern-9pt R_21$   &    $[3.33438668, $  &   $ \kern-9pt  3.33439261]  $    & E  \\
    $2 $ & $ \kern-9pt R_221$  &    $[3.3343815, $  &   $ \kern-9pt  3.3343847]    $    & S  \\
    $12 $ & $ \kern-9pt R_222$ &    $[3.33438429, $  &   $ \kern-9pt  3.3343856]   $    & E  \\
        \hline
    \end{tabular}
    \medskip
    \caption{Numerical data for the subdivision of the interval $J(R_2) = J(1111221112^*12211122)$. The
    subdivision tree is shown in Figure~\ref{fig:r2tree}. Strings
    corresponding to the intervals to the right of $T_2 = 3.334384009$ marked
    for exclusion. }
\label{tab:r2tree}
\end{table}

\begin{figure}
\includegraphics{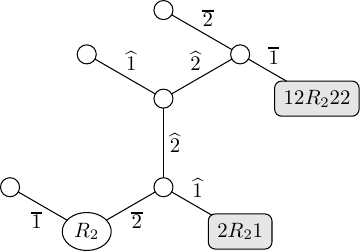}
\caption{Continuation of the string $R_2=1111221112^*12211122$.}
\label{fig:r2tree} 
\end{figure}

\subsubsection{Refinement of $J(221221112^*122111111)$} 
\label{sss:r3}
The tree depicting continuation of the string $R_3=221221112^*122111111$ is shown in
Figure~\ref{fig:r3tree} and the numerical data for the key intervals is
shown in Table~\ref{tab:r3tree}. Based on the threshold $T_2$ 
five additional strings are marked for
exclusion: $R_32$, $21R_3$, $1R_311$, $111R_3122$, and $22R_3112$

\begin{table}
    \begin{tabular}{|rl|rl|c|}
        \hline
        String & $\kern-9pt \alpha_{-k,j}$ & Interval
        $J_0 \supset$ & $ \kern-9pt   J(\alpha_{-k,j})$ & Action \\
        \hline
                 & \kern-9pt $R_32$      & $[3.33439,       $  &   $\kern-9pt  3.334402]$      &  E \\ 
      $21 $  &   $ \kern-9pt R_3$       & $[3.3343856,     $  &   $\kern-9pt  3.334402]$      &  E \\ 
       $1 $  &   $ \kern-9pt R_311$     & $[3.3343866,     $  &   $\kern-9pt  3.3343922]$    &  E \\   
     $211 $  &   $ \kern-9pt R_312$     & $[3.334383,      $  &   $\kern-9pt  3.33438429]$      &  S \\
     $111 $  &   $ \kern-9pt R_3121$    & $[3.33438375,    $  &   $\kern-9pt  3.334384636]$   &  S \\
     $111 $  &   $ \kern-9pt R_3122$    & $[3.3343846357,  $  &   $\kern-9pt  3.3343853]$   &  E \\
      $12 $  &   $ \kern-9pt R_31$      & $[3.334378,      $  &   $\kern-9pt  3.33438459]$     &  S \\
      $22 $  &   $ \kern-9pt R_312$     & $[3.334379,      $  &   $\kern-9pt  3.3343806]$      &  S \\
      $22 $  &   $ \kern-9pt R_3111$    & $[3.3343829,     $  &   $\kern-9pt  3.33438403]$    &  S \\
      $22 $  &   $ \kern-9pt R_3112$    & $[3.3343847,     $  &   $\kern-9pt  3.3343855]$     &  E \\
        \hline
    \end{tabular}
    \medskip
    \caption{Numerical data for the subdivision of the interval $J(R_3)=J(221221112^*122111111)$.
    The corresponding tree is shown in Figure~\ref{fig:r3tree}. Strings
    corresponding to the intervals to the right of $T_2 = 3.334384009$ marked
    for exclusion.}
\label{tab:r3tree}
\end{table}

\begin{figure}
\includegraphics{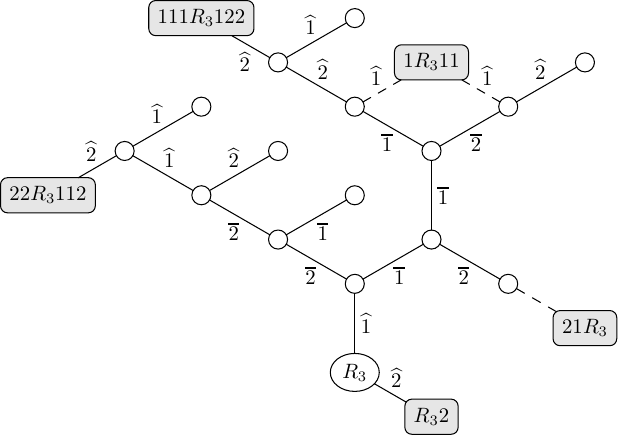}
\caption{Continuation of the string $R_3=1111221112^*12211122$.}
\label{fig:r3tree} 
\end{figure}

\subsubsection{Lower bound on~$t_1$} 
\label{sss:t1low2} In order to confirm the lower bound stated in
Theorem~\ref{t.A}, we collect together numerical data from calculations in~\S\S\ref{sss:r1}--\ref{sss:r3}. 
Consider the Cantor set~$X\subset E_2$ of numbers which continued fraction expansions  
do not contain neither any of the following $24$ strings nor their transposes: 
\begin{itemize}
\item The 14 words proposed by Bumby, listed in~\S\ref{sss:t1up}, cf.
        Table~\ref{tab:bumby}:
        $21212$, $21112121$, $211121222$, $111112121$,
        $12111212$, $111212111$, $21112122112$, $222111212211$,
        $12211112121122$, $112211121221111$, $22111121211221$, 
        $21122111212211$, \\$2111212211112$, $22221111212112222$ ;
    \item The 3 words obtained in~\S\ref{sss:r1} as continuations of $R_1$:
        $12222111121211222112$, $121222211112121122211$, and
        $2122221111212112221111$; 
\item The 2 words obtained in~\S\ref{sss:r2} as continuations of $R_2$:
$21111221112122111221$, $1211112211121221112222$; 
\item The 5 words obtained in~\S\ref{sss:r3} as continuations of $R_3$:
$2212211121221111112$, $21221221112122111111$, $122122111212211111111$, \\$111221221112122111111122$, $22221221112122111111112$
\end{itemize} 

In Subsection~\ref{setX}  the algorithm described in~\cite{PV20} will be
implemented to rigorously establish the bound 
    $\dim_H X < 0.5 - 10^{-8}$.
Summing up, we get the desired lower bound 
\begin{equation}
    \label{eq:dimX}
3.334384009 = T_2 \le t_1. 
\end{equation}

\subsubsection{Upper bound on $t_1$}
\label{sss:t1up2}
We are now ready to justify the upper bound $t_1 \le S_2 = 3.3343849341$ proposed in Theorem~\ref{t.A}. Following the method explained in
\S\ref{ss:aub}, we need to modify the set $X$, increasing its dimension, so that 
the right end point of a non-excluded interval is no smaller than~$S_2$. 
Therefore from the intervals marked for exclusion in Tables~\ref{tab:r1tree}, \ref{tab:r2tree}, \ref{tab:r3tree} we choose the
shortest ones which contain~$S_2$. These turn out to be the intervals
corresponding to the strings 
\begin{align*}
 1R_112 & = 1222211112^*1211222112,       &&  
 221R_1111  = 221222211112^*12112221111,  \\
 12R_222  & =  121111221112^*1221112222, &&  
 121R_111  =  121222211112^*1211222111.  
\end{align*}
We proceed to study their subintervals applying Lemma~\ref{lem:intlim} while excluding all intervals to the
right of the value
$$
 T_3:=3.3343846357 \in (T_2, S_2).
$$
The analysis of the first interval $J(1R_112)$ is relatively simple. 
More precisely, it breaks into 
\begin{align}
J( 2 \,& 1R_112 )\subset [3.334386, 3.3343876] > T_3, \mbox{ and }
\label{eq:yw2} \\
J(11 \,& 1R_112 )\subset [3.33438473, 3.3343858] > T_3 \label{eq:yw3} \\
J(21 \,& 1R_112 )\subset [3.334384049, 3.33438484] < T_3. \notag
\end{align}
Following the approach explained in the beginning of \S\ref{ss:abumby}, 
we exclude the string $21R_1112$ corresponding to the first of them, since every element is
larger than~$T_3$.
 
Similarly, $J(12R_222)$ breaks into
\begin{align}
    J(1\, & 12R_222 )\subset [3.33438429, 3.3343849341] < T_3, \ \mbox{ and }
    \label{eq:yw6} \\ 
    J(2\, & 12R_222 )\subset [3.3343851,  3.3343856] > T_3. \label{eq:yw1}
\end{align}
We exclude the string $212R_222$ corresponding to the second interval, since it
lies to the right of $T_3=3.3343846357$.

The third interval $J(221R_1111)$ subdivides into 
\begin{align}
    J(221R_1111\,& 1)\subset [3.33438448, 3.334384762] \ni T_3, \  \mbox{ and }
    \notag\\
J(221R_1111\,& 2)\subset J(21 R_11112) \subset [3.33438488, 3.33438551] > T_3.
\label{eq:yw4}
\end{align}

Finally, $J(121R_111)$ decomposes as 
\begin{align}
    J( 121R_111 &\, 1)\subset [3.3343848, 3.3343856] > T_3 \label{eq:yw5} \\
J( 121R_111 &\, 2)\subset [3.334384009, 3.33438445] < T_3. \notag
\end{align}

We may now define $Y$ to be the Cantor set of continued fraction expansions in
$\{1,2\}^{\mathbb{N}}$ which do not contain the following $25$ strings (nor their transposes): 
\begin{itemize}
    \item The $14$ words composed by Bumby, listed in \S\ref{ss:abumby} cf.
        Table~\ref{tab:bumby}:  $21212$, $21112121$, $211121222$, $111112121$, $12111212$, $111212111$,
    $21112122112$, $222111212211$, $12211112121122$, $112211121221111$,
    $22111121211221$, $21122111212211$,  \\ $2111212211112$, $22221111212112222$;
\item The $2$ words constructed as continuations of $R_2$:
    $2R_21 =21111221112122111221$ (cf. Table~\ref{tab:r2tree}) and $212R_222 =
    21211112211121221112222$ (see~\eqref{eq:yw1} above), 
\item  The $5$ words obtained in \S\ref{sss:r3} as continuations of $R_3$ : 
    $2212211121221111112$, $21221221112122111111$, $122122111212211111111$,
    $111221221112122111111122$, $22221221112122111111112$ (cf.
    Table~\ref{tab:r3tree} for numerical data on the intervals);
\item  The $4$ words composed as continuations of $R_1$: \\
    $21R_112 = 212222111121211222112$~\eqref{eq:yw2}, $111R_112 =
    1112222111121211222112$ (by~\eqref{eq:yw3}),
    $121R_1111=12122221111212112221111$ (by~\eqref{eq:yw5}), and $21R_11112=21222211112121122211112$
    (by~\eqref{eq:yw4}). 
\end{itemize} 
then the fact that a rigorous
estimate in Subsection~\ref{setY} gives that
$\dim_H Y > 0.5+10^{-8}$ allows to conclude that 
\begin{equation}
\label{eq:dimY}
t_1\le S_2 = 3.3343849341,
\end{equation} 
which is the right end point of the non-excluded interval corresponding to
$112R_222 = 1121111221112^*1221112222$~(see~\eqref{eq:yw6}). 

The inequalities~\eqref{eq:dimX} and~\eqref{eq:dimY} complete the proof of
Theorem~\ref{t.A}.

\section{  Bounds on $\dim_H (M\setminus L)$}
\label{s.M-L-new}
In this section we establish our second main result 
\begin{theorem}\label{t.B} 
The Hausdorff dimension of the difference of Markov and Lagrange spectra satisfies
$$
0.537152 < \dim_H(M\setminus L) < 0.796445
$$
\end{theorem}

\subsection{Lower bounds}\label{ss.lower-bound-M-L}
It was shown in~\cite[\S2.5.4]{book2021} that $\dim_H((M\setminus L) \cap (3.7,3.71))$
coincides with the dimension of a certain Gauss--Cantor set~$\Omega$ with
complicated structure. 
Implementing the algorithm described in~\S\ref{s.PV-new}, we obtain an estimate
$\dim_H\Omega=0.537152\dots$ (see \S~\ref{Omegaset} for computation notes).  
A combination of these two results gives the best lower bound on $M\setminus L$ so
far:
$$
\dim_H ((M \setminus L) \cap (\sqrt{13}, 3.84))\geq \dim_H\Omega > 0.537152.
$$

\subsection{Upper bounds}
Recall that Freiman and Schecker independently showed circa 1973 that see, e.g.~\cite{CF}
$$
[\sqrt{21},+\infty)=L\cap[\sqrt{21},+\infty)=M\cap[\sqrt{21},+\infty).
$$
More recently, it was shown in~\cite{MatheusMoreira}
and~\cite{PV20} that $\dim_H((M\setminus L)\cap (\sqrt{5}, \sqrt{13})) <
0.73$. Hence in order to establish an upper bound of $0.796455$, 
it suffices to study $M\setminus L$ within the interval $(\sqrt{13},\sqrt{21})$.
 
Let us now set out the strategy which we will employ for the rest of this section.
We consider a partition of $(\sqrt{13},\sqrt{21})$ into several small
intervals $(x,y)$ and study the intersections $(M\setminus L)\cap
(x, y)$. To find an upper bound for the Hausdorff dimension of $(M\setminus L)\cap
(x, y)$, we continue to develop the ideas from~\cite{MatheusMoreira}. 

Very roughly speaking, we select two transitive subshifts of finite type $B\subset C\subset(\mathbb{N}^*)^{\mathbb{Z}}$ with $m(\alpha)<x$ for all $\alpha\in B$ and any $\beta\in(\mathbb{N}^*)^{\mathbb{Z}}$ with $m(\beta)<y$ belongs to $C$. We require that $B$ and $C$ are symmetric in the sense that $K(B)=K^-(B)$ and $K(C) = K^-(C)$, where $K(A):=\{[0;\alpha_1, \alpha_2,\dots] \mid (\alpha_n)_{n\in\mathbb{Z}}\in A\}$
and $K^-(A):=\{[0;\alpha_{-1},\alpha_{-2},\dots]\mid (\alpha_n)_{n\in\mathbb{Z}}\in
A\}$ stand for the unstable and stable Gauss--Cantor sets associated to a given subshift of finite type $A\subset (\mathbb{N}^*)^{\mathbb{Z}}$. 

At this stage, we want to employ a shadowing lemma type argument to get that, up to transposition, any sequence $\zeta$ with $m(\zeta)\in(M\setminus L)\cap (x,y)$ has the property that if $N$ is large, $n\geq N$, $\tau$ is a finite string and $\alpha$, $\alpha'$ are infinite strings with distinct first elements such that the two sequences $\dots \zeta_{-N}\dots \zeta_n\tau\alpha$ and $\dots \zeta_{-N}\dots \zeta_n\tau\alpha'$ have Markov values in $(M\setminus L)\cap (x,y)$, then the unstable Cantor set $K(B)=\{[0;\theta_1,\theta_2,\dots]:(\theta_n)_{n\in\mathbb{Z}}\in B\}$ of $B$ doesn't intersect the interval $[[0;\alpha], [0;\alpha']]$. In particular, by taking $\tau=\varnothing$, the allowed continuations of $\zeta$ with $m(\zeta)\in (M\setminus L)\cap (x,y)$ live in a small ``Cantor set'' $K_{gap}$ in the gaps of $K(B)$, so that $\textrm{dim}_H((M\setminus L)\cap (x,y))\leq \textrm{dim}_H(K(C))+\textrm{dim}_H(K_{gap})$. 

As it turns out, the rest of this section relies on the formalisation of the idea of the previous paragraph based on a version of Lemma 6.1 of \cite{MatheusMoreira}. 

\begin{definition}
    \label{def:cplus}
    Consider two transitive and symmetric subshifts of finite type $\Sigma(B)
    \subset \Sigma(C)$. 
     Let $\alpha\in \Sigma(C)$ be a sequence with
$m(\alpha) = \lambda_0(\alpha) = m \in M$. We say that $\alpha$ {\it
connects positively} to~$B$, if for every $k\in\mathbb{N}$ there exist
a finite sequence~$\tau$ and 
an infinite sequence $\upsilon\in\Sigma^{+}(B)$ such that for
$\tilde{\alpha}:=\dots \alpha_{-2}\alpha_{-1}\alpha_0\dots \alpha_{k}
\tau\upsilon$ we have 
\begin{equation}
    \label{eq:m<m}
m(\tilde{\alpha})<m(\alpha)+2^{-k}.
\end{equation}
We say that $\alpha$
{\it connects negatively} to~$B$ if the reversed sequence $\alpha^t$ connects
positively to~$B$. 
\end{definition}
\begin{remark} Observe that we can replace~$2^{-k}$ in~\eqref{eq:m<m}
    by any sequence converging to~zero, or, in other words, the
    inequality~\eqref{eq:m<m} can be replaced by  
    \begin{equation}
        \label{eq:m<m2}
        \lim_{k\to\infty} \inf_{\substack{\tau\text{ finite word in }C, \\
        \upsilon \in\Sigma^{+}(B)}} m(\dots a_{-2}a_{-1}a_0^*\dots a_{k}
        \tau\upsilon)=m.  
    \end{equation}
\end{remark}
 The following equivalent definition is slightly more elaborate, but more 
 useful for our purposes. 
 \begin{manualdefinition}{3.2$^{\prime}$}
     \label{def:32p}
   Let $\Sigma(B) \subset \Sigma(C)$ be two transitive and symmetric subshifts of finite type. 
    We say that $\alpha \in \Sigma(C)$ connects positively to~$B$ if
    for every $k\in\mathbb{N}$ there exist a finite sequence $\tau$ and
    a pair of infinite sequences $\upsilon_C \in\Sigma^{+}(C)$ and
    $\upsilon_B \in\Sigma^{+}(B)$ such
    that the concatenation $\tilde{\alpha}:=\upsilon^t_C \alpha_{-k,k}
    \tau\upsilon_B$ satisfies $m(\tilde\alpha)<m(\alpha)+2^{-k}$.    
\end{manualdefinition}
 The advantage of 
    this more complicated alternative definition is that for each~$k$ the
    hypothesis is formulated in terms of the finite subsequence $\alpha_{-k,k}$. Notice that if $\alpha$ \emph{does not} connect to $B$, then there exists a fixed positive value of $k$ for which the condition above fails. In the sequel, instead of Lemma 6.1 in \cite{MatheusMoreira}, we shall use the following statement.

\begin{lemma}
    \label{l.key} 
    Consider two transitive and symmetric subshifts of finite type $\Sigma(B)
    \subset \Sigma(C)$. Let~$x$ be such that $m(\beta) \le x$
    for all $\beta \in \Sigma( B)$. 
    Suppose that a sequence $\gamma \in\Sigma(C)$ satisfies $m(\gamma) =
    \lambda_0(\gamma) = m > x$ and connects positively and negatively to~$B$. Then
    $m \in L$. 
\end{lemma}

\begin{proof} By Theorem~2 in Chapter~3 of Cusick--Flahive book~\cite{CF}, it is
    sufficient to show that $m=\lim\limits_{k\to\infty}m(P_k)$ where~$P_k$
    is a sequence of periodic points in~$\Sigma(C)$.  

    Since $\gamma$ connects positively and negatively to~$B$, there exist finite
    sequences $\tau, \widetilde\tau$ and infinite sequences $\upsilon, \widetilde
    \upsilon \in\Sigma^+(B)$ such that 
    $$
    m(\dots \gamma_{-2}\gamma_{-1}\gamma_0\dots
    \gamma_{k} \tau\upsilon)<m+2^{-k}\mbox{  and  } m(\widetilde \upsilon^t
    \widetilde\tau \gamma_{-k}\dots \gamma_0 \gamma_1 \gamma_2 \dots)<m+2^{-k}.
    $$
 
    Let $\upsilon^k:=\upsilon_1\dots\upsilon_k$ and 
    $\widetilde\upsilon^k: = \widetilde \upsilon_k \dots \widetilde
    \upsilon_1$ be the segments of $\upsilon$ and $\widetilde \upsilon^t$
    respectively. 
    By transitivity of $\Sigma(B)$, there exists $\beta \in \Sigma(B)$ which
    contains non-overlapping occurrencies of the strings $\upsilon^k$ and
    $\widetilde \upsilon^k$ in this order. Let us denote by
    $(\upsilon^k\ast\widetilde\upsilon^k)$ a finite substring of $\beta$ which begins with 
    $\upsilon^k$ and terminates with $\widetilde \upsilon^k$. 

   We next want to consider the periodic point $P_k\in \Sigma$ obtained by
   infinite concatenation of the finite block  
   $$
   \gamma_0 \dots \gamma_{k}\tau(\upsilon^k\ast\widetilde\upsilon^k)
   \gamma_{-k}\dots \gamma_{-1}.
   $$
   Recall that for any finite sequence $\xi=\xi_1\dots\xi_k$ of positive integers and for
   any pair of sequences $\alpha^{\prime}, \alpha^{\prime\prime} \in (\mathbb N^*)^{\mathbb N}$
   we have
   $|[0;\xi,\alpha^\prime]-[0;\xi,\alpha^{\prime\prime}]|<2^{1-k}$.
   Therefore for any  $j \in \mathbb Z$ we get
$\lambda_0(\sigma^j(P_k))\leq m + 2^{2-k}$ 
and $\lambda_0(P_k)>\lambda_0(\gamma)-2^{2-k} = m-2^{2-k}$.

In particular, $m=\lim\limits_{k\to\infty} m(P_k)$. This completes the argument. 
\end{proof}

\begin{remark}
    \label{rem:MtoL}
    Assume that $\Sigma(B) \subset \Sigma(C)$ are two transitive symmetric
    subshifts and let $x$ be such that $m(\beta) \le x$ for all $\beta \in
    \Sigma(B)$. Consider a sequence $\gamma \in \Sigma(C)$ with $m(\gamma) =
    \lambda_0(\gamma) = m > x$. 
    Then for any finite sequence $\tau$ and half-infinite sequence $\upsilon \in
    \Sigma^{+}(B)$ directly from definition of Lagrange and Markov
    numbers we get
$$
 \limsup_{j\to+\infty}\lambda_0(\sigma^j(\dots \gamma_{-2}\gamma_{-1}\gamma_0\dots
\gamma_{k} \tau\upsilon))<m,
$$
where $\sigma$ is the Bernoulli shift. Thus, if we want to get that $m(\dots \gamma_{-2}\gamma_{-1}\gamma_0\dots
\gamma_{k} \tau\upsilon)<m+2^{-k}$, then it suffices to check that  
\begin{equation*}
\lambda_0(\sigma^j(\dots \gamma_{-2}\gamma_{-1}\gamma_0\dots \gamma_{k}
\tau\upsilon))\!<\!m+2^{-k} 
\end{equation*}
for finitely many values of $j$, namely, for all $0\le j\le k+|\tau|+l$ where $l$ is 
sufficiently large (so that $2^{1-l}\le m-x+2^{-k}$). 
\end{remark}
The following elementary fact is quite useful to us. 
\begin{lemma} 
    \label{lem:tech}
    Let $\Sigma(C)$ be a transitive symmetric
    subshift.  
    Assume that three half-infinite sequences $\beta^1, \beta^2, \beta^3
    \in\Sigma^{+}(C)$ are such that 
    $[0;\beta^1]<[0;\beta^2]<[0;\beta^3]$. Then for all $\alpha \in \Sigma(C)$ 
    and for all $j \le n+1$
    \begin{multline*}
    \lambda_0\left(\sigma^j(\dots \alpha_{-2}\alpha_{-1}\alpha_0\dots
    \alpha_{n} \beta^2)\right) \le 
    \max\left(m(\dots \alpha_{-2}\alpha_{-1}\alpha_0\dots \alpha_{n} \beta^1),
    m(\dots \alpha_{-2}\alpha_{-1}\alpha_0\dots \alpha_{n} \beta^3)\right) .
    \end{multline*}
\end{lemma}

We will use Lemmas~\ref{l.key} and~\ref{lem:tech} and
Remark~\ref{rem:MtoL} in order to estimate Hausdorff
dimensions of $(M\setminus L)\cap(x,y)$ in the following way. Recall that~\cite[Lemma 6, Chapter 1]{CF}  
for any $m\in M$ there exists a sequence $\alpha$ such that $\lambda_0(\alpha) =
m(\alpha)$. Therefore to study $(M \setminus L) \cap (\sqrt{13},\sqrt{21})$ we may consider 
$$
Y:=\left\{\alpha \in \{1,2,3,4\}^{\mathbb Z} \mid m(\alpha)=\lambda_0(\alpha)\in
M\setminus L\right\}.
$$ 
In order to prove that $\dim_H(M\setminus L)\leq d$, it suffices to consider the cylinder sets $V_n(\alpha):=\left\{\tilde\alpha \in \{1,2,3,4\}^{\mathbb Z} \mid 
\tilde\alpha_j=\alpha_j, \, -n\le j\le n \right\}$ and 
to show that for every $\alpha \in Y$, there is $n\in \mathbb N$ such that  
$$
\dim_H(m(V_n(\alpha)\cap Y))\le d. 
$$ 
In this direction, we will associate (see Tables~\ref{tab:B1sets}
and~\ref{tab:Csets}) to an interval $(x,y)$ two symmetric
transitive subshifts of finite type $\Sigma(B) = \Sigma(B_x)\subset
\Sigma(C) = \Sigma(C_y)\subset(\mathbb{N}^*)^{\mathbb{Z}}$ such that 
\begin{itemize} 
    \item $m(\beta)<x$ for all $\beta \in \Sigma( B)$; and
    \item for all $\gamma\in(\mathbb{N}^*)^{\mathbb{Z}}$ such that $m(\gamma)<y$
        we have $\gamma \in \Sigma(C)$.  
\end{itemize}
If $m(\alpha) =
\lambda_0(\alpha) = m \in M\setminus L$, then by Lemma~\ref{l.key}, $\alpha$
doesn't connect neither positively nor negatively to~$B$. Suppose without loss of generality
that it doesn't connect positively to~$B$. Then by Definition~\ref{def:32p}
there exists $k\in\mathbb{N}$ such that, for any $N\geq k+2$, any finite sequence $\tau$ and infinite sequences
$\upsilon_C \in\Sigma^{+}(C)$ and $\upsilon_B \in\Sigma^{+}(B)$ the
concatenation $\widetilde{\alpha}= \upsilon_C^t \alpha_{-N,N} \tau\upsilon_B$
satisfies
$m(\widetilde \alpha)\ge m+2^{-k}\geq m+2^{-N+2}$.

At this point, we will proceed as follows. In the remainder of this section, for each interval $(x,y)$ introduced below, we will construct\footnote{In most cases below, $X_j$ is a \emph{pair} of finite sequences (e.g., $X_1=\{23, 1133\}$ in \S\ref{ss.384-392}), but sometimes we use larger finite sets (e.g., one of the $X_j$ in \S\ref{ss.44984-sqrt21} is $X_j=\{34313131, 344434, 213131\}$). In principle, we could explicitly list all $X_j$ appearing below, but, for the sake of simplicity of exposition, we will refrain from doing so: in other words, the relevant sets $X_j$ will always be \emph{implicit} in our subsequent discussions.} a finite collection $X_1, \dots, X_r$ of finite
sets of finite sequences over $\{1,2,3,4\}$ with the following property: if $x<m<m+2^{-N+2}<y$ and, for some $n\ge N$, a sequence $\upsilon_C^t \alpha_{-N,n}$ has continuations $\upsilon^1, \upsilon^2 \in
\{1,2,3,4\}^{\mathbb N}$ with \emph{different} subsequent term (of index $n+1$) leading to Markov values which are smaller than $m+2^{-N+2}$, then there is $X_j$ (depending only on $\alpha_{-N,n}$) such that the initial segment of \emph{any} $\upsilon \in \{1,2,3,4\}^{\mathbb N}$ with the
property that $m(\upsilon_C^t \alpha_{-N,n}\upsilon) < m + 2^{-N+2}$ belongs
to~$X_j$ (these elements tend to live on gaps of $K(B)$).

Notice now that $V_N(\alpha)\cap Y$ is contained in the set of sequences $\beta=\upsilon_1^t \alpha_{-N,N} \upsilon_2$ such that $m(\beta)<m+1/2^{N-2}$. Hence, if $s>0$ is such that, for all $X_1, \dots, X_r$ and all positive integers $b_1,\dots,b_n$, we have
$$\sum_{\tau\in X_j}|I(b_1,\dots,b_n,\tau)|^{s} \leq |I(b_1,\dots,b_n)|^{s},$$
where $I(a_1,\dots,a_k)=\{[0;a_1, \cdots, a_k,\rho] : \rho > 1\}$, then Markov values in $m(V_N(\alpha)\cap Y)$ belong to the arithmetic sum of $K(C)$ with a set $K_{gap}$ whose Hausdorff dimension is at most $s$, and thus its Hausdorff dimension is at most $d=\dim_H(K(C))+s$ (by a classical mass transference principle, see e.g. ~\cite[Proposition E.1]{book2021}).

We can now make a first choice of disjoint subintervals of $(\sqrt{13},\sqrt{21})$ with their
corresponding subshifts~$\Sigma(B)$, which we will subdivide further in the next subsections: cf. Table \ref{tab:B1sets} below. Note that these choices of $\Sigma(B)$ are \emph{simpler} than the original choices in \cite{MatheusMoreira} (and this is possible because Lemma \ref{l.key} is more flexible than \cite[Lemma 6.1]{MatheusMoreira}).
\begin{table}[h!]
\begin{tabular}{ |c| rl |  c |  l | }
\hline
$n$ & Interval & \kern-9pt$R_n$ & $B_n$ & \qquad  $\mathcal F_n$  \\ 
 \hline
1 & $(\sqrt{13}$, & $\kern-9pt3.92)$    & $ 1, 2 $ & $ \phantom{\liftl}  $ \\
\hline
2 &  $(3.92$, & $\kern-9pt4.32372)$     & $ 1, 2, 3 $ & $    13 $,  $  31 $ \\
\hline
3 & $(4.32372$, & $\kern-9pt4.4984)$    & $ 1, 2, 3 $ & $    131 $  \\
\hline
4 & $(4.4984$, & $\kern-9pt\sqrt{21})$   & $ 1, 2, 3 $ & $ 1313 $, $3131$  \\
\hline
\end{tabular}
\medskip
\caption{ Subshifts $\Sigma(B_n) = \{\beta \in 
B_n^{\mathbb Z} \mid \alpha \hbox{ has no substring from } \mathcal F_n \}$. }
\label{tab:B1sets}
\end{table}

Also, we collect together in Table~\ref{tab:Csets} the subshifts $C_y$ and the rigorous upper bounds 
on $\dim_H K(C_y)$ (derived from the same method as before, described in~\S\ref{s.PV-new}) we need for the sequel. 

\begin{table}
    \begin{tabular}{ |c| rl |  c |  l |  l |}
\hline
$n$ & Interval & \kern-9pt$S_n$ &$\mathcal A_n$ & \qquad \qquad \quad
$\mathcal F_n$ & $\dim_H K(C_n)$ \\ 
 \hline
1 & $(\sqrt{5}$, & $\kern-9pt3.042)$    & $ 1, 2 $ & $    121 $,   $  212  $,  $ 2111222 $, $ 2221112 $& $
 0.346453  \phantom{\liftl}  $ \\
 \hline
2 & $(\sqrt{13}$, & $\kern-9pt3.84)$    & $ 1, 2, 3 $ & $    13 $,   $  31  $ & $
 0.573961  \phantom{\liftl}  $ \\
\hline
3 &  $(3.84$, & $\kern-9pt3.92)$         & $ 1, 2, 3 $ & $    131 $,  $  313 $,  $
 231 $,  $132  $,  $  312  $, $213$ & $  0.594179  $ \\
\hline
4 & $(3.92$, & $\kern-9pt4.01)$         & $ 1, 2, 3 $ & $    131 $,  $  313 $,  $ 2312 $,  $2132 $ & $  0.643354  $ \\
\hline
5 & $(4.01$, & $\kern-9pt4.1165)$         & $ 1, 2, 3 $ & $    131 $ &  $0.666993$   \\
\hline
\multirow{2}{*}{6} & \multirow{2}{*}{ $(4.1165$,} & \multirow{2}{*}{$\kern-9pt4.1673)$}       &
\multirow{2}{*}{$ 1 $,  $  2 $,  $  3 $} & $1313$,  $
3131 $,  $  1312 $,  $  2131 $, & \multirow{2}{*}{$0.6694154  $} \\
& & & &  $  13111 $,   $  11131 $ &  \\
\hline
7 & $(4.1673$, & $\kern-9pt4.2527275 )$     & $1, 2, 3$ & $1313$, $3131$, $1312$,
 $2131$ & $  0.677846  $ \\ 
\hline
8 & $(4.2527275$, & $\kern-9pt4.32372)$     & $ 1 $,  $  2 $,  $  3     $ & $1313$,
 $3131$,  $21312 $ & $    0.691289  $ \\
\hline 
\multirow{3}{*}{9} &\multirow{3}{*}{  $(4.32372$,} & \multirow{3}{*}{$\kern-9pt4.385)$ }  &
\multirow{3}{*}{$1, 2, 3$}  &  $ 31313 $,  $  21313 $, $  31312 $,  $  21312 $,
& \multirow{3}{*}{$0.694718$} \\ 
& & & & $  1113131 $,  $  1313111 $,  $  3131112 $,  & \\
& & & &  $2111313 $, $  3131113 $,  $  3111313   $ &   \\
\hline 
$10$ & $(4.385$, & $\kern-9pt \sqrt{20})$      & $ 1 $,  $  2 $,  $  3     $ & $
31313 $,  $  31312 $,  $  21313  $,  $  121312 $,  $  213121 $ & $    0.697493  \phantom{\liftl}   $ \\
\hline 
$11$ & $(\sqrt{20}$, &\kern-9pt $4.4984)$     & $ 1 $,  $  2 $,  $  3 $ & 
$  31313   $ & $   0.704213  \phantom{\liftl}     $ \\
\hline 
$12$ & $(4.4984$, &\kern-9pt $4.513)$     & $ 1 $,  $  2 $,  $  3 $,  $  4  $ & $
14 $,  $  41 $,  $  24 $,  $  42 $,  $  343 $,  $  31313   $ & $   0.704700  \phantom{\liftl}     $ \\
\hline
$13$ & $(4.527$, &\kern-9pt $4.55)$     & $ 1 $,  $  2 $,  $  3 $,  $  4  $ & $
14 $,  $  41 $,  $  24 $,  $  42 $,  $  3433 $,  $  3343 $,  $  3434 $,  $  4343   $ & $   0.708245  \phantom{\liftl}     $ \\
\hline
$14$ & $(4.4984$, & $\kern-9pt\sqrt{21})$     & $ 1 $,  $  2 $,  $  3 $,  $  4  $ & $
14 $,  $  41 $,  $  24 $,  $  42   $ & $   0.709394  \phantom{\liftl} $ \\
\hline
\end{tabular}
\medskip
\caption{ Subshifts $\Sigma(C_n) = \{\alpha \in 
\mathcal{A}_n^{\mathbb Z} \mid \alpha \hbox{ has no substring from } \mathcal F_n \}$ used in our analysis, and dimension of
$K(C_n)$ calculated using the method from~\S\ref{s.PV-new}.}
\label{tab:Csets}
\end{table}


We are now ready to proceed to the detailed analysis of the sets $K_{gap}$ constructed below to
analyse different parts of $M\setminus L$. However, for the sake of completeness, let us briefly postpone this to the next subsections while closing the current discussion with an illustration of the method for the region $(M\setminus L)\cap(\sqrt{5},\sqrt{13})$. 

Take $C_1\subset\{1, 2\}^{\mathbb{Z}}$ where $121$, $212$, $2111222$ and $2221112$ are forbidden. Notice that $\lambda_0(12^*1)>3.15$, in the sense that if $\underline{a}=(a_n)_{n\in\mathbb{Z}}\in\Sigma(\mathcal{A}_1)$ and $(a_{-1},a_0,a_1)=(1,2,1)$ then $\lambda_0(\underline{a})>3.15$. Indeed, in this case, we have $\lambda_0(\underline{a})\ge [2;1,\overline{1,2}]+[0;1,\overline{1,2}]>3.15$. We also have $\lambda_0(21^*2)\ge [2;1,2,\overline{2,1}]+[0;\overline{2,1}]>3.06$ and $\lambda_0(222^*1112)\ge [2;2,2,\overline{2,1}]+[0;1,1,1,2,\overline{2,1}]>3.042$ (and so, by symmetry, $\lambda_0(21112^*22)>3.042$). These inequalities imply that $M\cap (\sqrt{5},3.042] \subset 2+K(C_1)+K(C_1)$, and thus 
$$\dim_H((M\setminus L)\cap (\sqrt{5},3.042]) \leq \dim_H(M\cap (\sqrt{5},3.042])\le 2\dim_H(K(C_1))<0.693,$$ since $\dim_H(K(C_1))<0.3465$ (as it can be checked with the method from ~\S\ref{s.PV-new}). 

Now let $(\mu,\nu)=(3.042,\sqrt{13})$. Here we take $C=\{1, 2\}^{\mathbb{Z}}$ and $B\subset\{1, 2\}^{\mathbb{Z}}$ where $121$, $212$ and $21112$ are forbidden. Note that if $\underline{b}\in \Sigma(B)$, then $m(\underline{b})\le [2;2,\overline{1,1,1,2,2,2}]+[0;1,1,1,1,1,\overline{2,2,2,1,1,1}]<3.041$. Given $\underline{a}\in Y$ with  $m=m(\underline{a})\in (\mu,\nu)$, we observe that if $N\ge 5$, then $m+1/2^{N-2}\le\sqrt{12}+1/2^3<\nu$. By the previous discussion (cf. Lemma \ref{l.key}), there is an integer $k$ (which we may assume to be at least $3$) such that, for any $N\ge k+2$, any finite sequence $\tau$ and infinite sequences $\gamma\in\Sigma^{+}(C)$, $\theta\in\Sigma^{+}(B)$, if $\hat{\underline{a}}=\gamma^t a_{-N}\dots a_0^*\dots a_{N} \tau\theta$, then $m(\hat{\underline{a}})\ge m+1/2^k\ge m+1/2^{N-2}$. 

Suppose that for some $n\ge N$, a sequence $\gamma^t a_{-N}\dots a_0^*\dots a_{n}$ has continuations with different subsequent term (of index $n+1$) whose Markov value are smaller than $m+1/2^{N-2}$ - in this case this means that this sequence has such a continuation with $a_{n+1}=1$ and another one with $a_{n+1}=2$. We \emph{claim} that these continuations should be of the type $a_n\alpha_n = a_n112\alpha_{n+3}$ and $a_n\beta_n = a_n221\beta_{n+3}$ thanks to the presence of the continuations $\overline{1122}$ and $\overline{2211}$.
Indeed, we have two cases:

$\bullet$ If $a_n1$ can be continued with $a_{n+2}a_{n+3}\ne 12$, since $[0;1,a_{n+2},a_{n+3}]>[0;1,1,2]$, it follows from  Lemma \ref{lem:tech} that the Markov values centered at $a_k$, with $k\le n$ and at $a_{n+1}=1$ are smaller than $m+1/2^{N-2}$; by Remark \ref{rem:MtoL}, it is enough to verify that $[2;2,\overline{1,1,2,2}]+[0;1,1,a_n,\dots]\le [2;2,\overline{1,1,2,2}]+[0;1,1,\overline{1,2}]<3.021<m$ in order to conclude that the Markov value of $\gamma^t a_{-N}\dots a_0^*\dots a_{n}\overline{1122}$ is smaller than $m+1/2^{N-2}$ and get the desired contradiction. 

$\bullet$ If $a_n2$ can be continued with $a_{n+2}a_{n+3}\ne 21$, since $[0;2,a_{n+2},a_{n+3}]<[0;2,2,1]$, it follows from  Lemma \ref{lem:tech} that the Markov values centered at $a_k$, with $k\le n$ and at $a_{n+1}=2$ are smaller than $m+1/2^{N-2}$;  by Remark \ref{rem:MtoL}, it is enough to verify that $[2;\overline{1,1,2,2}]+[0;2,a_n,\dots]\le [2;\overline{1,1,2,2}]+[0;2,\overline{2,1}]<3.01<m$ in order to conclude that the Markov value of $\gamma^t a_{-N}\dots a_0^*\dots a_{n}\overline{1122}$ is smaller than $m+1/2^{N-2}$ and derive again a contradiction. 

At this point, we recall that it was shown in \cite{MatheusMoreira} that, for $s=0.174813$, and all positive integers $b_1,\dots,b_n$, we have\footnote{Here and in the sequel, we use the well-known formula $|I(a_1,\dots,a_k)|=\frac{1}{q_k(q_k+q_{k-1})}$ where $q_j$ stands for the denominator of $[0;a_1,\dots, a_j]$.}
$$|I(b_1,\dots,b_n,1,1,2)|^{s}+|I(b_1,\dots,b_n,2,2,1)|^{s} \leq |I(b_1,\dots,b_n)|^{s}.$$
Thus, $\dim_H((M\setminus L)\cap (3.042,\sqrt{13}))\le 0.174813+\dim_H(E_2)<0.174813+0.531281=0.706094$. 

Hence, 
\begin{eqnarray*}
& &\dim_H((M\setminus L)\cap (-\infty,\sqrt{13})) = \\ &=&\max\{\dim_H((M\setminus L)\cap (-\infty,3.042],\dim_H((M\setminus L)\cap (3.042,\sqrt{13})\}\\ &\le&\max\{\dim_H(M\cap (-\infty,3.042],\dim_H((M\setminus L)\cap (3.042,\sqrt{13})\} \\ &\le&\max\{0.693,0.706094\}=0.706094 
\end{eqnarray*} 

\subsection{Improvement of the upper bounds in the region $(\sqrt{13}, 3.84)$}
As specified in Table~\ref{tab:Csets}, we choose
$\Sigma(C_2)  = \left\{ \alpha\in\{1,2,3\}^{\mathbb Z} \mid 13 \mbox{ and } 31 \mbox{ are not substrings of } \alpha \right\}$ and $\Sigma(B_1)=\{1,2\}^{\mathbb{Z}}$ to show that if $\alpha
\in Y$ and $m(\alpha) \in M \setminus L$ then there are two possibilities for 
the sequences $\alpha_n=(\upsilon^1_n, \upsilon^1_{n+1},\dots)$ and $\beta_n=(\upsilon^2_n, \upsilon^2_{n+1},\dots)$ with $\upsilon_n^1 \ne \upsilon_n^2$
corresponding Markov values in $(M\setminus L)\cap (\sqrt{13}, 3.84)$: 
\begin{itemize}
    \item[(A)] $\alpha_n=3\alpha_{n+1}$
        and $\beta_n=2\beta_{n+1}$ (i.e., $\upsilon^1_n=3$, $\upsilon_n^2=2$)
    \item[(B)] $\alpha_n=2\alpha_{n+1}$ and
        $\beta_n=1\beta_{n+1}$ (i.e., $\upsilon^1_n=2$, $\upsilon_n^2=1$)
\end{itemize} 
where $\alpha_{n+k}:=(\upsilon_{n+k}^1,\upsilon_{n+k+1}^1,\dots)$ and $\beta_{n+k}:=(\upsilon_{n+k}^2,\upsilon_{n+k+1}^2,\dots)$.

Let us first look at (A). It continues with $\beta_n=21\beta_{n+2}$ and, in
fact, we see that $21\beta_{n+2}=\overline{21}$ because $2$ appears in odd
positions, $1$ appears in even positions and $13$ is forbidden. Thus 
$[[0;\alpha_n],[0;\beta_n]] \cap
K(B)\neq \varnothing$. 

Let us now look at (B). It continues with $\beta_n=11\beta_{n+2}$. Since $13$ is
forbidden, $11\beta_{n+2}\to\dots\to 11\overline{21} \in K(B_1)$. Thus 
$
[[0;\alpha_n],[0;\beta_n]] \cap K(B_1) \ne
\varnothing. 
$
Therefore it is not possible to have two different continuations which do not
connect to~$B$. Hence $\dim_H K_{gap} = 0$. 

In particular, $\textrm{dim}_H((M\setminus L)\cap (\sqrt{13}, 3.84)) \leq \dim_H
(K(C_1)) \le 0.574$. 

\begin{remark} This estimate should be compared with the inequality $\textrm{dim}((M\setminus L)\cap (\sqrt{13}, 3.84))\geq \textrm{dim}(\Omega)>0.537109$ from \S\ref{ss.lower-bound-M-L}. 
\end{remark}

\subsection{Improvement of the upper bounds in the region $(3.84, 3.92)$}\label{ss.384-392} 

Similarly to~\cite{MatheusMoreira}, we can use $C_3 \subset
\{1,2,3\}^{\mathbb Z}$ where $131, 313, 231, 132$ are forbidden and a certain
block $B$ to show that the continuations of words with values in $(M\setminus
L)\cap (3.84, 3.92)$ are 
\begin{itemize}
\item $33$ and $21$
\item $23$ and $113$ or $1121$
\end{itemize} 

We affirm that the first Cantor set of gaps is trivial. Indeed, if the continuation $21$ is not $\overline{21}$, it must be $(21)^n3$ for some $n\in\mathbb{N}$, a contradiction because $213$ is \emph{forbidden} in this region as
$$[3, 1, 2, \overline{3, 1}] +[0, \overline{3, 1}]>3.95$$

Also, a similar argument shows that the option $1121$ is trivial. Thus, the Cantor set of the gaps in this region consist of the options $23$ and $1133$ (since $131$ and $132$ are forbidden in $C$). It follows that the Cantor set of the gaps has dimension $\dim_H K_{gap}<0.133$ because 
$$0.016134^{0.133}+(1/690)^{0.133}<1.$$
Since $\dim_H K(C_3) < 0.5942$,  we deduce that $\textrm{dim}_H((M\setminus L)\cap(3.84, 3.92)) < 0.5942+0.133=0.7272$. 

\subsection{Refinement of the control in the region $(3.92, 4.01)$} 
\subsubsection{Refinement of the control in the region $(3.92, 3.9623)$} 

Similarly to \cite{MatheusMoreira}, we can use $C\subset\{1,2,3\}^{\mathbb{Z}}$ where $131, 313, 2312, 2132$ are forbidden and a certain block $B$ to show that the continuations of words with values in $(M\setminus L)\cap (3.92, 3.9623)$ are 
\begin{itemize}
\item $331$ and $21$
\item $23$ and $113$ 
\end{itemize} 

Note that in this regime we have 
$$\lambda_0(3^*12)>[3;1,2,3,1,1,1,\overline{3,1}]+[0;3,1,2,1,3,3,\overline{3,1}] > 3.96238$$
so that the strings $312$ and $213$ are forbidden. Similarly, the strings $3231$, $1323$, $2231$, $1322$ are also forbidden. 

Thus, the continuations $331$ and $21$ are not possible in this regime: indeed, given that $213$ is forbidden, the smallest continuation of $21$ would be $\overline{21}$, so that we would be able to connect to the block $B$, a contradiction. 

Next, we affirm that the continuation $23$ and $113$ leads to $231$ and $113$: otherwise, if $232$ or $233$ is an allowed continuation, then we could use the largest continuation $\overline{23}$ to connect to an adequate block $B$, a contradiction. Since  
\begin{eqnarray*}
& & \left(\frac{|I(a_1,\dots, a_n,2,3,1)|}{|I(a_1,\dots, a_n)|}\right)^{0.153}+\left(\frac{|I(a_1,\dots, a_n,1,1,3)|}{|I(a_1,\dots, a_n)|}\right)^{0.153} \\ &\leq& (0.00254)^{0.153}+(1/63)^{0.153} < 1,
\end{eqnarray*}
we deduce that $\textrm{dim}_H((M\setminus L)\cap(3.92, 3.9623)) < 0.643355 + 0.153 = 0.796355$. 

\subsubsection{Refinement of the control in the region $(3.9623, 3.9845)$} 

Similarly to \cite{MatheusMoreira}, we can use $C\subset\{1,2,3\}^{\mathbb{Z}}$ where $131, 313, 2312, 2132$ are forbidden and a certain block $B$ to show that the continuations of words with values in $(M\setminus L)\cap (3.9623, 3.9845)$ are 
\begin{itemize}
\item $331$ and $21$
\item $23$ and $113$ 
\end{itemize}  

Since the strings $3231$ and $1323$ are forbidden in this regime (as $\lambda_0(323^*1)>3.99$), the same argument of the previous subsection says that $23$ and $113$ actually must be $231$ and $113$ where 
\begin{eqnarray*}
& & \left(\frac{|I(a_1,\dots, a_n,2,3,1)|}{|I(a_1,\dots, a_n)|}\right)^{0.153}+\left(\frac{|I(a_1,\dots, a_n,1,1,3)|}{|I(a_1,\dots, a_n)|}\right)^{0.153} \\ &\leq& (0.00254)^{0.153}+(1/63)^{0.153} < 1,
\end{eqnarray*}

Thus, it suffices to analyse the case $\alpha_n=331\alpha_{n+3}$ and $\beta_n=21\beta_{n+2}$. For this sake, note that, in the current region, the strings $1213$ and $3121$ are forbidden because $\lambda_0(3^*121) > [3;1,2,1,\overline{1,3}] + [0;3,1,2,1,\overline{3,1}] > 3.9866$ (as $313$ is forbidden). Also, the strings $23312$ and $33312$ are forbidden because $\lambda_0(333^*12)>\lambda_0(233^*12)>[3;1,2,\overline{3,1}]+[0;3,2,3,2,\overline{3,1}] > 3.98459$ (as $3231$ is forbidden). 

We claim that the $n$th digit $a_n$ (before $\alpha_n$ and $\beta_n$) is $2$ or
$3$: otherwise, we would have a continuation $1\beta_n = 121\beta_{n+2}$
connecting to the block $B$ (as the smallest continuation would be
$\overline{21}$). In view of the fact that $a_n\in\{2, 3\}$, we have that
$a_n\alpha_n = a_n331\alpha_{n+3} = a_n3311\alpha_{n+4}$ (as $313$, $23312$ and
$33312$ are forbidden) and, \emph{a fortiori}, $a_n\alpha_n =
a_n33111\alpha_{n+5}$ thanks to the presence of the continuation
$33\overline{1}$. Indeed, if $a_n3311$ can be continued with $r>1$, since
$[0;3,3,1,1,r]<[0;3,3,1,1,1]<[0;2,1]$, by Remark 3.5, it follows that the Markov
values centered at $a_k$, with $k\le n$ and at $a_{n+1}=3$ are smaller than $m$,
and it is enough to verify that $[3;\overline{1}]+[0;3,a_n,\dots]\le
[3;\overline{1}]+[0;3,\overline{3,1}]<3.93<m$. We will use implicitly this kind
of argument in several forthcoming cases.  Since 
\begin{eqnarray*}
& & \left(\frac{|I(a_1,\dots, a_n,3,3,1,1,1)|}{|I(a_1,\dots, a_n)|}\right)^{0.15}+\left(\frac{|I(a_1,\dots, a_n,2,1)|}{|I(a_1,\dots, a_n)|}\right)^{0.15} \\ &\leq& (2/3619)^{0.15}+(0.0718)^{0.15} < 1,
\end{eqnarray*}
we derive that $\textrm{dim}_H((M\setminus L)\cap(3.9623, 3.9845)) < 0.643355 + 0.153 = 0.796355$. 

\subsubsection{Refinement of the control in the region $(3.9845, 4.01)$} Similarly to \cite{MatheusMoreira}, we can use  $C\subset\{1,2,3\}^{\mathbb{Z}}$ where $131, 313, 2312, 2132$ are forbidden and a certain block $B$ to show that the continuations of words with values in $(M\setminus L)\cap (3.9845, 4.01)$ are 
\begin{itemize}
\item $\alpha_n=331\alpha_{n+3}$ and $\beta_n=21\beta_{n+2}$; 
\item $\alpha_n=23\alpha_{n+2}$ and $\beta_n=113\beta_{n+3}$. 
\end{itemize}   

Let us analyse the first possibility depending on the $n$th digit $a_n$ appearing before $331\alpha_{n+3}$ and $21\beta_{n+2}$:
\begin{itemize}
\item if $a_n=1$, then $\alpha_n=3312\alpha_{n+4}$ thanks to the presence of the continuation $3312\overline{21}$ (which is valid as $\lambda_0(133^*12)<3.984$); 
\item if $a_n\in\{2, 3\}$, then $\beta_n=213\beta_{n+3}$ thanks to the continuation $21331\overline{2}$. 
\end{itemize}

Similarly, we can decompose the second possibility into two subcases depending on the digit appearing before $23\alpha_{n+2}$ and $113\beta_{n+3}$: 
\begin{itemize}
\item if $a_n=1$, then $\alpha_n=231\alpha_{n+3}$ in view of $231\overline{12}$; 
\item if $a_n\in \{2, 3\}$, then $\beta_n=1132\beta_{n+4}$ in view of $1132\overline{12}$. 
\end{itemize}

Since 
\begin{eqnarray*}
& & \left(\frac{|I(a_1,\dots, a_n,3,3,1,2)|}{|I(a_1,\dots, a_n)|}\right)^{0.152}+\left(\frac{|I(a_1,\dots, a_n,2,1)|}{|I(a_1,\dots, a_n)|}\right)^{0.152} \\ &\leq& (1/1504)^{0.152}+(0.0718)^{0.152} < 1,
\end{eqnarray*}
\begin{eqnarray*}
& & \left(\frac{|I(a_1,\dots, a_n,3,3,1)|}{|I(a_1,\dots, a_n)|}\right)^{0.133}+\left(\frac{|I(a_1,\dots, a_n,2,1,3)|}{|I(a_1,\dots, a_n)|}\right)^{0.133} \\ &\leq& (1/255)^{0.133} + (0.0071)^{0.133} < 1,
\end{eqnarray*}
\begin{eqnarray*}
& & \left(\frac{|I(a_1,\dots, a_n,2,3,1)|}{|I(a_1,\dots, a_n)|}\right)^{0.153}+\left(\frac{|I(a_1,\dots, a_n,1,1,3)|}{|I(a_1,\dots, a_n)|}\right)^{0.153} \\ &\leq& (0.0071)^{0.153} + (1/63)^{0.153} < 1,
\end{eqnarray*}
\begin{eqnarray*}
& & \left(\frac{|I(a_1,\dots, a_n,2,3)|}{|I(a_1,\dots, a_n)|}\right)^{0.14}+\left(\frac{|I(a_1,\dots, a_n,1,1,3,2)|}{|I(a_1,\dots, a_n)|}\right)^{0.14} \\ &\leq& (0.0162)^{0.14} + (1/368)^{0.14} < 1,
\end{eqnarray*}
we derive that $\textrm{dim}_H((M\setminus L)\cap(3.9845, 4.01)) < 0.643355 + 0.153 = 0.796355$. 

\subsection{Refinement of the control in the region $(4.01, 4.1165)$}
\subsubsection{Refinement of the control in the region $(4.01, 4.054)$} Similarly to \cite{MatheusMoreira}, we can use  $C=\{1,2,3\}^{\mathbb{Z}}$ and a certain block $B$ to show that the continuations of words with values in $(M\setminus L)\cap (4.01, \sqrt{20})$ are 
\begin{itemize}
\item $\alpha_n=331\alpha_{n+3}$ and $\beta_n=213\beta_{n+3}$; 
\item $\alpha_n=23\alpha_{n+2}$ and $\beta_n=113\beta_{n+3}$. 
\end{itemize} 

Since $\lambda_0(13^*1)>4.1165$, the string $131$ is forbidden in our current regime. Also, $\lambda_0(213^*2)\geq [3;2,\overline{1,3,2,3}]+[0;1,2,\overline{3,1,3,2}] > 4.054$ when $13^*1$ is forbidden. Hence, the first transition extends as $331\alpha_{n+3}$ and $2133\beta_{n+4}$. 

Next, since $131$ is forbidden and $\lambda_0(113^*2\overline{12})<4.0078$, the second transition above is actually $\alpha_n=23\alpha_{n+2}$ and $\beta_n=1132\beta_{n+4}$. This transition extends in two possible ways: 
\begin{itemize}
\item if the digit appearing before is $a_n\in\{1,2\}$, we have $\lambda_0(a_n23^*1\overline{12})<4.0014$ and, hence, the transition becomes $\alpha_n=231\alpha_{n+3}$ and $\beta_n=1132\beta_{n+4}$; 
\item if the digit appearing before is $a_n=3$, we have $\lambda_0(a_n113^*23\overline{2})<4.0026$ and, thus, the transition becomes $\alpha_n=23\alpha_{n+2}$ and $\beta_n=11323\beta_{n+5}$. 
\end{itemize} 

Now, we recall that $\frac{|I(a_1,\dots,a_n,3,3,1)|}{|I(a_1,\dots,a_n)|}\leq \frac{1}{255}$, $\frac{|I(a_1,\dots,a_n,2,1,3,3)|}{|I(a_1,\dots,a_n)|}\leq 0.000641$, $\frac{|I(a_1,\dots,a_n,2,3,1)|}{|I(a_1,\dots,a_n)|}\leq 0.0071$, $\frac{|I(a_1,\dots,a_n,1,1,3,2)|}{|I(a_1,\dots,a_n)|}\leq \frac{1}{368}$, $\frac{|I(a_1,\dots,a_n,2,3)|}{|I(a_1,\dots,a_n)|}\leq 0.0162$, $\frac{|I(a_1,\dots,a_n,1,1,3,2,3)|}{|I(a_1,\dots,a_n)|}\leq \frac{1}{3905}$, and 
$$\max\{(\tfrac{1}{255})^{0.11}+(0.000641)^{0.11}, (\tfrac{1}{368})^{0.129}+(0.0071)^{0.129}, (\tfrac{1}{3905})^{0.117}+(0.0162)^{0.117}\}<1.$$ 
Therefore, $\textrm{dim}_H((M\setminus L)\cap(4.01, 4.054)) < 0.667 + 0.129 = 0.796$ (because the Cantor set of continued fraction expansions in $\{1,2,3\}^{\mathbb{N}}$ which avoid $131$ has dimension $<0.667$). 

\subsubsection{Refinement of the control in the region $(4.054, 4.06326)$} The string $131$ is still forbidden in our current regime and the same argument of the previous subsection can be employed to treat the second transition. Thus, it remains only to analyse the first transition $\alpha_n=331\alpha_{n+3}$ and $\beta_n=213\beta_{n+3}$. 

If the digit appearing before the first transition is $a_n=1$, we get a valid continuation $\lambda_0(a_n33^*133\overline{12})<4.0468$, so that the first transition becomes $\alpha_n=3313\alpha_{n+4}$ and $\beta_n=213\beta_{n+3}$. 

If the digit appearing before the first transition is $a_n=2$, we claim that $a_n2132$ is forbidden: indeed, $\lambda_0(213^*2b_m) > 4.1$ when $b_m\in\{2,3\}$, $\lambda_0(213^*211)>4.072$, $\lambda_0(a_n213^*212)>4.067$, $\lambda_0(a_n213^*2133) > 4.06352$, and 
$$\lambda_0(a_n213^*21321) \geq [3;2,1,3,2,1,2,\overline{3,1}]+[0;1,2,2,\overline{1,3,2,3}] > 4.06326$$ 
(here we used that $213211$ and $131$ are forbidden), so that all continuations of $a_n2132$ are large. Thus, the first transition becomes $\alpha_n=331\alpha_{n+3}$ and $\beta_n=2133\beta_{n+4}$. 

If the digit appearing before the first transition is $a_n=3$, we have that $a_n3313$ is also forbidden (as $\lambda_0(a_n33^*13)>4.0679$) and $a_n3312\overline{12}$ is a valid continuation (as $\lambda_0(a_n33^*12\overline{12}) < 4.03845$), so that the first transition becomes $\alpha_n=33121\alpha_{n+3}$ and $\beta_n=213\beta_{n+3}$. 

Since $\frac{|I(a_1,\dots,a_n,3,3,1,3)|}{|I(a_1,\dots,a_n)|}\leq \tfrac{1}{2592}$, $\frac{|I(a_1,\dots,a_n,2,3,1)|}{|I(a_1,\dots,a_n)|}\leq 0.0071$, $\frac{|I(a_1,\dots,a_n,3,3,1)|}{|I(a_1,\dots,a_n)|}\leq \frac{1}{255}$, $\frac{|I(a_1,\dots,a_n,2,1,3,3)|}{|I(a_1,\dots,a_n)|}\leq 0.000641$,  $\frac{|I(a_1,\dots,a_n,3,3,1,2,1)|}{|I(a_1,\dots,a_n)|}\leq \frac{1}{3552}$,  $\frac{|I(a_1,\dots,a_n,2,1,3)|}{|I(a_1,\dots,a_n)|}\leq 0.0071$ and 
$$\max\{(\tfrac{1}{2592})^{0.111}+(0.0071)^{0.111}, (\tfrac{1}{255})^{0.11}+(0.000641)^{0.11}, (\tfrac{1}{3552})^{0.11}+(0.0071)^{0.11}\}<1,$$
we conclude that $\textrm{dim}_H((M\setminus L)\cap(4.054, 4.06326)) < 0.667 + 0.129 = 0.796$. 

\subsubsection{Refinement of the control in the region $(4.06326, 4.0679)$} Once again, $131$ is still forbidden in our current regime, so that we can focus on the first transition $\alpha_n=331\alpha_{n+3}$ and $\beta_n=213\beta_{n+3}$. Actually, the fact that we are looking at Markov values below $4.0679$ makes that the same argument above can still be employed to treat the first transition in the case of the digits appearing before are $a_n\in \{1,3\}$. 

Finally, if the digit appearing before is $a_n=2$, then $a_n33133\overline{12}$ is a valid continuation because the fact that $131$ and $2132b_m$, $b_m\in\{2,3\}$ are forbidden says that 
$$\lambda_0(a_n33^*133\overline{12})\leq [3;1,3,3,\overline{1,2}]+[0;3,2,1,3,2,1,\overline{1,3}]<4.063251.$$
In particular, the first transition becomes $3313\alpha_{n+4}$ and $213\beta_{n+3}$ and we derive that $\textrm{dim}_H((M\setminus L)\cap(4.06326, 4.0679)) < 0.667 + 0.129 = 0.796$. 

\subsubsection{Refinement of the control in the region $(4.0679, 4.1)$} Since $131$ is forbidden here, it suffices to analyse the first transition $\alpha_n=331\alpha_{n+3}$ and $\beta_n=213\beta_{n+3}$. Moreover, the same argument above can still be employed to treat the first transition in the case of the digit appearing before is $a_n=1$. Furthermore, $2132b_m$, $b_m\in\{2,3\}$ is also forbidden here, so that the same argument above also treats the case of the first transition when the digit appearing before is  $a_n=2$. Finally, if the digit appearing before is $a_n=3$, we see that the first transition becomes $331\alpha_{n+3}$ and $2132\beta_{n+4}$ as $a_n213\overline{21}$ is a valid continuation (since $\lambda_0(a_n213^*\overline{21})<4.063582$). Given that $\frac{|I(a_1,\dots,a_n,3,3,1)|}{|I(a_1,\dots,a_n)|}\leq \frac{1}{255}$, $\frac{|I(a_1,\dots,a_n,2,1,3,2)|}{|I(a_1,\dots,a_n)|}\leq 0.00121$, and $(\tfrac{1}{255})^{0.114}+(0.00121)^{0.114}< 1$, we obtain that $\textrm{dim}_H((M\setminus L)\cap(4.06326, 4.0679)) < 0.667 + 0.129 = 0.796$. 

\subsubsection{Refinement of the control in the region $(4.1, 4.1165)$} Using for the last time that $131$ is forbidden, we will again concentrate only on the first transition $\alpha_n=331\alpha_{n+3}$ and $\beta_n=213\beta_{n+3}$. Here, we observe that $33133\overline{12}$ is a valid continuation (as $\lambda_0(33^*133\overline{12}) < 4.0721$), so the first transition becomes $3313\alpha_{n+4}$ and $213\beta_{n+3}$ and we get $\textrm{dim}((M\setminus L)\cap(4.1, 4.1165)) < 0.667 + 0.129 = 0.796$. 

\subsection{Refinement of the control in the region $(4.1165, 4.1673)$}
\subsubsection{Refinement of the control in the region $(4.1165, 4.1271)$} Recall that in the region $(M\setminus L)\cap (4.01, \sqrt{20})$ our task is to analyse the transitions   
\begin{itemize}
\item $\alpha_n=331\alpha_{n+3}$ and $\beta_n=213\beta_{n+3}$; 
\item $\alpha_n=23\alpha_{n+2}$ and $\beta_n=113\beta_{n+3}$. 
\end{itemize} 

We begin by observing that $\lambda_0(313^*1)>4.32372$, $\lambda_0(213^*1)>4.2527275$ (when $3131$ is forbidden), and the dimension of $C=1,2,3$ with $3131$, $2131$, $13111$ and their transposes forbidden is $<0.66942$. Recalling that $33133\overline{12}$ is a valid continuation (as $\lambda_0(33^*133\overline{12}) < 4.0721$), the first transition always becomes $3313\alpha_{n+4}$ and $213\beta_{n+3}$ in the region between $4.1165$ and $4.2527275$. Since $\frac{|I(a_1,\dots,a_n,3,3,1,3)|}{|I(a_1,\dots,a_n)|}\leq \tfrac{1}{2592}$, $\frac{|I(a_1,\dots,a_n,2,1,3)|}{|I(a_1,\dots,a_n)|}\leq 0.0071$, $(1/2592)^{0.111}+(0.0071)^{0.111}<1$, and $0.66942+0.111=0.78042$, the first transition is completely treated in the region between $4.1165$ and $4.2527275$. 

Let us now focus on the second transition. Since $23\overline{1}$ is a valid continuation (as $\lambda_0(23^*\overline{1})<4.06$), the second transition becomes $231\alpha_{n+3}$ and $113\beta_{n+3}$. 

If the digit appearing before is $a_n\in\{1,2\}$, since $\lambda_0(a_n113^*1) > 4.134215$ and $113\overline{23}$ is a valid continuation (as $\lambda_0(113^*\overline{23})<4.079$), the second transition becomes $231\alpha_{n+3}$ and $11323\beta_{n+5}$. Given that $\frac{|I(a_1,\dots,a_n,1,1,3,2,3)|}{|I(a_1,\dots,a_n)|}\leq \frac{1}{3905}$, $\frac{|I(a_1,\dots,a_n,2,3,1)|}{|I(a_1,\dots,a_n)|}\leq 0.0071$, $(1/3905)^{0.11}+(0.0071)^{0.11}<1$, and $0.66942+0.11=0.77942$, we are done. 

If the digit appearing before is $a_n=3$, we observe that $\lambda_0(a_n23^*1b_m)>4.1271$ for $b_m\in\{2,3\}$ and $\lambda_0(23^*111\overline{3})<4.081$, so that the second transition becomes $231113\alpha_{n+6}$ and $113\beta_{n+3}$. Given that $\frac{|I(a_1,\dots,a_n,2,3,1,1,1,3)|}{|I(a_1,\dots,a_n)|}\leq 0.0001$, $\frac{|I(a_1,\dots,a_n,1,1,3)|}{|I(a_1,\dots,a_n)|}\leq 1/63$, $(0.0001)^{0.11}+(1/63)^{0.11}<1$, and $0.66942 + 0.11 = 0.77942$, we are done. In summary, we showed that $\textrm{dim}((M\setminus L)\cap(4.1165, 4.1271)) < 0.78042$. 

\subsubsection{Refinement of the control in the region $(4.1271, 4.12733)$} In view of the arguments of the previous subsection, our task is reduced to discuss the second transition $231\alpha_{n+3}$ and $113\beta_{n+3}$ when the digit appearing before is $a_n=3$. 

Note that $\lambda_0(a_n23^*13)=\lambda_0(323^*13) > 4.199$. Moreover, we claim that all continuations of $a_n2312$ are large. Indeed, $\lambda_0(a_n23^*12b_m)>4.1358$ for $b_m\in\{1,2\}$,  $\lambda_0(a_n23^*123c_m)>4.1296$ for $c_m\in\{2,3\}$, $\lambda_0(a_n23^*1231d_m)>4.1275$ for $d_m\in\{1,2\}$. Since $\lambda_0(a_n23^*123133)>4.12733$, $\lambda_0(23^{**}132)>4.1288$, and $3131$ is forbidden, we conclude that $a_n2312$ has no short continuation. In view of the valid continuation $23111\overline{3}$ (with $\lambda_0(23^*111\overline{3})<4.081$), we see that the second transition becomes $231113\alpha_{n+5}$ and $113\beta_{n+3}$. Hence, we can apply again the argument from the previous subsection to derive that $\textrm{dim}((M\setminus L)\cap(4.1271, 4.12733)) < 0.78042$. 

\subsubsection{Refinement of the control in the region $(4.12733, 4.12762)$} In view of the arguments of the previous subsection, our task is again reduced to discuss the second transition $231\alpha_{n+3}$ and $113\beta_{n+3}$ when the digit appearing before is $a_n=3$. 

If the digit appearing before $a_n$ is $a_{n-1}\in\{1,2\}$, we have $\lambda_0(a_{n-1}a_n113^*\overline{113})<4.1264$, so that the second transition becomes $231\alpha_{n+3}$ and $113113\beta_{n+6}$ (since $1313$, $1312$, $13111$ and $13112$ are forbidden for any sequence with Markov value $<4.134215$). Given that $\frac{|I(a_1,\dots,a_n,2,3,1)|}{|I(a_1,\dots,a_n)|}\leq 0.0071$, $\frac{|I(a_1,\dots,a_n,1,1,3,1,1,3)|}{|I(a_1,\dots,a_n)|}\leq 0.000241$, $(0.0071)^{0.11}+(0.000241)^{0.11}<1$, and $0.66942 + 0.11 = 0.77942$, we are done. 

If the digit appearing before $a_n$ is $a_{n-1}=3$, we have three possibilities.
If the digit before $a_{n-1}$ is $a_{n-2}=3$, we have
$\lambda_0(a_{n-2}a_{n-1}a_n113^*\overline{113})<4.1272999969$ and, hence, the
argument of the previous paragraph can be repeated. If $a_{n-2}=2$, we recall
that $\lambda_0(323^*13)>4.199$,  $\lambda_0(23323^*12) > 4.1277$ (when $3131$
is forbidden) and $\lambda_0(23111\overline{3})<4.081$ to get that the second
transition becomes $231113\alpha_{n+6}$ and $113\beta_{n+3}$. Finally, if
$a_{n-2}=1$, then we have three subcases: if $a_{n-3}\in\{2,3\}$,  we note that
$\lambda_0(a_{n-3}a_{n-2}a_{n-1}a_n113^*1)\geq
\lambda_0(2133113^*113113\overline{3132})> 4.1277$ (when $1313$, $1312$, $13111$
and $13112$ are forbidden) and $\lambda_0(113^*\overline{23})<4.079$, so that
the second transition becomes $231\alpha_{n+3}$ and $11323\beta_{n+5}$; if
$a_{n-3}=1$ and $a_{n-4}\in\{1,2\}$, we get
$\lambda_0(a_{n-4}a_{n-3}a_{n-2}a_{n-1}a_n113^*1)\geq
\lambda_0(21133113^*113113\overline{3132}) > 4.12762$, so that the second
transition still is $231\alpha_{n+3}$ and $11323\beta_{n+5}$; if $a_{n-3}=1$ and
$a_{n-4}=3$, we get $\lambda_0(a_{n-4}a_{n-3}a_{n-2}a_{n-1}a_n23^*12)\geq
\lambda_0(3113323^*123133\overline{31})>4.12762$ (because
$\lambda_0(23^{**}132)>4.1288$), so that the second transition becomes
$231113\alpha_{n+6}$ and $113\beta_{n+3}$. In any event, we conclude that
$\textrm{dim}((M\setminus L)\cap(4.12733, 4.12762)) < 0.78042$. 

\subsubsection{Refinement of the control in the region $(4.12762, 4.134215)$} In view of the arguments of the previous subsection, our task is reduced to discuss the second transition $231\alpha_{n+3}$ and $113\beta_{n+3}$ when the digits appearing before are $a_n=3$, $a_{n-1}=3$ and $a_{n-2}\in\{1,2\}$. 

If $a_{n-2}=2$, we have $\lambda_0(233113^*\overline{113})<4.12751$, and $1313$, $1312$, $13111$ and $13112$ are forbidden on any sequence with Markov value $<4.134215$, so that the second transition becomes $231\alpha_{n+3}$ and $113113\beta_{n+6}$ and we are done. 

If $a_{n-2}=1$ and $a_{n-3}\in\{2,3\}$, we recall that $\lambda_0(323^*13)>4.199$, $\lambda_0(323^*12b_m)>4.1358$ for $b_m\in\{1,2\}$ and $\lambda_0(a_{n-3}a_{n-2}a_{n-1}a_n23^*123133\overline{1}) < 4.127471$, so that the second transition becomes $23123\alpha_{n+5}$ and $113\beta_{n+3}$. Given that $\frac{|I(a_1,\dots,a_n,2,3,1,2,3)|}{|I(a_1,\dots,a_n)|}\leq 0.000111$, $\frac{|I(a_1,\dots,a_n,1,1,3)|}{|I(a_1,\dots,a_n)|}\leq 1/63$, $(0.000111)^{0.111}+(1/63)^{0.111}<1$,  and $0.66942 + 0.111 = 0.78042$, we are done in this case. If $a_{n-2}=1=a_{n-3}$ and $a_{n-4}\in\{1,2\}$, we get $\lambda_0(a_{n-4}a_{n-3}a_{n-2}a_{n-1}a_n23^*123133211\overline{3}) < 4.12761982$ and the second transition still is $23123\alpha_{n+5}$ and $113\beta_{n+3}$, and we are done. If $a_{n-2}=1=a_{n-3}$ and $a_{n-4}=3$, we have $\lambda_0(a_{n-4}a_{n-3}a_{n-2}a_{n-1}a_n113^*113113\overline{32})<4.127618$, so that the second transition becomes $231\alpha_{n+3}$ and $113113\beta_{n+6}$, and we are done. 

In any case, we get that $\textrm{dim}((M\setminus L)\cap(4.127672, 4.134215)) < 0.78042$. 

\subsubsection{Refinement of the control in the region $(4.134215, 4.137519)$} In view of the arguments of the previous subsections, our task is to discuss the second transition $231\alpha_{n+3}$ and $113\beta_{n+3}$ when the digits appearing before are $a_n=2$ and $a_n=3$. For later reference, we remark that the Cantor set $C=1,2,3$ where $1313$, $1312$, $13111$ and their transposes are forbidden has dimension $<0.6694155$. 

If $a_n=2$, we have two possibilities. If the digit appearing before $a_n=2$ is $a_{n-1}\in \{2,3\}$, then $\lambda_0(a_{n-1}a_n113^*1) > 4.143241$ and $\lambda_0(113^*\overline{23})<4.079$, so that the second transition becomes $231\alpha_{n+3}$ and $11323\beta_{n+3}$ and we are done. If the digit before $a_n=2$ is $a_{n-1}=1$, we have $\lambda_0(a_{n-1}a_n23^*13)>4.1837$, $\lambda_0(a_{n-1}a_n23^*121)>4.137519$ (as $1313$, $1312$ and $13111$ are forbidden), and $\lambda_0(1223^*122\overline{12})<4.127$, so that the second transition becomes $23122\alpha_{n+5}$ and $113\beta_{n+3}$. Given that $\frac{|I(a_1,\dots,a_n,2,3,1,2,2)|}{|I(a_1,\dots,a_n)|}\leq 0.00021$, $\frac{|I(a_1,\dots,a_n,1,1,3)|}{|I(a_1,\dots,a_n)|}\leq 1/63$, $(0.00021)^{0.115}+(1/63)^{0.115}<1$,  and $0.67 + 0.115 = 0.785$, we are done. 

If $a_{n}=3$, since $1313$, $1312$ are forbidden and $\lambda_0(3113^*\overline{113})<4.128$, the second transition becomes $231\alpha_{n+3}$ and $11311\beta_{n+5}$. Since $\frac{|I(a_1,\dots,a_n,2,3,1)|}{|I(a_1,\dots,a_n)|}\leq  0.007042603$, $\frac{|I(a_1,\dots,a_n,1,1,3,1,1)|}{|I(a_1,\dots,a_n)|}\leq 1/400$,  $(0.007042603)^{0.1270292}+(1/400)^{0.1270292}<1$, and $0.6694155 + 0.1270292 =0.7964447<0.796445$, we are done. 

In any event, we get that $\textrm{dim}((M\setminus L)\cap(4.134215, 4.137519)) < 0.796445$. 

\subsubsection{Refinement of the control in the region $(4.137519, 4.1407)$} In view of the arguments of the previous subsections, our task is reduced to discuss the second transition $231\alpha_{n+3}$ and $113\beta_{n+3}$ when the digits appearing before are $a_n=2$ and $a_{n-1}=1$. 

If the digit before $a_{n-1}$ is $a_{n-2}\in\{2,3\}$, we have $\lambda_0(1223^*13)>4.1837$ and 
$$\lambda_0(a_{n-2}a_{n-1}a_n23^*121)>4.1409$$ 
(thanks to the fact that $1313$, $1312$, $13111$ are forbidden), so that we are back to the situation in the previous subsection. 

If the digit before $a_{n-1}$ is $a_{n-2}=1$, we have $\lambda_0(a_{n-2}a_{n-1}a_n113^*1)>4.1407$ 
(as $1313$, $1312$, $13111$ are forbidden) and $\lambda_0(2113^*\overline{23})<4.027$, so that the second transition becomes $231\alpha_{n+3}$ and $11323\beta_{n+5}$ and we are done. 

In summary, we get that $\textrm{dim}((M\setminus L)\cap(4.137519, 4.1407)) < 0.796445$. 

\subsubsection{Refinement of the control in the region $(4.1407, 4.1673)$} In view of the arguments of the previous subsections, our task is reduced to discuss the second transition $231\alpha_{n+3}$ and $113\beta_{n+3}$ when the digit appearing before is $a_n=2$ (since $11131$ is forbidden because $\lambda_0(1113^*1)>4.1673$). 

If the digit before $a_{n}$ is $a_{n-1}\in\{2,3\}$, we have $\lambda_0(a_{n-1}a_n23^*12\overline{1}) < 4.1387$ and $\lambda_0(a_{n-1}a_n23^*13)>4.175$, so that the second transition becomes $23121\alpha_{n+5}$ and $113\beta_{n+3}$. Since $\frac{|I(a_1,\dots,a_n,2,3,1,2,1)|}{|I(a_1,\dots,a_n)|}\leq 0.00051$, $\frac{|I(a_1,\dots,a_n,1,1,3)|}{|I(a_1,\dots,a_n)|}\leq \tfrac{1}{63}$, $(0.00051)^{0.122}+(\tfrac{1}{63})^{0.122}<1$, and $0.67 + 0.122 = 0.792$, we are done. 

If the digit before $a_{n}$ is $a_{n-1}=1$, we have two possibilities. If the digit before $a_{n-1}$ is $a_{n-2}=1$, then $\lambda_0(1223^*13)>4.1837$ and $\lambda_0(a_{n-2}a_{n-1}a_n23^*121133\overline{1})<4.13997$  
(as $1313$, $1312$, $13111$ are forbidden), so that the second transition becomes $23121\alpha_{n+5}$ and $113\beta_{n+3}$ and we are back to the situation of the previous paragraph. If the digit before $a_{n-1}$ is $a_{n-2}\in\{2,3\}$, then the facts that $1313$, $1312$ are forbidden, and $\lambda_0(a_{n-2}a_{n-1}a_n113^*\overline{113})<4.13984$ imply that the second transition is $231\alpha_{n+3}$ and $11311\beta_{n+5}$ and we are done.

In summary, we get that $\textrm{dim}((M\setminus L)\cap(4.1407, 4.1673)) < 0.796445$. 

\subsection{Refinement of the control in the region $(4.1673, 4.2527275)$} In view of the arguments of the previous subsections, our task is reduced to discuss the second transition $231\alpha_{n+3}$ and $113\beta_{n+3}$. We shall describe the possible extensions of this transition in terms of the digits appearing before and/or the Markov values of the words. 

If the digit appearing before is $a_n=1$, the second transition becomes $23132\alpha_{n+5}$ and $113\beta_{n+3}$ because $3131$ is forbidden and $\lambda_0(a_n23132\overline{12})<4.1619$. 

If the digit appearing before is $a_n=3$, the facts that $3131$ and $2131$ are forbidden, $\lambda_0(a_n23^*13)>4.1991$, $\lambda_0(a_n23^*1\overline{2})<4.149$, and $\lambda_0(a_n113^*\overline{113})<4.128$ can be used to say that the second transition becomes $2312\alpha_{n+4}$ and $11311\beta_{n+5}$ in the region $(4.1673, 4.199)$. Furthermore, the fact that $\lambda_0(a_n113^*111\overline{12}) < 4.1785$ allows to conclude that the second transition becomes $231\alpha_{n+3}$ and $113111\beta_{n+6}$ in the region $(4.199, 4.2527275)$.  

If the digit appearing before is $a_n=2$, the facts that $3131$ and $2131$ are forbidden, $\lambda_0(a_n23^*13)>4.175$, $\lambda_0(a_n23^*1\overline{2})<4.132$, and $\lambda_0(a_n113^*\overline{113})<4.1521$ can be used to say that the second transition becomes $2312\alpha_{n+4}$ and $11311\beta_{n+5}$ in the region $(4.1673, 4.175)$. Moreover, the fact that $\lambda_0(a_n113^*111) > 4.1857$ and $\lambda_0(a_n113^*11\overline{21})<4.16781$ allows to conclude that the second transition becomes $231\alpha_{n+3}$ and $113112\beta_{n+6}$ in the region $(4.175, 4.1857)$. Hence, it remains to analyse the region $(4.1857, 4.2527275)$ when $a_n=2$. 

If the digit appearing before $a_n=2$ is $a_{n-1}\in\{2,3\}$, the facts that $3131$ and $2131$ are forbidden, $\lambda_0(a_{n-1}223^*1\overline{3}) < 4.18261$, and $\lambda_0(a_n113^*\overline{113})<4.1521$ imply that the second transition becomes $2313\alpha_{n+4}$ and $11311\beta_{n+5}$ in the region $(4.1857, 4.2527275)$. Thus, it suffices to treat the case $a_n=2$ and $a_{n-1}=1$ in the region $(4.1857, 4.2527275)$. 

If the digit appearing before $a_{n-1}a_n=12$ is $a_{n-2}=1$, the facts that $3131$ and $2131$ are forbidden, $\lambda_0(a_{n-2}12113^*111) > 4.189$, and $\lambda_0(a_n113^*11\overline{21})<4.16781$ imply that the second transition becomes $231\alpha_{n+3}$ and $113112\beta_{n+6}$ in the region $(4.1857, 4.189)$. Also, the second transition becomes $2313\alpha_{n+4}$ and $11311\beta_{n+5}$ in the region $(4.189, 4.2527275)$ because $\lambda_0(a_{n-2}1223^*13\overline{3})<4.1881$.  

If the digit appearing before $a_{n-1}a_n=12$ is $a_{n-2}=3$, the facts that $3131$ and $2131$ are forbidden, $\lambda_0(a_{n-2}1223^*13) > 4.1889$, $\lambda_0(a_n23^*1\overline{2})<4.132$ and $\lambda_0(a_n113^*\overline{113})<4.1521$ imply that the second transition becomes $2312\alpha_{n+4}$ and $11311\beta_{n+5}$ in the region $(4.1857, 4.1889)$. Also, the second transition becomes $231\alpha_{n+3}$ and $113111\beta_{n+6}$ in the region $(4.1889, 4.2527275)$ because $\lambda_0(a_{n-2}12113^*1111\overline{32})<4.1865$. 

At this point, it remains only to investigate the region $(4.1857, 4.2527275)$ when the digit appearing before $a_{n-1}a_n=12$ is $a_{n-2}=2$. For this sake, we shall distinguish three subcases. 

\subsubsection{The subcase $a_{n-3}=3$ and $a_{n-2}a_{n-1}a_n=212$} Since $3131$ and $2131$ are forbidden, $\lambda_0(a_{n-3}212113^*111)\geq[3;111\overline{1311}]+[0;11212a_{n-3}\overline{31}]> 4.1876$. Because $\lambda_0(a_n113^*11\overline{21})<4.16781$, we conclude that the second transition becomes $231\alpha_{n+3}$ and $113112\beta_{n+6}$ in the region $(4.1857, 4.1876)$. Moreover, $\lambda_0(a_{n-3}21223^*133132\overline{1})<4.1874$ and $\lambda_0(a_n113^*\overline{113})<4.1521$, so that the second transition becomes $2313\alpha_{n+4}$ and $11311\beta_{n+5}$ in the region $(4.1876, 4.2527275)$. 

\subsubsection{The subcase $a_{n-3}=1$ and $a_{n-2}a_{n-1}a_n=212$} Since $3131$ and $2131$ are forbidden, $\lambda_0(a_{n-3}21223^*13)> 4.1878$. Because $\lambda_0(a_n23^*1\overline{2})<4.132$ and $\lambda_0(a_n113^*\overline{113})<4.1521$, we conclude that the second transition becomes $2312\alpha_{n+4}$ and $11311\beta_{n+5}$ in the region $(4.1857, 4.1878)$. Also, $\lambda_0(a_{n-3}212113^*11113113\overline{32})<4.1873$, so that the second transition becomes $231\alpha_{n+3}$ and $113111\beta_{n+6}$ in the region $(4.1878, 4.2527275)$. 

\subsubsection{The subcase $a_{n-3}=2$ and $a_{n-2}a_{n-1}a_n=212$} If the digit appearing before $a_{n-3}$ is $a_{n-4}\in\{2,3\}$, we have $\lambda_0(a_{n-4}221223^*13)>4.187566$, $\lambda_0(a_n23^*1\overline{2})<4.132$ and $\lambda_0(a_n113^*\overline{113})<4.1521$, so that the second transition becomes $2312\alpha_{n+4}$ and $11311\beta_{n+5}$ in the region $(4.1857, 4.187566)$. Moreover, $\lambda_0(a_{n-4}2212113^*11113113\overline{32}) < 4.187564$, so that the second transition becomes $231\alpha_{n+3}$ and $113111\beta_{n+5}$ in the region $(4.187566, 4.2527275)$. 

If the digit appearing before $a_{n-3}$ is $a_{n-4}=1$, we have $\lambda_0(a_{n-4}2212113^*111)\geq  [3;111131123\overline{1311}]+[0;1121221\overline{13}]>4.187546$ in the region $(4.1857, 4.199)$ because the strings $3131$, $2131$, $1113111$ are forbidden. Thus, the second transition becomes $231\alpha_{n+3}$ and $113112\beta_{n+6}$ in the region $(4.1857, 4.187546)$ since $\lambda_0(a_n113^*11\overline{21})<4.16781$. Moreover, $\lambda_0(a_{n-4}221223^*13313\overline{21}) < 4.187543$, so that the second transition becomes $2313\alpha_{n+4}$ and $11311\beta_{n+4}$ in the region $(4.187546, 4.2527275)$. 

In summary, we showed that the possibilities for the second transition in the region $(4.1673, 4.2527275)$ are $23132\alpha_{n+5}-113\beta_{n+3}$, $2312\alpha_{n+4}-11311\beta_{n+5}$, $231\alpha_{n+3}-113111\beta_{n+6}$, $231\alpha_{n+4}-113112\beta_{n+6}$ and $2313\alpha_{n+4}-11311\beta_{n+5}$. By combining this information with the facts that the Cantor set $C=$ $1,2,3$ with $3131$ and $2131$ forbidden has dimension $<0.67785$, and $0.67785+ 0.118 =0.79585$, we derive that $\textrm{dim}((M\setminus L)\cap(4.1673, 4.2527275)) < 0.79585$. 

\subsection{Refinement of the control in the region $(4.2527275, 4.32372)$} Recall that in the region $(4.01,\sqrt{20})$, we have to investigate the transitions: 
\begin{itemize}
\item $\alpha_n=331\alpha_{n+3}$ and $\beta_n=213\beta_{n+3}$; 
\item $\alpha_n=23\alpha_{n+2}$ and $\beta_n=113\beta_{n+3}$. 
\end{itemize} 

Since $3131$ is forbidden here, $\lambda_0(33^*13\overline{23}) < 4.081$, $\lambda_0(3313^{**}\overline{23})<4.21$, $\lambda_0(23^{***}2)<4$, the first transition becomes $331323\alpha_{n+6}$ and $213\beta_{n+3}$. 

Similarly, since $3131$ is forbidden, $\lambda_0(23^*132\overline{12}) < 4.21271$, $\lambda_0(113^*113\overline{1}) < 4.2022$ and $\lambda_0(113113^*\overline{1}) < 4.18$, so the second transition becomes $23132\alpha_{n+5}$ and $1131\beta_{n+4}$. 

Given that $\frac{|I(a_1,\dots, a_n,3,3,1,3,2,3)|}{|I(a_1,\dots, a_n)|}<\tfrac{1}{160678}$, $\frac{|I(a_1,\dots, a_n,2,1,3)|}{|I(a_1,\dots, a_n)|}<0.0071$, $\frac{|I(a_1,\dots, a_n,2,3,1,3,2)|}{|I(a_1,\dots, a_n)|} < 0.00012$, $\frac{|I(a_1,\dots, a_n,1,1,3,1)|}{|I(a_1,\dots, a_n)|}<\tfrac{1}{144}$, and 
$$(\tfrac{1}{160678})^{0.09}+(0.0071)^{0.09}<1, \quad (0.00012)^{0.1021}+(\tfrac{1}{144})^{0.1021}<1,$$
we conclude that $\textrm{dim}((M\setminus L)\cap(4.2527275, 4.32372)) < 0.6913+0.1021=0.7934$ thanks to the fact that $C=$ $1, 2, 3$ with $1313$, $3131$ and $21312$ forbidden has dimension $<0.6913$. 

\subsection{Refinement of the control in the region $(4.32372, 4.385)$} Similarly to \cite{MatheusMoreira}, we can use $C=\{1,2,3\}^{\mathbb{Z}}$ and a block $B$ to show that the continuations of words with values in $(M\setminus L)\cap (4.01, \sqrt{20})$ are 
\begin{itemize}
\item $\alpha_n=331\alpha_{n+3}$ and $\beta_n=213\beta_{n+3}$; 
\item $\alpha_n=23\alpha_{n+2}$ and $\beta_n=113\beta_{n+3}$. 
\end{itemize} 

For later use, we note that $31313$, $21313$, $31312$, $21312$, $1113131$, $1313111$, $3131112$, $2111313$, $3131113$, $3111313$ are forbidden here, and the corresponding Cantor set has dimension $<0.6948$. 

Let us discuss the first transition. The continuation $21311\overline{3231}$ is valid in the region $(4.3, 4.385)$, so that the first transition becomes $3313\alpha_{n+4}$ and $21311\beta_{n+5}$. 

Let us now investigate the second transition. The validity of the continuations $2313\overline{3}$ and $113\overline{113}$ say that the second transition becomes $2313\alpha_{n+4}$ and $1131\beta_{n+4}$. 

In the region $(4.3353, 4.385)$, the second transition is $231311\alpha_{n+4}$ and $1131\beta_{n+3}$ thanks to the valid continuation $231311311\overline{1323}$ and the fact that $31313$ and $31312$ are forbidden. 

Next, we observe that $313112$, $211313$, $313111$ and $111313$ are forbidden in the region $(4.32372, 4.3353)$. We will analyse the second transition in this region depending on the digits appearing before it. 

If the digit before is $a_n=3$, the second transition becomes $2313\alpha_{n+4}$ and $11313\beta_{n+5}$ in the region $(4.332, 4.3353)$ because the continuation $a_n11313\overline{3132}$ is valid, and it becomes $23132313\alpha_{n+8}$ and $1131\beta_{n+4}$ in the region $(4.32372, 4.332)$ because the string $23131$ is forbidden and the continuation $a_n2313231\overline{3}$ is valid. 

If the digit before is $a_n=2$, the second transition is $2313\alpha_{n+4}$ and $113121\beta_{n+5}$ in the region $(4.32372, 4.3353)$ because the continuation $a_n113121\overline{1323}$ is valid and $211313$ is forbidden. 

If the digits before are $a_{n-1}a_n=11$, the second transition is $2313\alpha_{n+4}$ and $11312\beta_{n+5}$ in the region $(4.32372, 4.3353)$ because $111313$ is forbidden and $a_{n-1}a_n11312\overline{3132}$ is a valid continuation. 

If the digits before are $a_{n-1}a_n=21$, the second transition is $2313\alpha_{n+4}$ and $11312\beta_{n+5}$ in the region $(4.329, 4.3353)$ because $111313$ is forbidden and $a_{n-1}a_n11312\overline{3132}$ is valid, and it becomes $23132313\alpha_{n+8}$ and $1131\beta_{n+4}$ in the region $(4.32372, 4.329)$ since $23131$ is forbidden and $2313231\overline{3}$ is valid.  

If the digits before are $a_{n-1}a_n=31$, the second transition is $23131\alpha_{n+5}$ and $1131\beta_{n+4}$ in the region $(4.332, 4.3353)$ because $a_{n-1}a_n231311311\overline{1323}$ is valid, and it becomes $2313\alpha_{n+4}$ and $1131113\beta_{n+7}$ in the region $(4.32372, 4.332)$ because $a_{n-1}a_n11312$ and $a_{n-1}a_n11313$ are forbidden and $a_{n-1}a_n1131113\overline{3132}$ is valid.  

In summary, we established that the possible transitions are $3313\alpha_{n+4}-21311\beta_{n+5}$, $231311\alpha_{n+6}-1131\beta_{n+4}$, $2313\alpha_{n+4}-11313\beta_{n+5}$, $23132313\alpha_{n+8}-1131\beta_{n+4}$, $2313\alpha_{n+4}-11312\beta_{n+5}$, $2313\alpha_{n+4}-1131113\beta_{n+7}$. Since $0.6948+ 0.10155=0.79635$, we conclude that $\textrm{dim}((M\setminus L)\cap(4.32372, 4.385)) < 0.79635$. 

\subsection{Refinement of the control in the region $(4.385, 4.41)$} Recall that in the region between $4.01$ and $\sqrt{20}$ our goal is to study the transitions:  
\begin{itemize}
\item $\alpha_n=331\alpha_{n+3}$ and $\beta_n=213\beta_{n+3}$; 
\item $\alpha_n=23\alpha_{n+2}$ and $\beta_n=113\beta_{n+3}$. 
\end{itemize} 

Let us discuss the second transition. It extends as $2313111\alpha_{n+7}$ and $1131\beta_{n+4}$ in this region because the strings $31313$, $31312$ and $21313$ are forbidden and the continuations $2313111\overline{1323}$ and $113\overline{113}$ are valid. 

We observe that the Cantor set $C=$ $1,2,3$ with $31313$, $31312$, $21313$, $121312$ and $213121$ forbidden (as they are in this region) has Hausdorff dimension $<0.6975$. 

The first transition becomes $3313111\alpha_{n+7}$ and $2131\beta_{n+4}$ in this region due to the valid continuations $3313111\overline{1323}$ and $2131\overline{1323}$ and to the fact that $331312$ and $31313$ are forbidden. 

In summary, we showed that the possible transitions in our region are $3313111\alpha_{n+7}-2131\beta_{n+4}$, and $2313111\alpha_{n+7}-1131\beta_{n+4}$.

Since $0.6975+ 0.0963 = 0.7938$, we conclude that $\textrm{dim}((M\setminus L)\cap(4.385, 4.41)) < 0.7938$. 

\subsection{Refinement of the control in the region $(4.41, \sqrt{20})$}
The second transition extends as $23131\alpha_{n+5}$ and $11313\beta_{n+5}$ in this region due to the valid continuations $23131\overline{1323}$ and $11313\overline{3132}$.

Let us now discuss the first transition. Since $33131\overline{1323}$ and $2131\overline{1323}$ are valid continuations when the Markov value is $>4.332$, the first transition extends $33131\alpha_{n+5}$ and $2131\beta_{n+4}$ in our region. 

In the region $(4.46151, \sqrt{20})$, we have a valid continuation $331312\overline{3132}$, so that the first transition becomes $331312\alpha_{n+6}$ and $2131\beta_{n+4}$ (as $31313$ is forbidden). Thus, it remains only to treat the region $(4.41, 4.46151)$. 

If the digits appearing before the first transition are $a_{n-1}a_n= 13$, the first transition becomes: 
\begin{itemize}
\item $33131\alpha_{n+5}$ and $21313\beta_{n+5}$ in the region $(4.4608, 4.46151)$ thanks to the valid continuation $21313\overline{3132}$; 
\item $3313111\alpha_{n+7}$ and $2131\beta_{n+4}$ in the region $(4.41, 4.461)$ due to the valid continuation $3313111\overline{1323}$ and the fact that $a_{n-1}a_n331312$ and $31313$ are forbidden.  
\end{itemize}  

If the digits appearing before the first transition are $a_{n-1}a_n\neq 13$, the first transition becomes: 
\begin{itemize}
\item $33131\alpha_{n+5}$ and $213121\beta_{n+5}$ in the region $(4.456, 4.46151)$ thanks to the fact that $a_{n-1}a_n21313$ is forbidden and the validity of the continuation $213121\overline{1323}$; 
\item $3313111\alpha_{n+7}$ and $2131\beta_{n+4}$ in the region $(4.41, 4.459)$ due to the valid continuation $3313111\overline{1323}$ and the fact that $a_{n-1}a_n331312$ and $31313$ are forbidden.  
\end{itemize}  

In summary, we showed that the possible transitions in our region are $331312\alpha_{n+6}-2131\beta_{n+4}$, $33131\alpha_{n+5}-213121\beta_{n+6}$, $3313111\alpha_{n+7}-2131\beta_{n+4}$, $33131\alpha_{n+5}-21313\beta_{n+5}$, and $23131\alpha_{n+5}-11313\beta_{n+5}$. Since $0.7057+  0.0903= 0.796$, we conclude that $\textrm{dim}((M\setminus L)\cap(4.41, \sqrt{20})) < 0.796$. 
 
\subsection{Refinement of the control in the region $(\sqrt{20}, 4.4984)$} 

Similarly to \cite{MatheusMoreira}, we can use $C\subset\{1,2,3, 4\}^{\mathbb{Z}}$ where $14, 41, 24, 42$ are forbidden and a certain block $B$ to show that the continuations of words with values in $(M\setminus L)\cap (\sqrt{20}, 4.4984)$ are 
\begin{itemize}
\item $\alpha_n=4\alpha_{n+1}$ and $\beta_n=3131\beta_{n+4}$, or 
\item $\alpha_n\in\{33131\alpha_{n+5}, 34\alpha_{n+3}\}$ and $\beta_n=2131\beta_{n+4}$, or 
\item $\alpha_n=23\alpha_{n+2}$ and $\beta_n=1131\beta_{n+4}$. 
\end{itemize} 

Note that a sequence containing the strings $343$ or $31313$ has Markov value $>4.52$. In particular, we can refine $C$ into $C=$ $1,2,3, 4$ where $14, 41, 24, 42, 343, 31313$ are forbidden. Note that $\textrm{dim}(C)<0.705$. 

If the sequence $\theta$ in $C$ contains $4$ and it is not $\overline{4}$ (whose Markov value is $\sqrt{20}$), then it contains $43$ and, \emph{a fortiori}, $\theta=\dots443\dots$. 

Suppose that $\lambda_0(\dots44^*3\dots)\geq \lambda_0(\dots4^*43\dots)$, i.e., 
$$4+\alpha+\frac{1}{4+\beta}\geq 4+\beta+\frac{1}{4+\alpha}$$ 
where $\alpha=[0;3,\dots]$. This would imply that $\alpha\geq\beta$, so that 
$$\lambda_0(\dots44^*3\dots)\geq 4+\alpha+\frac{1}{4+\alpha}$$ 
for $\alpha=[0;3,\dots]$. Because the minimal value of $\alpha$ extracted from a sequence $\theta\in C$ is 
$$\alpha\geq [0;\overline{3,1,3,1,2,1}],$$
we would have that 
$$\lambda_0(\dots44^*3\dots)\geq [4;\overline{3,1,3,1,2,1}]+[0;4,\overline{3,1,3,1,2,1}] > 4.4984.$$
Therefore, we can assume that $4$ doesn't appear in sequences $\theta$ producing Markov values in the interval $(\sqrt{20}, 4.4984)$. In particular, the continuations of words with values in $(M\setminus L)\cap (\sqrt{20}, 4.4984)$ are actually 
\begin{itemize}
\item[(i)] $\alpha_n=33131\alpha_{n+5}$ and $\beta_n=2131\beta_{n+4}$, or 
\item[(ii)] $\alpha_n=23\alpha_{n+2}$ and $\beta_n=1131\beta_{n+4}$. 
\end{itemize} 

We affirm that (i) has $\alpha_n=331312\alpha_{n+6}$: indeed, this happens because of the presence of the continuation $3313123\overline{1}$ (which is valid as $\lambda_0(\dots3313^*123\overline{1})\leq 4.463<\sqrt{20}$). Similarly, we affirm that (ii) has $\alpha_n=23131\alpha_{n+5}$ and $\beta_n=113131\beta_{n+6}$: in fact, this happens because of the presence of the continuations $2313\overline{1}$ and $11313\overline{1}$ (which are valid as $\lambda_0(\dots2313^*\overline{1}) < 4.394 < \sqrt{20}$ and $\lambda_0(\dots113^*13\overline{1})<4.42521<\sqrt{20}$). 

Since 
\begin{eqnarray*}
& &\left(\frac{|I(a_1,\dots, a_n, 3,3,1,3,1,2)|}{|I(a_1,\dots, a_n)|}\right)^{0.09} + \left(\frac{|I(a_1,\dots,a_n,2,1,3,1)|}{|I(a_1,\dots, a_n)|}\right)^{0.09} \\ 
&=&\left(\frac{(r+1)}{(53r+173)(72r+235)}\right)^{0.09} + \left(\frac{(r+1)}{(5r+14)(9r+25)}\right)^{0.09} \\ 
&\leq & (1/34691)^{0.09}+(0.003106)^{0.09} < 0.985 < 1 
\end{eqnarray*}
and 
\begin{eqnarray*}
& &\left(\frac{|I(a_1,\dots, a_n, 2,3,1,3,1)|}{|I(a_1,\dots, a_n)|}\right)^{0.09} + \left(\frac{|I(a_1,\dots,a_n,1,1,3,1,3,1)|}{|I(a_1,\dots, a_n)|}\right)^{0.09} \\
&=&\left(\frac{(r+1)}{(19r+43)(34r+77)}\right)^{0.09} + \left(\frac{(r+1)}{(24r+43)(43r+77)}\right)^{0.09} \\ 
&\leq& (0.00031)^{0.09}+(1/3311)^{0.09} < 0.966 < 1
\end{eqnarray*}
for $0<r<1$, we derive that $\textrm{dim}((M\setminus L)\cap(\sqrt{20}, 4.4984)) < 0.705+0.09=0.795$. 

\begin{remark} Even though this fact will not be used here, we note that $4.4984$ is somewhat close to the point $\alpha+\frac{1}{\alpha}=4.49846195\dots$, where 
$$\alpha=[4;31312133\overline{1131312231312111233131212112}],$$ which is the smallest element of the Lagrange spectrum accumulated by Lagrange values of sequences containing the letter $4$ infinitely often: cf. \cite{PMM}. 
\end{remark}

\subsection{Refinement of the control in the region $(4.4984, \sqrt{21})$}\label{ss.44984-sqrt21} Similarly to \cite{MatheusMoreira}, we can use $C\subset\{1,2,3, 4\}^{\mathbb{Z}}$ where $14, 41, 24, 42$ are forbidden and a certain block $B$ to show that the continuations of words $\gamma$ with values in $(M\setminus L)\cap (4.4984, \sqrt{21})$ are 
\begin{itemize}
\item[$(1)$] $\alpha_n=4\alpha_{n+1}$ and $\beta_n=3131\beta_{n+4}$, or 
\item[$(2)$] $\alpha_n\in\{33131\alpha_{n+5}, 34\alpha_{n+3}\}$ and $\beta_n=2131\beta_{n+4}$, or 
\item[$(3)$] $\alpha_n=23\alpha_{n+2}$ and $\beta_n=1131\beta_{n+4}$. 
\end{itemize} 

In the sequel, we shall significantly refine the analysis of these continuations. 

\subsubsection{The case $(1)$ of $\alpha_n=4\alpha_{n+1}$ and $\beta_n=3131\beta_{n+4}$} We affirm that $\alpha_n=44\alpha_{n+2}$ in this situation. In fact, given the nature of $C$, our task is to rule out the other possibility that $\alpha_n=43\alpha_{n+2}$. In this direction, the following lemma (obtained from a direct calculation) will be helpful:

\begin{lemma}\label{l.rational-Lagrange} $[4;\overline{4}]+[0;3,\overline{1,3,1,2}] = \frac{9}{2} = [3;1,3,\overline{4}]+[0;\overline{1,2,1,3}]$. In particular, 
$$9/2=m(\overline{4}3\overline{1312}) = \lim\limits_{n\to\infty} m(\overline{4^n 3 (1312)^n1313})\in L\cap\mathbb{Q}.$$
\end{lemma}

If $\gamma$ has allowed continuations $\alpha_n=43\alpha_{n+2}$ and $\beta_n=3131\beta_{n+4}$, then its Markov value is $<9/2$. Otherwise, its Markov value $m(\gamma)$ would be $>9/2$ (as $9/2\in L$ and $\gamma$ is assumed to give rise to an element of $M\setminus L$), and the previous lemma would permit us to connect $\gamma$ with an adequate block $B$ via $\gamma4^n3\overline{1312}$, a contradiction. 

Now, if $\gamma$ has allowed continuations $\alpha_n=43\alpha_{n+2}$ and $\beta_n=3131\beta_{n+4}$, and its Markov value is $<9/2$, then let us write $\gamma=\theta a_{n+1} \dots$, select $\beta\in\{\theta, \theta 4\}$ with $[\beta^t]\geq [4;\overline{4}]$ and let us consider the Markov value $m(\beta3\rho)<9/2$ of an allowed continuation $\beta3\rho$. If $[0;3\rho]\geq [0;3,\overline{1,3,1,2}]$, then $m(\beta3\rho)\geq [\beta^t]+[0;3,\rho]\geq [\overline{4}]+ [0;3,\overline{1,3,1,2}]=9/2$, a contradiction. If $[0;3,\rho]<[0;3,\overline{1,3,1,2}]$, then $\rho=13\mu$ with $[0;\mu]>[0;\overline{1,2,1,3}]$, but $m(\beta3\rho)<9/2$ would force $[3;\mu]+[0;1,3,\beta^t]<9/2$, so that  
$$9/2 > m(\beta3\rho)\geq [\beta^t]+[0;3,1,3,\mu] >  [\beta^t]+[0;3,1,9/2-[0;1,3,\beta^{t}]].$$  
This is a contradiction because the right-hand side is an increasing function of $[\beta^t]\geq [\overline{4}]$ whose value at $[\beta^t]=[\overline{4}]$ is $9/2$ after the previous lemma. 

In summary, we proved that, in any scenario, the case $(1)$ is actually 
\begin{itemize}
\item[$(1')$] $\alpha_n=44\alpha_{n+2}$ and $\beta_n=3131\beta_{n+4}$.
\end{itemize}

In what follows, we shall analyse the natural subdivision $(1')$ into two scenarios: 
\begin{itemize}
\item[$(1i)$] $\alpha_n=444\alpha_{n+3}$ is an allowed continuation, 
\item[$(1ii)$] $\alpha_n=443\alpha_{n+3}$ and $\beta_n=3131\beta_{n+4}$.  
\end{itemize}

\subsubsection{The subcase $(1i)$}

We affirm that $\beta_n=3131213\beta_{n+7}$ in this situation. In fact, let us begin by noticing that the Markov value of $\gamma=\rho a_{n+1}\dots$ with allowed continuations of type $(1i)$ is $<4.513$: otherwise, we would be able to connect to an adequate block $B$ by continuing with $\alpha_n'=443\overline{1}$ (since $[4;3,\overline{1}]+[0;4,\rho^t]\leq [4;3,\overline{1}]+[0;4,\overline{4,3}]<4.513$). Next, we observe that the strings $31313$ and $343$ are forbidden for any sequence with Markov value $<4.513$. 

Now, let us study the possible extensions of $\beta_n=3131\beta_{n+4}$. We have that $\gamma=\rho a_{n+1}\dots$ where  $\rho$ ends with $3$ or $4$ (because $\alpha_n=4\dots$ is allowed and $14, 24$ are forbidden strings). If $\rho$ ends with $4$, we observe that the estimate 
$$[3;1,3,\rho^t]+[0;1,2,1,\overline{3,3,1,3}]\leq [3;1,3,\overline{4,4,3,4}]+[0;1,2,1,\overline{3,3,1,3}]<4.49838$$ 
would allow to connect $\gamma$ to an adequate block $B$ unless  $\beta_n=3131213\beta_{n+7}$. Similarly, if $\rho$ ends with $3$, say $[\rho^t]=3+x$ with $0< x < [0;1,3,1,2,\overline{1,3}]$, then the continuation $\alpha_n=444\alpha_{n+3}$ would lead to an estimate  
\begin{eqnarray*}
m(\rho444\alpha_{n+3})&\geq& [4;3+x]+[0;4,4,\alpha_{n+3}] \geq [4;3+x]+[0;4,4,\overline{4,3,4,4}] \\ 
&\geq& [3;1,3,3+x] + [0;1,2,1,\overline{3,3,1,3}] + 0.000076
\end{eqnarray*}
allowing to connect $\gamma$ to an adequate block $B$ unless $\beta_n=3131213\beta_{n+7}$. 

In other terms, we showed that $(1i)$ actually is 
\begin{itemize}
\item[$(1i')$] $\alpha_n\in\{444\alpha_{n+3}, 443\alpha_{n+3}\}$ and $\beta_n=3131213\beta_{n+4}$.
\end{itemize}

Here, note that the relevant Cantor set $C=$ $1,2,3,4$ with where $14, 41, 24, 42, 343, 31313$ are forbidden has $\textrm{dim}(C)<0.705$, and 
\begin{eqnarray*}
& & \frac{|I(a_1,\dots, a_n,4,4)|^{0.086}+|I(a_1,\dots, a_n,3,1,3,1,2,1,3)|^{0.086}}{|I(a_1,\dots, a_n)|^{0.086}} \\ 
&\leq& \left(\frac{r+1}{(4r+17)(5r+21)}\right)^{0.086}+\left(\frac{r+1}{(71r+269)(90r+341)}\right)^{0.086} \\ 
&\leq& \left(\frac{1}{273}\right)^{0.086} + \left(\frac{1}{73270}\right)^{0.086}<1.
\end{eqnarray*}

\subsubsection{The subcase $(1ii)$} If the Markov value of $\gamma$ is $m(\gamma)\leq 4.5274$, then the string $31313$ is forbidden. In particular, $\beta_n=313121\beta_{n+6}$ (by comparison with $31312\overline{1}$). Here, 
\begin{eqnarray*}
& & \frac{|I(a_1,\dots, a_n,4,4,3)|^{0.087}+|I(a_1,\dots, a_n,3,1,3,1,2,1)|^{0.087}}{|I(a_1,\dots, a_n)|^{0.087}} \\ 
&\leq& \left(\frac{r+1}{(13r+55)(17r+72)}\right)^{0.087}+\left(\frac{r+1}{(14r+53)(19r+72)}\right)^{0.087} \\ 
&\leq& \left(\frac{1}{3026}\right)^{0.087} + \left(\frac{2}{6097}\right)^{0.087}<1.
\end{eqnarray*}

If the Markov value of $\gamma$ is $m(\gamma) > 4.5274$, then we affirm that it can not continued as $\alpha_n=4431\alpha_{n+4}$: otherwise, we would have a continuation $44323\overline{1}$ connecting to an adequate block $B$, a contradiction. This leaves us with two possibilities: 
\begin{itemize}
\item[$(1ii')$] $m(\gamma) \le 4.53422$, so $\alpha_n = 443\alpha_{n+3}$ cannot extend as $4433$ nor $4434$, and thus $\alpha_n = 443\alpha_{n+3}$ extends only as $\alpha_n=4432\alpha_{n+4}$;
\item[$(1ii'')$] $m(\gamma) > 4.53422$. 
\end{itemize}
In the subcase $(1ii')$, we observe that 
\begin{eqnarray*}
& & \frac{|I(a_1,\dots, a_n,4,4,3,2)|^{0.0881}+|I(a_1,\dots, a_n,3,1,3,1)|^{0.0881}}{|I(a_1,\dots, a_n)|^{0.0881}} \\ 
&\leq& \left(\frac{r+1}{(30r+127)(43r+182)}\right)^{0.0881}+\left(\frac{r+1}{(5r+19)(9r+34)}\right)^{0.0881} \\ 
&\leq& \left(\frac{2}{35325}\right)^{0.0881} + \left(\frac{1}{516}\right)^{0.0881}<1.
\end{eqnarray*}
In this case we will use the fact that the Cantor set $C=1,2,3,4$ where $41$, $42$, $434$, $433$ and their transposes are forbidden has dimension $<0.7081$. Notice that $0.7081+0.0881=0.7962$.

In the subcase $(1ii'')$, we note that $\beta_n=31313\beta_{n+5}$: otherwise, a continuation $31313443\overline{2}$ would allow to connect to an adequate block $B$, a contradiction. Here, we observe for later use that 
\begin{eqnarray*}
& & \frac{|I(a_1,\dots, a_n,4,4,3)|^{0.084}+|I(a_1,\dots, a_n,3,1,3,1,3)|^{0.084}}{|I(a_1,\dots, a_n)|^{0.084}} \\ 
&\leq& \left(\frac{r+1}{(13r+55)(17r+72)}\right)^{0.084}+\left(\frac{r+1}{(19r+72)(24r+91)}\right)^{0.084} \\ 
&\leq& \left(\frac{1}{3026}\right)^{0.084} + \left(\frac{2}{10465}\right)^{0.084}<1.
\end{eqnarray*}

\subsubsection{The case $(3)$ of $\alpha_n=23\alpha_{n+2}$ and $\beta_n=1131\beta_{n+4}$} Analogously to the analysis of this situation in the region $(\sqrt{20}, 4.4984)$, we have that $\alpha_n=23131\alpha_{n+5}$ and $\beta_n=113131\beta_{n+6}$ together with the estimate 
\begin{eqnarray*}
& &\left(\frac{|I(a_1,\dots, a_n, 2,3,1,3,1)|}{|I(a_1,\dots, a_n)|}\right)^{0.0857} + \left(\frac{|I(a_1,\dots,a_n,1,1,3,1,3,1)|}{|I(a_1,\dots, a_n)|}\right)^{0.0857} \\
&=&\left(\frac{(r+1)}{(19r+43)(34r+77)}\right)^{0.0857} + \left(\frac{(r+1)}{(24r+43)(43r+77)}\right)^{0.0857} \\ 
&\leq& (0.00031)^{0.0857}+(1/3311)^{0.0857} < 0.9997 < 1. 
\end{eqnarray*}

\subsubsection{The case $(2)$ of $\alpha_n\in\{33131\alpha_{n+5}, 34\alpha_{n+3}\}$ and $\beta_n=2131\beta_{n+4}$} Suppose that \emph{both} continuations $33131\alpha_{n+5}$ and $34\alpha_{n+3}$ are allowed. In this context, \emph{any} extension of $34\alpha_{n+3}$ which does \emph{not} increase Markov values would be valid. Among them, we see from the discussion in the beginning of subsection about the region $(\sqrt{20}, 4.4984)$ that such a minimal extension has the form $\rho3443\rho^t$. Thus, $\rho$ and $\rho^t$ can not connect on an adequate block $B$ and, hence, we could use Proposition 7.8 in \cite{MatheusMoreira} to get that the set of the Markov values associated to such $\gamma=\rho a_{n+1}\dots$ has Hausdorff dimension $<2\cdot 0.173 < 0.35$. 

Therefore, there is no loss of generality in assuming that only one of the continuations $33131\alpha_{n+5}$ and $34\alpha_{n+3}$ is allowed. 

If the continuation $34\alpha_{n+3}$ is not allowed, we have two possibilities:
\begin{itemize}
\item if $m(\gamma)<4.52$, then the strings $31313$ and $343$ are forbidden and we get $\alpha_n=331312\alpha_{n+6}$ (by comparison with $33131\overline{2}$) and $\beta_n=2131\beta_{n+4}$; 
\item if $m(\gamma)\geq 4.52$, then we get $\alpha_n=33131\alpha_{n+5}$ and $\beta_n=213131\beta_{n+6}$ (thanks to the continuation $21313\overline{1}$ with $\lambda_0(\dots213^*13\overline{1})<4.5197$). 
\end{itemize} 
In the first case, since $\frac{|I(a_1,\dots, a_n, 3,3,1,3,1,2)|}{|I(a_1,\dots, a_n)|}\leq \frac{r+1}{(53r+173)(72r+235)}\leq \frac{1}{34691}$, $\frac{|I(a_1,\dots, a_n, 2,1,3,1)|}{|I(a_1,\dots, a_n)|}\leq \frac{r+1}{(5r+14)(9r+25)}\leq 0.00311$, and 
$$\left(\frac{1}{34691}\right)^{0.09} + (0.00311)^{0.09} < 0.9851,$$
recalling that, if $C=$ $1,2,3, 4$ where $14, 41, 24, 42, 343, 31313$ are forbidden, then $\textrm{dim}(C)<0.705$ we will get the upper estimate $0.705+0.09=0.795$.
In the second case, since $\frac{|I(a_1,\dots, a_n, 3,3,1,3,1)|}{|I(a_1,\dots, a_n)|}\leq \frac{r+1}{(19r+62)(34r+111)}\leq \frac{2}{11745}$, $\frac{|I(a_1,\dots, a_n, 2,1,3,1,3,1)|}{|I(a_1,\dots, a_n)|}\leq \frac{r+1}{(24r+67)(43r+120)}\leq 0.000136$, and 
$$\quad \left(\frac{2}{11745}\right)^{0.08} + (0.000136)^{0.08} < 0.991,$$
it remains only to treat the possibility of $\alpha_n=34\alpha_{n+3}$ and $\beta_n=2131\beta_{n+4}$ being the unique allowed continuation. 

In the case of $\alpha_n=34\alpha_{n+3}$ and $\beta_n=2131\beta_{n+4}$, if $m(\gamma)<4.527$, then the  strings $31313$ and $343$ are forbidden and $\alpha_n=344\alpha_{n+4}$. If \emph{both} continuations $\alpha_n=3443\alpha_{n+5}$ and $\alpha_n=3444\alpha_{n+5}$ are allowed, then any continuation $3444\dots$ which does not increase Markov values would be allowed and the same analysis of the first paragraph of this subsection (considering sequences of the type $\gamma344443\gamma^t$) implies that the corresponding set of Markov values $m(\gamma)<4.527$ has Hausdorff dimension $<0.35$. In other words, there is no loss of generality in assuming that only one of the continuations $\alpha_n=3443\alpha_{n+5}$ or $\alpha_n=3444\alpha_{n+5}$ when $m(\gamma)<4.527$. Since $\frac{|I(a_1,\dots, a_n, 3,4,4,4)|}{|I(a_1,\dots, a_n)|}, \frac{|I(a_1,\dots, a_n, 3,4,4,3)|}{|I(a_1,\dots, a_n)|}\leq \frac{2}{71065}$, $\frac{|I(a_1,\dots, a_n, 2,1,3,1)|}{|I(a_1,\dots, a_n)|}\leq 0.00311$ and 
$$\left(\frac{2}{71065}\right)^{0.09}+(0.00311)^{0.09} < 0.985 < 1,$$
our task is reduced to discuss the case of $\alpha_n=34\alpha_{n+3}$, $\beta_n=2131\beta_{n+4}$, and $m(\gamma)\geq 4.527$. In this regime, the continuation $21313\overline{1}$ is allowed, so that $\beta_n=213131\beta_{n+7}$. If $4.527\leq m(\gamma)\leq 4.55$, the  strings $3433$, $3434$ and $2131313$ are forbidden (as $\lambda_0(34^*33)>4.56593$ and $\lambda_0(21313^*13)>4.55065$) so that $\beta_n=21313121\beta_{n+9}$ (thanks to the continuation $2131312\overline{1}$) in this context. Since 
 \begin{eqnarray*}
& & \left(\frac{|I(a_1,\dots, a_n, 3,4)|}{|I(a_1,\dots, a_n)|}\right)^{0.08745} + \left(\frac{|I(a_1,\dots, a_n, 2,1,3,1,3,1,2,1)|}{|I(a_1,\dots, a_n)|}\right)^{0.08745} \leq \\ &  & \left(\frac{2}{357}\right)^{0.08745}+(9.71\times 10^{-6})^{0.08745} < 0.99992 < 1,
\end{eqnarray*} 
 In this case we will use the fact that the Cantor set $C=1,2,3,4$ where $41$, $42$, $3433$, $3434$ and their transposes are forbidden has dimension $<0.7083$. Notice that $0.7083+0.08745=0.79575$.

It remains to treat the case $\alpha_n=34\alpha_{n+3}$, $\beta_n=213131\beta_{n+7}$, and $m(\gamma)> 4.55$. If $\alpha_n$ can \emph{not} be extended as \emph{both} $343$ or $344$, we can use the estimates $\frac{|I(a_1,\dots, a_n, 3,4,4)|}{|I(a_1,\dots, a_n)|}, \frac{|I(a_1,\dots, a_n, 3,4,3)|}{|I(a_1,\dots, a_n)|}\leq \frac{1}{1980}$, $\frac{|I(a_1,\dots, a_n, 2,1,3,1,3,1)|}{|I(a_1,\dots, a_n)|}\leq 0.000136$ and 
$ (1/1980)^{0.085}+(0.000136)^{0.085} <0.994$
to reduce our task to the study of the situation where \emph{both} extensions $343$ and $344$ \emph{are} allowed. Here, we can use the case (1) above to see that $344$ extends as $34443$ (as $\lambda_0(34^*444)<4.546$ would permit to  continue as $344443\overline{1}$ and so to connect to an adequate block $B$). Furthermore, $34443$ must continue as $344434$ (in view of the allowed continuation $3444344\overline{3}$). Also, $343$ extends as $34313131$ (thanks to the  continuation $3431313\overline{1}$ - notice that $[3;1,3,\overline{1}]+[0;1,3,4,3,\dots]\le [3;1,3,\overline{1}]+[0;1,3,4,3,\overline{1,3}]<4.55$ and, if $343$ can be followed by some word $a_{n+4}a_{n+5}a_{n+6}a_{n+7}a_{n+8}$ with $[0;a_{n+4},a_{n+5},a_{n+6},a_{n+7},a_{n+8}]<[0;1,3,1,3,1]$, then $[4;3,1,3,1,3,\overline{1}]<[4;3,a_{n+4},a_{n+5},a_{n+6},a_{n+7},a_{n+8},\dots]$). In summary, if $m(\gamma)>4.55$ and both $343$ and $344$ are permitted, then  
$\alpha_n\in\{34313131\alpha_{n+9}, 344434\alpha_{n+7}\}$ and $\beta_n=213131\beta_{n+7}$. Since 
$$
\frac{|I(a_1,\dots, a_n, 3,4,3,1,3,1,3,1)|}{|I(a_1,\dots, a_n)|}\leq \frac{2}{4835349}, \frac{|I(a_1,\dots, a_n, 3,4,4,4,3,4)|}{|I(a_1,\dots, a_n)|}\leq \frac{1}{11142860}, $$ 
$$\frac{|I(a_1,\dots, a_n, 2,1,3,1,3,1)|}{|I(a_1,\dots, a_n)|} \leq 0.000136$$ 
and $(\frac{2}{4835349})^{0.0857}+(\frac{1}{11142860})^{0.0857}+(0.000136)^{0.0857}< 0.999 < 1$, the analysis of the case (2) is now complete. 

\subsubsection{End of the study of the region $(4.4984,\sqrt{21})$} 

Our discussion of the cases $(1)$, $(2)$ and $(3)$ above shows that $\textrm{dim}((M\setminus L)\cap(4.4984, \sqrt{21})) < \max\{0.7094+0.0857, 0.7962\}=0.7962$.

\section{Algorithm for computing the Hausdorff dimension}\label{s.PV-new}

It remains to get rigorous  bounds for~$\dim_H B_1$, $\dim_H B_2$, $\dim_H X$,
$\dim_H Y$ and~$\dim_H\Omega$ claimed in the previous sections. 
Our approach uses connection between the Hausdorff dimension of a limit set and
the eigenvalue of the transfer operator. It was applied, for instance,
in~\cite{JP18} to estimate~$\dim_H E_2$. The algorithm used in~\cite{JP18} is
the so-called periodic points method and requires, in particular, accurate computation of all
periodic points up to a large period~$n$ in order to get accurate estimate on
dimension. The number of periodic points growth exponentially with~$n$, which
makes this method non-applicable to systems with large number of maps. 

In~\cite{PV20} a new approach has been developed for approximation of
the eigenvalue of the transfer operator using an approximation of the
corresponding eigenfunction. It is based on the Chebyshev spectral collocation
method. However, the complex character of the Gauss--Cantor sets $B_1$, $B_2$ and others
means that the associated space of functions will have a very large dimension
about $3\times 10^7$, which makes an accurate computation of the eigenfunction 
impossible at first sight. Nevertheless it turns out that in the cases we are
interested in the eigenfunction can be approximated by a polynomial which 
lies in a subspace of dimension less than $4000$! This allows us to approximate
it sufficiently accurately.  

In this section we would like to explain how to adapt the method developed
in~\cite{PV20} to the present setting, to make the computation practical.

In Appendix we give pseudocode for the computations described below.
The master program is given in Algorithm~\ref{alg:master}. It splits into two
parts. The first is combinatorial, where the problem of computing the
dimension of a subset of~$E_2$ is turned into a problem of computing the dimension of the
limit set of a certain iterated function scheme; following description
in~\S\ref{ss:setting}--\S\ref{ss:simplify} below; this part is covered by
Algorithm~\ref{alg:combi}. The second part deals with the
computation of the Hausdorff dimension and uses approach via transfer operators
inspired by thermodynamic formalism theory. This is covered by the main program
given in Algorithm~\ref{alg:main}, it has several subroutines given in
Algorithms~\ref{alg:poly}--\ref{alg:extra}. 

\subsection{The setting}
\label{ss:setting}
We begin with a general setting.

\begin{definition}
Let $\mathcal A = \{1,2\}$ be an alphabet. Let $\bar r = (r_1, \ldots, r_k) \in \mathbb N^k$ be the
vector of lengths. Consider a set of \emph{forbidden words} 
$$
F := F(\bar r) = \left\{
 d^{(1)}_{1} d^{(1)}_{2} \ldots d^{(1)}_{r_1} \in \mathcal A^{r_1}, \ 
 d^{(2)}_{1} d^{(2)}_{2} \ldots d^{(2)}_{r_2} \in \mathcal A^{r_2}, \ldots , 
 d^{(k)}_{1} d^{(k)}_{2} \ldots d^{(k)}_{r_k} \in \mathcal A^{r_k} 
 \right\}.
$$
A set $X_F\subset [0,1]$ is defined by continued
fraction expansions of its elements with extra Markov conditions:
\begin{align*}
    X_{F}:=\Bigl\{ [0; a_1,a_2,\ldots ] \mid  \ &a_n \in \mathcal \{1,2\}, \mbox{
such that for all } j \geq 1 \\ 
&a_j a_{j+1} \ldots a_{j+r_1} \ne d^{(1)}_{1} d^{(1)}_{2} \ldots d^{(1)}_{r_1} \\
&a_j a_{j+1} \ldots a_{j+r_2} \ne d^{(2)}_{1} d^{(2)}_{2} \ldots d^{(2)}_{r_2} \\
& \ \vdots \quad \vdots  \quad \qquad \vdots \qquad \qquad \vdots  \\
&a_j a_{j+1} \ldots a_{j+r_k} \ne d^{(k)}_{1} d^{(k)}_{2} \ldots d^{(k)}_{r_k} \Bigr\} \subsetneq E_2. 
\end{align*}
\end{definition}

We next want to introduce a Markov iterated function scheme of uniformly
contracting maps  whose limit set is
$X_{F}$. 
We take two maps
$T_1, T_2: [0,1] \to [0,1]$ defined by 
 $T_1 (x) =
\frac{1}{1+x}$ and $T_2(x) = \frac{1}{2+x}$ and consider all possible
compositions of length $n := \max_{1 \le i \le k} r_i - 1 $, i.e. we consider a collection of
maps 
$$
\mathcal T_n = \{ T_{\underline a_n}
:=T_{a_1} \circ \ldots \circ T_{a_n} \mid \underline a_n = a_1 \ldots a_n \in \{1,2\}^n 
\}.
$$ The Markov condition can be written as a
$2^n \times 2^n$--matrix $M = M(\underline a_n^j,\underline b_n^k)$ where
$$
M(\underline a_n^j, \underline b_n^k) = 
\begin{cases} 
    0, & \mbox{ if concatenation } \underline a_n^j \underline b^k_n \mbox{ contains }
    d_1^{(i)}\ldots d_{r_i}^{(i)} \mbox{ as a subword for some }  1 \le i \le k; \\
    1, & \mbox{ otherwise. }
\end{cases}
$$
It is a simple observation  that the limit set of $( \mathcal T_n, M)$ is equal to
$X_F$.

\subsection{The transfer operator}
 In order to compute the Hausdorff dimension of the limit set of
$(\mathcal T_n, M)$ we follow a general approach which dates back to Bowen and
Ruelle~\cite{Ruelle}. More precisely, we use the connection between Hausdorff
dimension of the limit set and the spectral radius of  a  transfer operator. 

The transfer operator associated to a Markov iterated function scheme is a
linear operator acting on the space of H\"older-continuous functions
$C^\alpha(\{1,\ldots,2^n\}\times [0,1])$, where $\{1,\ldots,2^n\}\times [0,1]$
represents  a disjoint union of $2^n$ copies of $[0,1]$
(\cite{PV20}, Section 2.4). It is defined by
$$
\mathcal L_t\colon (f_{\underline a^1_n},\ldots,f_{\underline a^{2^n}_n})
    \mapsto
    (F^t_{\underline b_n^1}, \ldots, F^t_{\underline b_n^{2^n}}),
$$
where
\begin{equation}
    \label{Ldef:eq}
    F^t_{b_n^k} (x) = \sum_{j=1}^{2^n} M(\underline a^j_n,\underline b^k_n)
    \cdot
    f_{\underline a_n^j}(T_{\underline a_n^j}(x)) \cdot |T_{\underline
    a_n^j}^\prime(x)|^t.
\end{equation}
Our method is based on the following result (originally due to Ruelle, generalizing a more specific result of Bowen for limit sets of Fuchsian-Schottky groups):  
\begin{proposition}[after~\cite{Ruelle}]
    Assume that the maximal positive eigenvalue of $\mathcal L_t$ is equal to~$1$. Then
    $\dim_H X_F = t$. 
\end{proposition}

In order to obtain lower and upper bounds on the maximal eigenvalue of $\mathcal
L_t$ we  use $\min$-$\max$ inequalities as described in
(\cite{PV20}, Section 3.1) which has the dual advantages of being easy to implement and also leading to rigorous results. More
precisely, our numerical estimates are based on the following, realised in
practice using Algorithms~\ref{alg:min-max}, \ref{alg:taylor} and
\ref{alg:extra}.
\begin{lemma}[\cite{PV20}]
    \label{lem:minmax}
    Assume that there exist two positive functions 
 $$
 \overline f =
(f_{\underline a^1_n}, \ldots, f_{\underline a^{2^n}_n} ), \ 
\overline g = (g_{\underline a^1_n}, \ldots, g_{\underline a_n^{2^n}}) \in C^\alpha(\{1,\ldots,2^n\}\times [0,1])
$$ 
such that 
for $\mathcal L_{t_0} \bar f = (F_{\underline b_n^1}, \ldots, F_{\underline
b_n^{2^n}})$ and  $\mathcal L_{t_1} \bar g = (G_{\underline b_n^1}, \ldots,
G_{\underline b_n^{2^n}})$ we have
\begin{equation}
    \label{eq:minmax}
    \min_j \inf_x \frac{F_{\underline b_n^j}(x)}{f_{\underline a_n^j}(x)} > 1 \qquad \mbox{ and } \qquad
\max_j \sup_x \frac{G_{\underline b_n^j}(x)}{g_{\underline a_n^j}(x)} < 1 
\end{equation}
Then $t_0 \le \dim_H X_F \le t_1$. 
\end{lemma}
We attempt to construct good choices of  functions $f_{\underline a_n^j}$ and $g_{\underline a_n^j}$ for
$1 \le j \le 2^n$ as positive polynomials of a relatively small degree
using the collocation method.
We fix a small natural~$m$ and define~$m$ Chebyshev nodes by
$$x_k :=\frac12\left(1 + \cos \left( \frac{\pi (2k-1)}{2m} \right)\right) \in [0,1]
\hbox{ for } 
k = 1, \ldots, m.$$ The
Lagrange interpolation polynomials are defined by $p_l(x) := \prod_{k=1}^{m}
\frac{x - x_k}{x_l - x_k}$. These are the unique polynomials of minimal degree
with the property that $p_l(x_k) = \delta_l^k$.  We then  consider the  subspace of
$C^\alpha(\{1,\ldots,2^n\} \times [0,1])$ spanned by $2^n$ copies of the space
$\langle p_k\rangle_{k=1}^m$: 
$$
\Pi(n,m) := \left\langle \{1, \ldots, 2^n\} \times \langle p_1, \ldots, p_m
\rangle \right\rangle 
\subset C^\alpha(\{1,\ldots,2^n\}\times [0,1]).
$$
Then the components of any $\bar q = (q_1, \ldots, q_{2^n}) \in \Pi(n,m)$ are
uniquely defined by their values at the Chebyshev nodes: 
\begin{equation}
    \label{eq:isoR2m}
q_j(x) = \sum_{i=1}^m q_j(x_i) p_i(x) \in \mathbb R[x]; \quad j = 1, \ldots, 2^n. 
\end{equation}
In particular, the formula~\eqref{eq:isoR2m} defines a bijection  $I \colon \mathbb
R^{2^nm} \to \Pi(n,m)$. 
We introduce a  projection operator 
$P \colon C^\alpha(\{1,\ldots,2^n\}\times[0,1])\to \mathbb R^{2^nm}$ given by 
$$
P (f_1, \ldots, f_{2^n}) \mapsto (f_1(x_1), \ldots, f_1(x_m), f_2(x_1), \ldots,
f_2(x_m), \ldots, f_{2^n}(x_1), \ldots, f_{2^n}(x_m)) \in \mathbb R^{2^nm}.
$$ 
We may now consider a finite rank  linear operator 
$B^t \colon \mathbb R^{2^nm} \to \mathbb R^{2^nm}$ defined by  
\begin{equation}
\label{eq:Btdef}
B^t \bar v = P \mathcal L_t I \bar v
\end{equation}
and construct the test functions~$\bar f$ and $\bar g$ in~\eqref{eq:minmax} from 
the eigenvectors $v_{t_0}$ and $v_{t_1}$ corresponding to the leading
eigenvalues of~$B^{t_0}$ and $B^{t_1}$ respectively using the formulae 
$f = I v_{t_0}$ and $g = I v_{t_1}$. The pseudocode is given in
Algorithms~\ref{alg:poly} and~\ref{alg:matrix}.

\begin{remark}
This approach appears to be relatively straightforward to implement numerically compared
to other methods. The bisection method can be used to get a refined estimate. 

Nevertheless, practical implementation is challenging for large values of~$n$. 
The first complication here is the computation of the matrix~$M$ that gives the
Markov condition, since at first sight it requires analysing of $2^{2n}$ words of
length~$2n$ searching for forbidden substrings, and the resulting matrix of the
size $2^{2n}$ would take about $2$GB of computer memory\footnote{A very
optimistic estimate is that one needs at least $2^{2n}$ bits and for $n=17$ we
get $2^{34}$ bits, which is $2^{31}$ bytes, exactly $2$GB.} to store for a modest 
value $n=17$ and for larger values $n>19$ the resulting Markov matrix wouldn't
fit into RAM memory of a personal computer. 

Furthermore, the matrix~$B^t$ is even larger and requires much more space as it is not a binary matrix and its
values need to be computed with higher accuracy. Typically we would like to work 
with~$128$ bits precision, so for modest values of $n = 17$ and $m=6$ it would
require $1512$GB just to store. 

A final  complication is that the computation of the eigenvector of a huge 
matrix with high accuracy is also very time-consuming in practice. 
The best method here for us would be the power method, which has complexity
of the matrix multiplication. The latter depends on the realisation, but is no less
than $O(n^{2.5})$.

\end{remark}

In the remainder of the section we explain how to 
refine the basic algorithm to make it more practical.

\subsection{
Simplifying the computation of the Markov matrix $M$
} 
\label{ss:simplify}
The next statement gives the basis for our approach for making the computation
possible. 
\begin{proposition}
    \label{prop:subspace}
    Assume that the columns $j_1$ and $j_2$ of the Markov matrix~$M$ are identical, i.e. for
    all $1\le k \le 2^n$ we have that $M(\underline a_n^k, \underline b_n^{j_1})
    \equiv M(\underline a_n^k, \underline b_n^{j_2})$. Then any eigenvector
    $\bar f$ of $B^t$ lies in the subspace of $\Pi(n,m)$ for which $f_{\underline
    a_n^{j_1}}= f_{\underline a_n^{j_2}}$.
\end{proposition}
We postpone the proof of this Proposition until Section~\ref{SS:matrixBt}.
Fortunately, it  turns out that for the sets of forbidden words we need to deal with, the
Markov matrix has a very small number of pairwise different columns compared to
its size.
\begin{example}
    For the specific sets which we study in this paper, we have the following. 
    \begin{enumerate} 
        \item In the case of the set $B_1$ which appears in
            Section~\S\ref{sss:t1low} the Markov matrix has $41186$ 
            columns, of which only $138$ are pairwise distinct. 
         \item In the case of the set $B_2$ which appears in Section~\S\ref{sss:t1up} the
    Markov matrix has $79034$ columns of which only $184$ are pairwise distinct. 

\item In the case of the set~$X$ which is used in Section~\S\ref{sss:t1low2} to
    obtain a lower bound on the transition value~$t_1$, the
    Markov matrix has $3940388$ columns of which only~$429$ are pairwise distinct.

\item    In the case of the set~$Y$ which is used in Section~\S\ref{sss:t1up2} to obtain
    the upper bound on the transition value $t_1$, the Markov
    matrix has $3940438$ columns of which only~$434$ are pairwise distinct. 
    
      \item In the case of the set $\Omega$ defined in \cite{MatheusMoreira} the Markov matrix has $45059$ 
            columns, of which only $114$ are pairwise distinct. 
        \end{enumerate}
\end{example}
Proposition~\ref{prop:minidim} below gives an upper bound on the number of
pairwise different columns in the transition matrix in terms of forbidden words.  

Therefore instead of computing (and storing) the entire Markov matrix~$M$ it is
sufficient to identify and to compute only unique columns, and to keep a record
of the indices of columns which are identical. This is a significant saving in
memory already, but there is room for even more. 

\begin{remark}
We note that if the rows $i_1$ and $i_2$ and the columns $j_1$ and $j_2$ of
the Markov matrix agree, i.e. $M(\underline a_n^{i_1}, \underline b_n^{k})
\equiv M(\underline a_n^{i_2}, \underline b_n^{k})$ and $M(\underline a_n^k, \underline b_n^{j_1})
    \equiv M(\underline a_n^k, \underline b_n^{j_2})$ for all $ 1 \le k \le
    2^n$, then $M(\underline a_n^{i_1}, \underline b_n^{j_1}) = M(\underline
    a_n^{i_1}, \underline b_n^{j_2}) = M(\underline a_n^{i_2}, \underline
    b_n^{j_2}) = M(\underline a_n^{i_2}, \underline b_n^{j_1})$. 
\end{remark}
This simple observation allows us to reduce significantly the memory needed to
store the Markov matrix~$M$. In particular, in our considerations the Markov
matrix~$M$ can be replaced by a smaller \emph{reduced} Markov matrix $\widehat M$ as follows. 
\begin{itemize}
    \item[Step 1.] Identify the words $\underline a_n^j$, $j = 1, \ldots, K$ such that the
        rows $M(\underline a_n^j, \cdot)$ are pairwise different and define the
        map $R$ which associates a row~$j$ with a row $R(j)$ from the set of unique rows. 
    \item[Step 2.] Identify the words $\underline b_n^j$, $j = 1, \ldots, K$ such
        that the columns $M(\cdot, \underline b_n^j)$ are pairwise different and
        define the map $C$ which associates a column~$j$ with a column $C(j)$
        from the set of unique columns. 
    \item[Step 3.] Compute the reduced Markov matrix $\widehat M = \widehat
        M(\underline a_n^j, \underline b_n^k)$ of the size $K\times K$. 
\end{itemize}
It is clear that the huge Markov matrix $M$ can be easily recovered from the
reduced matrix $\widehat M$ using the correspondence maps~$R$ and~$C$, since 
$M\left(\underline a_n^j,\underline b_n^k\right) = \widehat M\bigl(\underline a_n^{R(j)},
\underline b_n^{C(k)}\bigr)$, $1 \le k \le 2^n$.
Therefore, the main step in computing the Markov matrix~$M$ is the
computation of the sets of words which give unique columns and rows of~$M$
together with the maps~$R$ and $C$. 

\subsubsection{An upper bound for the number of unique rows and columns}
\begin{definition}
    We call the word $w^\prime = w_n \ldots w_1$ the 
    semordnilap or 
    reverse of the word $w =
    w_1 \ldots w_n$. 
\end{definition}
 It is easy to see that if the set of forbidden words includes every word together with its reverse, 
    then the number of pairwise different columns in the Markov matrix is equal to
    the number of pairwise different rows. In particular, the structure of the set of 
    forbidden words we consider implies that the reduced Markov matrix is a square matrix.

We need the following notation for the sequel.
\begin{definition}
    For any $ 1 \le k \le n-1$ we call a subword $ w_k \ldots w_1$  \emph{a
    suffix} of the  word $w_n \ldots w_1$ and a subword $w_n \ldots w_{k+1}$ \emph{a
    prefix} of the word $w_n \ldots w_1$.  
\end{definition}
We now want to give an absolute upper bound on the number of unique columns or,
equivalently, rows, of the reduced Markov matrix $\widehat M$ in terms of 
forbidden words. 
\begin{proposition}
    \label{prop:minidim}
    Assume that there are $k$ forbidden words which have~$P$ different suffixes
    in total. Then the number of pairwise distinct rows in the Markov matrix is no more
    than $P+1$. Similarly, the number of pairwise distinct suffixes gives an
    upper bound on the number of pairwise distinct columns.  
\end{proposition}

The first step in the computation, as outlined in
Algorithm~\ref{alg:combi}, is to identify all the words which contain a forbidden word as a
subword, since all of them give the zero row (or column) in the
transition matrix. After removing them from our consideration, we obtain the set of allowed words.
$$
A := \bigl\{ \underline w_n = w_1 \ldots w_n \in \mathcal A^n, \, w_j \ldots w_{j+r_i} \not\in F,
\mbox{ for all }  1 \le j \le n-r_i, \ 1 \le i \le k \bigr\}.
$$

Once the set~$A$ is computed, it is of course possible to study all concatenations of allowed words and to
identify those which give the unique rows and columns to the transition matrix.
However, this would require $O\left(\bigl(\sum_{i=1}^k r_i\bigr)\cdot(\#A)^2\right)$ operations, which is
prohibitively time-consuming, since $\#A$ is typically very large; more
precisely, in the
examples we consider we have $\#A \approx 2^{n-1}$. In the next subsection we give a
faster algorithm, which requires only $O\left( \#A \bigl(\sum_{i=1}^k
r_i\bigr)\right) + O\left(\bigl(\sum_{i=1}^k r_i\bigr)^4\right)$ operations. 

We denote by $P_F$ the set of prefixes of forbidden words and we denote by $S_F$ the set of
suffixes of forbidden words. Observe that entries of a column which corresponds to a word are determined by
suffixes of forbidden words it starts with.  

\begin{proof} (of Proposition~\ref{prop:minidim}). We may consider a mapping $g \colon A \to P_F$ that
associates to every allowed word~$w$ the longest prefix from $P_F$ which is a suffix of~$w$.
In other words $g(w) = \overline{w}$, if $w = w^\prime \overline{w}$, where
$\overline{w} \in P_F$ and $w^\prime$ is the shortest word with this property. 
Evidently, the function $g$ takes at most $P+1$ different values. 

We claim that if $g(w_1) = g(w_2)$, then the rows corresponding to the words $w_1$ and $w_2$
in the transition matrix are identical. Indeed, assume for a contradiction that
the rows are different. In other words, there exist a word $u \in A$ such that
concatenation $w_1 u$ contains a forbidden subword $fw \in F$, and concatenation $w_2 u$
doesn't contain any words from~$F$. Since $w_1, u \in \mathcal A$, we deduce
that $w_1$ contains a non-empty prefix of $fw$ as a suffix and $u$
contains a non-empty suffix of $fw$ as a prefix: 
$$
w_1 = w_1^\prime \overline{fw}, \quad u = \widehat{fw} u^\prime, \quad fw =
\overline{fw}\widehat{fw},  
$$
where $w_1^\prime \ne \varnothing$ and $u^\prime \ne \varnothing$ are some
words. Therefore we may write $g(w_1) = w_1^{\prime\prime} \overline{fw}$. Since
by assumption $g(w_1) = g(w_2)$, we see that $w_2 = w_2^\prime
w_1^{\prime\prime}\overline{fw}$. Hence concatenation $w_2 u$ contains the
word $fw$ and we get a contradiction. 
\end{proof}

Evidently, the total number of suffixes (or prefixes) is bounded by the sum of the
    lengths of all forbidden words: $\sum\limits_{j=1}^k r_j$. 

\subsubsection{Computation of the reduced Markov matrix $\widehat M$.}
\label{sss:reducedMat}
This can be realised by a number of technical steps 
(see Algorithm~\ref{alg:combi} for pseudocode). Let us denote by $|w|$ the
length of the word~$w$.
\begin{enumerate}
    \item Compute the sets $P_F = \{ \overline w \mbox{ is a prefix of } w \mid
        w \in F\}$ and $S_F = \{ \widehat w \mbox{ is a suffix of } w \mid w \in
        F \}$. 
    \item For every word $w \in A$ we compute:
        \begin{enumerate}
            \item The set of suffixes of forbidden
        words which are prefixes of $w$: $ SF_w   = \{ \overline w   \in
        S_F | \overline w  \mbox{ is a prefix of }w \}$; and 
    \item  The set of prefixes of forbidden words which are suffixes of $w$: 
        $
        PF_w = \{\widehat w \in P_F | \widehat w
        \mbox{ is a suffix of } w \}.
        $
\end{enumerate}
    \item We say that two words $w_1, w_2 \in A$ are ``suffix--equivalent'' if $SF_{w_1} =
        SF_{w_2}$ and we say that $w_1, w_2\in A$ are ``prefix--equivalent''
        if $PF_{w_1} = PF_{w_2}$. Thus we split the set of
        allowed words in equivalence  classes by suffixes $A/_{\sim S}$ and prefixes
        $A/_{\sim P}$. It turns out that there is
        relatively small number of equivalence classes compared to the number of
        allowed words. 

        In the next steps we explain that in order to decide whether two words
        are compatible it is sufficient to work with their equivalence classes.
            \item The following encoding is handy to study compatibility of words based
        on equivalence classes. First, we fix enumeration of the set of forbidden words~$F = \{ d_1,
        \ldots, d_k\}$. To every suffix~$\widehat d \in S_F$ of a forbidden word
        we associate a set of pairs $\{ (j,|\widehat d|) \mid \widehat d \mbox{ is
        a suffix of } d_j, d_j \in F \}$. To every prefix~$\overline d \in P_F$  we
        associate a set pairs $\{(j,|d_j| - |\overline d|) \mid \overline d
        \mbox{ is a prefix of } d_j, d_j \in F \}$. 

        Note that concatenation of a prefix and a
        suffix is a forbidden word, if their encodings are the same. 

    \item For any allowed word $w \in A$ we apply the encoding described above 
        to the equivalence classes $A/_{\sim S}$ and $A/_{\sim P}$. 
    \item It is clear that the concatenation of the words $w_1$ and $w_2$
        doesn't have a forbidden subword, if and only if the corresponding equivalence
        classes $AP_{w_1}$ and $AS_{w_2}$ do not have any common pairs after
        encoding. Therefore, instead of computing the Markov matrix for the
        set of allowed words it is sufficient to compute the compatibility
        matrix for the  equivalence classes. 
    \item We identify unique rows and columns in the compatibility matrix for the
        equivalency classes and choose representatives from each class to obtain
        words which give unique rows and columns in the reduced matrix $\widehat
        M$.
\end{enumerate}

The main advantage of this approach is that in order to compute the 
equivalency classes $A/_{\sim S}$ and $A/_{\sim P}$ it is sufficient to parse the
huge set of allowed words only once. The number of operations on subsequent
steps is $O\left( \bigl(\sum_{i=1}^k r_i\bigr)^4 \right)$. 

\subsection{Computation of the test functions}
\label{SS:matrixBt}
In order to construct the test functions to use in
Lemma~\ref{lem:minmax} we need to compute the eigenvector of the matrix $B^t$
defined by~\eqref{eq:Btdef}. By straightforward computation we can obtain the
explicit form of $B^t$. Indeed for any $v \in \mathbb R^{2^nm}$ we have 
\begin{equation}
    \label{Imap:eq}
Iv = (q_1(x), \ldots, q_{2^n}(x)), \quad q_j(x) = \sum_{l = 1}^{m} v_{ (j-1)m +
l} \cdot p_l(x), \mbox{ for } 1 \le l \le 2^n.  
\end{equation}
Therefore using~\eqref{Ldef:eq} we get $\mathcal L_t Iv = (Q_1, \ldots, Q_{2^n})$ 
where for all $1 \le k \le 2^n$ we have 
$$
\begin{aligned}
Q_k (x) &= \sum_{j=1}^{2^n} M(\underline a^j_n,\underline b^k_n)
    \cdot q_j (T_{\underline a_n^j}(x)) \cdot |T_{\underline
    a_n^j}^\prime(x)|^t \\ &=
    \sum_{j=1}^{2^n} M(\underline a_n^j, \underline b_n^k) \cdot |T_{\underline
    a_n^j}^\prime(x)|^t \cdot \left(\sum_{l=1}^m
    v_{(j-1)m + l} \cdot p_l( T_{\underline a_n^j}(x) ) \right).
\end{aligned}
$$
Hence the components of $u^t = (u_1^t, \ldots, u_{2^n}^t) = P \mathcal L_t I v$
are given by 
$$
u_{(k-1)m+i} = Q_k(x_i) = \sum_{j=1}^{2^n}  \sum_{l=1}^m
M(\underline a_n^j, \underline b_n^k) \cdot |T_{\underline
    a_n^j}^\prime(x_i)|^t\cdot p_l( T_{\underline a_n^j}(x_i) )  \cdot v_{(j-1)m
    + l} .
$$
Introducing $2^n$ small $m \times m $ matrices 
\begin{equation}
    \label{eq:Bsmall}
    B^{j,t}(i,l) :=  |T_{\underline
a_n^j}^\prime(x_i)|^t \cdot p_l(T_{\underline a_n^j}(x_i) )
\end{equation} 
we get 
\begin{equation}
    \label{Btmat:eq}
B^t = \begin{pmatrix}
    M(\underline a_n^1,\underline b_n^1) \cdot B^{1,t} &
    M(\underline a_n^2,\underline b_n^1) \cdot B^{2,t} &
    \ldots &
    M(\underline a_n^{2^n},\underline b_n^1) \cdot B^{2^n,t}  \\
    M(\underline a_n^1,\underline b_n^2) \cdot B^{1,t} &
    M(\underline a_n^2,\underline b_n^2) \cdot B^{2,t} &
    \ldots &
    M(\underline a_n^{2^n},\underline b_n^2) \cdot B^{2^n,t}  \\
    \vdots & \vdots & \ddots & \vdots \\
    M(\underline a_n^1,\underline b_n^{2^n}) \cdot B^{1,t} &
    M(\underline a_n^2,\underline b_n^{2^n}) \cdot B^{2,t} &
    \ldots &
    M(\underline a_n^{2^n},\underline b_n^{2^n}) \cdot B^{2^n,t}  \\
\end{pmatrix}.
\end{equation}
We are now ready to prove Proposition~\ref{prop:subspace}. 
\begin{proof} (of Proposition~\ref{prop:subspace}).
    Since by assumption $M(\underline a_n^k,\underline b_n^{j_1}) = M(\underline
    a_n^k, \underline b_n^{j_2})$ for all $k = 1, \ldots, 2^n$, using
    representation~\eqref{Btmat:eq} of the matrix~$B^t$ we conclude that 
    $$B^t\left(
    (j_1-1)m + l, k ) \right) = B^t\left( (j_2-1)m + l, k\right)$$
     for all $1 \le
    l \le m$ and $1 \le k \le 2^n$. Therefore for any $v \in \mathbb R^{2^nm}$
    and $ u = B^t v $ we have 
    $$
    u_{(j_1-1)m + l} = \sum_{k =1}^{2^n} B^t( (j_1-1)m+l,k) \cdot v_k = \sum_{k
    =1}^{2^n} B^t( (j_2-1)m+l,k) \cdot v_k = u_{(j_2-1)m + l}.  
    $$
    The result follows from~\eqref{Imap:eq} applied to~$u$. 
\end{proof}
We proceed to computing the leading eigenvector of $B^t$. By
Proposition~\ref{prop:subspace} it belongs to the subspace of dimension $\mathbb
R^{Km}$ where~$K$ is the number of pairwise different columns of the Markov
matrix~$M$. We have already mentioned that it is not possible to work with the 
matrix $B^t$ itself. The next Lemma allows us to replace the matrix
$B^t$ with a smaller reduced matrix $\widehat B^t$ in our considerations. 

\begin{definition}
   Assume that the Markov matrix $M$ has $K$ pairwise different columns. Let
    $j_1, \ldots, j_K$ be the indices of the unique columns of $M$ and let $i_1
    \ldots, i_K$ be the indices of the unique rows of $M$. Let
    $R$ be the correspondence map as constructed in Step 1. We define the
    reduced matrix $\widehat B^t$ by 
    \begin{equation}
        \label{eq:Breduced}
    \widehat B^t( (k-1)m+i, (l-1)m + j) = \sum_{\stackrel{s: R(s) = l }{ 1 \le
    s \le 2^n}  } M(i_k, s) B^{s,t}(i, j), \quad 1 \le l \le K.    
    \end{equation}
\end{definition}

\begin{remark}
    Note that in order to compute the matrix $\widehat B^t$ there is no need to
    \emph{store} the matrix $B^t$. It is sufficient to add elements of
    $B^t$ to the corresponding elements of $\widehat B^t$ as we compute them. 
\end{remark}

\begin{lemma}
    Let $\hat v$ be the eigenvector of $\widehat B^t$. Then the eigenvector of
    $B^t$ can be computed using the formula $v_{(j-1)m+l} := \hat v_{(R(j)-1) m
    + l}$ for $ 1 \le l \le m$ and $1 \le j \le 2^n$. 
\end{lemma}
\begin{proof}
    Let $P \colon \mathbb R^{2^n m} \to \mathbb R^{Km}$ be the orthogonal
    projector onto the subspace defined by the system of equations $v_{i_k} = v_{s}$ where
    $R(s) = i_k$ for all $1 \le s \le 2^n$ and $ 1 \le k \le K$. Then $\widehat
    B^t = P B^t P^*$, where $P^*$ stands for the transposed matrix $P$.  
\end{proof}
Therefore, in order to recover the eigenvector of $B^t$ and to compute the test
functions, it is sufficient to compute the eigenvector of a much smaller reduced matrix
$\widehat B^t$, defined above. The latter can be realised using simple
iterations method (see Algorithms~\ref{alg:poly} and~\ref{alg:matrix} for
pseudocode). 

\subsection{Verification of the min-max inequalities}

Finally, to verify the conditions of Lemma~\ref{lem:minmax} numerically, we
follow the same method as proposed in~\cite{PV20}. 

First, we compute the coefficients of the
polynomials $p_1, \ldots, p_K$ from the eigenvector of $\widehat B^t$. 

The transfer
operator~$\mathcal L_t$ given by~\eqref{Ldef:eq} can be written using the
reduced Markov matrix $\widehat M$ and the correspondence map~$R$. More
precisely, let $j_1, \ldots, j_K$ be the indices of the unique columns in
matrix~$M$. Let $p_1, \ldots, p_K$ be the polynomials constructed from the
eigenvector of the $\widehat B^t$. Then the transfer operator takes the form
$\mathcal L_t\colon (p_1,\ldots,p_K)  \mapsto (Q^t_1, \ldots, Q^t_K)$
where
\begin{equation}
    Q^t_{k} (x) = \sum_{i=1}^{2^n} \widehat M(\underline
    a^{R(i)}_n,\underline b^{j_k}_n) \cdot
    p_{R(i)} (T_{\underline a_n^i}(x)) \cdot |T_{\underline
    a_n^i}^\prime(x)|^t.
\end{equation}

In order to obtain upper and a lower bounds on $\frac{Q^t_K}{p_K}$ we 
take a partition of the interval $[0,1]$ into $256$ equal intervals. Then we
evaluate $\frac{Q_k}{p_k}$ at the centre and to compute
$\sup\Bigl|\frac{d}{dx} \frac{Q_k(x)}{p_k(x)} \Bigr|$ on each interval using
Taylor series expansion. The
latter is realised using arbitrary precision ball arithmetic~\cite{J16}.
Pseudocode is given in Algorithms~\ref{alg:taylor} and~\ref{alg:extra}.

\subsection{Computational aspects}
Here we give some numerical data. The decimal numbers which we give in this section are truncated, not rounded. 
\subsubsection{Set $B_1$}\label{B1set}
We apply the method described above to estimate the dimension of the set $B_1$. 
In this case, the set of forbidden words constitutes of $27$ words of length
from~$5$ to~$17$.  
\begin{equation*} 
\begin{split} 
 F&=\{
 2   1   2   1   2 , \, 
 2   1   1   1   2   1   2   1 , \, 
 1   2   1   1   1   2   1   2 , \, 
 2   1   1   1   2   1   2   2   2 , \, 
 1   1   1   1   1   2   1   2   1 , \,  
 1   1   1   1   2   1   2   1   1   1 , \, \\ & 
 2   1   1   1   2   1   2   2   1   1   2    , \, 
 2   2   2   1   1   1   2   1   2   2   1   1 , \, 
 1   2   2   1   1   1   1   2   1   2   1   1   2   2 , \, 
 1   1   2   2   1   1   1   2   1   2   2   1   1   1   1 , \,  \\ &
 2   2   2   1   1   1   1   2   1   2   1   1   2   2   1 , \, 
 2   1   1   2   2   1   1   1   2   1   2   2   1   1   1   2 , \, 
 2   1   1   1   1   2   2   1   2   1   1   1   2   2   1   2 , \, 
2   2   2   2   1   1   1   1   2   1   2   1   1   2   2   2   2, \, \\
&  \mbox{ and their reverses } \}
 \end{split}
 \end{equation*}
Therefore we will consider the words of length~$16$ and we begin by computing
the set of allowed words $A_F$, which do not contain a forbidden word as a
subword. The computation leaves us with $41186$ allowed words (down from $2^{16} =
65536$). We also compute 
the coefficients of the M\"obius maps corresponding to allowed compositions 
$T_{j_1} \circ \ldots \circ T_{j_{16}}$. 

Then we employ the algorithm described in \S\ref{sss:reducedMat} to identify the
words which give unique columns and rows to the Markov matrix together with
correspondence maps $R$ and $C$. It turns out that
there are $138$ such words. Finally, we calculate the Markov matrix itself. 
The computations we have done so far take less than a minute. 

Afterwards, we choose $m=8$, $t_0 = 0.5$ and compute the reduced matrix $\widehat B^t$ using the
formulae~\eqref{eq:Bsmall} and~\eqref{eq:Breduced} and find its eigenvector
using the power method. We work $128$ bit for precision and the eigenvector is
computed with an error of $10^{-26}$. We recover $138$ polynomials from the eigenvector
applying the formula~\eqref{eq:isoR2m}. Then we take a uniform partition of the
interval~$[0,1]$ into $256$ intervals and estimate the ratios $\frac{F_j}{f_j}$ 
on each of the intervals using ball-precision arithmetic. To obtain accurate
bounds on the numerator of the derivative $F_j^\prime f_j - f_j^\prime F_j$, we
compute the first~$4$ of its derivatives. We omit the eigenvector here, but we
note that for $t_0 = 0.5$ the leading eigenvalue of $\widehat B^{t_0}$ is
$1.0004258\ldots>1$ and the ratios satisfy 
$$
1.000425 < \frac{F^{t_0}_j}{f_j} < 1.000426 \mbox{ for } j = 1, \ldots, 138. 
$$
We then test another two values to get a more accurate estimate. For $t_1=
0.50001$ we have that the leading eigenvalue of $\widehat B^{t_1}$ is
$1.0002239\ldots>1$ and the ratios can be bounded as 
$$
1.000223 < \frac{F^{t_1}_j}{f_j} < 1.000225 \mbox{ for } j = 1, \ldots, 138. 
$$
For $t_2= 0.50005$ we get that the leading eigenvalue of $\widehat B^{t_2}$ is 
$0.99941699\ldots < 1$ and the ratios can be bounded as 
$$
0.999416 < \frac{F^{t_2}_j}{f_j} < 0.999418 \mbox{ for } j = 1, \ldots, 138. 
$$
It takes about 20 minutes to complete the estimates for a single value of~$t$
using $8$ threads running in parallel. 
\footnote{Computations were done using 4 Core 8 Threads Intel(R) Core(TM) i7-6700 CPU @ 3.40GHz }
\subsubsection{Set $B_2$}\label{B2set}
In this case we have $33$ forbidden words of length from $5$ to $18$. Thus we
consider the words of length $17$ and after removing those which contain a
forbidden word as a subword, obtain $79034$ allowed words. Among those we
identify $184$ words which give unique columns to the Markov matrix $M$ and
(another) $184$ words which give unique rows. Using the same parameters $m=8$
and $t_0=0.5$ we compute the matrix $\widehat B^{t_0}$ and its eigenvector with an
error of $10^{-40}$. The leading eigenvalue is $0.9996\ldots$ and after
another $40$ minutes we have lower and upper bounds on the ratios 
$$
0.999606 < \frac{F^{t_0}_j}{f_j} < 0.999607 \mbox{ for } j = 1, \ldots, 184. 
$$
Therefore we deduce that $\dim_H B_2 < 0.5$. In order to obtain more refined 
estimates we consider another two values  $t_1 = 0.499975$ and $t_2 = 0.499995$.
It turns out that the largest eigenvalue or $\widehat B^{t_1}$ is
$1.0001426\ldots > 1$ and we have the following bounds for the ratios
$$
1.000141 < \frac{F^{t_1}_j}{f_j} < 1.000143 \mbox{ for } j = 1, \ldots, 184. 
$$
The largest eigenvalue of $\widehat B^{t_2}$ is 
$0.99971391\ldots<1$ and the bounds for the ratios are 
$$
0.999713 < \frac{F^{t_2}_j}{f_j} < 0.999714 \mbox{ for } j = 1, \ldots, 184. 
$$
Therefore we conclude that $ 0.499975 <\dim_H B_2 < 0.499995$. 
It takes about $90$ minutes to obtain estimates on $\frac{F_j^t}{f_j}$ for 
a single value of~$t$. The time is evidently affected by the number of functions. 

\subsubsection{Set $X$}\label{setX}
The set~$X$ is specified by exclusion of $46$ words of length from~$5$ to~$24$. 
To compute its Hausdorff dimension we consider an iterated function scheme of
compositions of length~$23$. After removing all compositions which correspond to
forbidden words, we are left with $3940388$ maps, which is slightly less than a
half of $2^{23}$. The algorithm also identifies $429$ unique columns and rows in
the Markov matrix~$M$; some of them are repeated as many as $141030$ times.  
These computations take about~$5$ minutes. 

We then choose $m = 8$ and $t_0=0.5$ as initial dimension guess and work with
precision of $190$ bits. It takes about $2.5$ hours to compute the eigenvector
of the reduced matrix $\widehat B^{t_0}$ of dimension $429\cdot8$ with an error of
$10^{-40}$ and to obtain coefficients of $429$ polynomials of degree $7$. The
corresponding eigenvalue is $0.999973\ldots < 1$.

Most of the time is then taken by calculation of the images of these polynomials
under the map $\mathcal L_{t_0}$ as it involves taking compositions with all
$3940388$ maps. We use a partition of the interval $[0,1]$ into $1024$
intervals. The computation takes around $15$ days with~$8$ threads running in
parallel. Finally, we obtain 
$$
0.9999732 < \frac{F^{t_0}_j}{f_j} < 0.9999738 \qquad \mbox{ for } j = 1, \ldots, 429. 
$$
which allows us to conclude that $\dim_H X < t_0 = 0.5$.
\subsubsection{Set $Y$}\label{setY}
The set~$Y$ is specified by exclusion of $48$ words of length from~$5$ to~$24$. 
To compute its Hausdorff dimension we consider an iterated function scheme of
compositions of length~$23$. After removing all compositions which correspond to
forbidden words, we are left with $3940438$ maps, which is slightly less than a
half of $2^{23}$. The algorithm also identifies $434$ unique columns and rows in
the Markov matrix~$M$; some of them are repeated as many as $176015$ times.  
These computations take about~$5$ minutes. 

We then choose $m = 8$ and $t_0=0.5$ as initial dimension guess and work with
precision of $190$ bits. It takes about $2.5$ hours to compute the eigenvector
of the reduced matrix $\widehat B^{t_0}$ of dimension $434\cdot8$ with an error of
$10^{-40}$ and to obtain coefficients of $434$ polynomials of degree $7$. The
corresponding eigenvalue is $1.0000162\ldots > 1$.

Most of the time is then taken by calculation of the images of these polynomials
under the map $\mathcal L_{t_0}$ as it involves taking compositions with all
$3940438$ maps. This time in attempt to make the computation faster we use a
uniform partition of the interval $[0,1]$ into $256$
intervals. The computation takes around $3$ days with~$8$ threads running in
parallel. Finally, we obtain 
$$
1.0000160 < \frac{F^{t_0}_j}{f_j} < 1.0000166 \qquad \mbox{ for } j = 1, \ldots, 434. 
$$
which allows us to conclude that $\dim_H Y > t_0=0.5$.

\subsubsection{Set $\Omega$}\label{Omegaset}
The set~$\Omega$ is specified by exclusion of $26$ words of length from~$5$
to~$15$. To compute its Hausdorff dimension we consider an iterated function scheme of
compositions of length~$14$. After removing all compositions which correspond to
forbidden words, we are left with $45059$ maps, which is ten times less than $3^{14}$. 
The algorithm also identifies $114$ unique columns and rows in
the Markov matrix~$M$; some of them are repeated as many as $3745$ times, but
some occur only once. These computations take less than a minute. 

We then choose $m = 8$ and $t_0=0.5$ as initial dimension guess and work with
precision of $190$ bits. It less than a minute to compute the eigenvector
of the reduced matrix $\widehat B^{t_0}$ of dimension $114\cdot8$ with an error of
$10^{-40}$ and to obtain coefficients of $114$ polynomials of degree $7$. The
corresponding eigenvalue is $1.956\ldots > 1$. 

The subsequent estimates of the ratios $\frac{F_j^{t_0}}{f_j}$ using partition of
the interval $[0,1]$ into $256$ intervals take about $15$ minutes
and give 
$$
1.956990 < \frac{F^{t_0}_j}{f_j} < 1.9569915 \qquad \mbox{ for } j = 1, \ldots, 114. 
$$
We therefore conclude that $\dim_H \Omega > 0.5$ and apply bisection method to
get a better estimate. Taking the value $t_1=0.537152$ we get the leading eigenvalue 
$1.000031\ldots > 1$ and 
$$
1.0000315 < \frac{F^{t_1}_j}{f_j} < 1.0000320 \qquad \mbox{ for } j = 1, \ldots, 114. 
$$
and for $t_2=0.537155$ we get the leading eigenvalue of $\widehat B^{t_2}$ to be 
$0.999977 < 1$ and 
$$
0.999977 < \frac{F^{t_2}_j}{f_j} < 0.999979 \qquad \mbox{ for } j = 1, \ldots, 114. 
$$
We therefore conclude that $0.537152 <\dim_H \Omega < 0.537155$. 

\medskip 
This information is summarized in Table~\ref{tab:numerics} below.
\begin{table}[h!]
    \begin{tabular}{c|ccccc|cccc|c}
        Set       &  $\mathcal A$   & $\# F$ &     $n$     & $\#A_F$     & $K$ &
        $t$ & $\lambda_{max{ }}  $ & $r_1$ & $r_2$  & time  \\
        \hline
        \multirow{3}{*}{$B_1$}     &  \multirow{3}{*}{$\{1,2\}$}      &
        \multirow{3}{*}{$27$}   &    \multirow{3}{*}{$16$}  &
        \multirow{3}{*}{$41186$}     & \multirow{3}{*}{$138 $} &
        $0.5$     &   $1.000425$ & $1.000424$ & $1.000426$  &    \multirow{3}{*}{$15$mins}   \\
        & & & & & & 
        $0.50001$ &   $1.000223$ & $1.000222$ & $1.000225$  &       \\
        & & & & & & 
        $0.50005$ &   $0.999416$ & $0.499415$ & $0.499418$  &        \\
        \hline
        \multirow{3}{*}{$B_2$}     &  \multirow{3}{*}{$\{1,2\}$}      &
        \multirow{3}{*}{$33$}   &    \multirow{3}{*}{$17$}  &
        \multirow{3}{*}{$79034$}     & \multirow{3}{*}{$184$} &
        $0.5$     &   $0.999606$ & $0.999600$ & $0.999607$ &     \multirow{3}{*}{$90$ mins} \\
        & & & & & & 
        $0.499975$&   $1.000142$ & $1.000141$ & $1.000143$ &  \\ 
        & & & & & &
        $0.499995$&   $0.999713$ & $0.999712$ & $0.999714$ &    \\
        \hline
        $X$       &  $\{1,2\}$      & $46$   &    $23$  & $3940388$   & $429 $ & 
        $0.5$     &  $0.999973$ & $0.999972$ & $0.999974$  &    $4$ days\\      
        \hline
        $Y$       &  $\{1,2\}$      & $48$   &    $23$  & $3940438$   & $434 $ & 
        $0.5$     &  $1.000016$ & $1.000015 $ & $ 1.000017$ &  $4$ days\\          
        \hline
        \multirow{3}{*}{$\Omega$}  &  \multirow{3}{*}{$\{1,2,3\}$}    &
        \multirow{3}{*}{$26$}   &    \multirow{3}{*}{$14$}  &
        \multirow{3}{*}{$45059$}     & \multirow{3}{*}{$114 $} & 
        $0.5$     &  $1.956$  & $1.955$ & $1.957$ &
        \multirow{3}{*}{$12$mins}\\     
        & & & & & &
        $0.537152$&  $1.000031$ & $1.000030$ & $1.000032$ &   \\ 
        & & & & & &
        $0.537155$&  $0.999977$ & $0.999976$ & $0.999979$ &   \\ 
\end{tabular}
\caption{Numerical output of the algorithm for computing dimension of the sets
$B_1$, $B_2$, $X$, $Y$, $\Omega$. Time refers to the time needed to compute the
lower and the upper bounds on the ratios $r_1 < \frac{F_j^t}{f_j} < r_2$ for a
single value of~$t$.}
\label{tab:numerics}
\end{table}

\section{Some comments and open questions}\label{s.open-problems} Closing this article, let us briefly mention some problems left open by the present paper.  

\subsection{Modulus of continuity of the dimension function} The function 
$$f(t)=\dim_H \left( (-\infty,t] \cap M)\right)$$ is \emph{not} H\"older continuous at $3$ (cf.~\cite[p.147]{Mor18}), but an estimate on its modulus of continuity at any $t\in [3,c]$ was given in~\cite{Mor18}. In particular, it is not clear what should be expected about its local H\"older continuity properties at non-isolated points of $L$. 

\subsection{Interior of the intermediate of the classical spectra}  
The folklore conjecture that $(3, \sqrt{12})\cap L $ has non-empty interior is
natural because Perron showed that $[\sqrt{5},\sqrt{12}]\cap
M=\{m(\alpha):\{1,2\}^{\mathbb{Z}}\}$, so that $[\sqrt{5},\sqrt{12}]\cap M$ is
closely related to the arithmetic sum $E_2+E_2$ where $E_2=\{[0;a_1,a_2,\dots]:
a_n\in\{1,2\} \,\forall\,n\in\mathbb{N}\}$. In particular,
$[\sqrt{5},\sqrt{12}]\cap M$ should have non-empty interior because Marstrand's
projection theorem and Moreira--Yoccoz's work (cf.~\cite{MY}) say that $E_2+E_2$
is expected to contain intervals as $E_2$ is a ``nonlinear'' Cantor set with
Hausdorff dimension $>0.5$. Nonetheless, $E_2+E_2$ probably does \emph{not}
contain large intervals (as the Hausdorff dimension of $E_2$ is very close to
$0.5$), and, hence, it is not easy to convert this intuition into a concrete
result. On the other hand, the first two authors and L. Jeffreys hope to use
these ideas to establish in a future work that $L\cap [3,c]$ has non-empty
interior (where $c$ is Freiman's constant).  

\begin{remark} In a related direction, let us recall that
    Berstein~\cite{Berstein} conjectured in 1973 that $[4.1, 4.52]\subset L$. 
\end{remark}

\subsection{The Hausdorff dimension of $M\setminus L$} Despite our efforts to establish Theorem \ref{t.B}, we could not compute the first digit of $\dim_H (M \setminus L)$. Here, the answer to this problem (a digit in $\{5,6,7\}$) seems to depend on better bounds on $\sup(M\setminus L)$: for instance, if one can find new elements $M\setminus L$ ``sufficiently close'' to Freiman's constant $c=4.5278\dots$ (say, nearby $4.5251$), then it is likely that the first digit of $\dim_H (M \setminus L)$ is $7$. 


\newpage

\appendix

\newcommand*\Let[2]{\State #1 $\gets$ #2}
\algrenewcommand\algorithmicrequire{\textbf{Input:}}
\algrenewcommand\algorithmicensure{\textbf{Output:}}


\algrenewcommand\alglinenumber[1]{ \sf\scriptsize{#1}}
\newcommand{\nocontentsline}[3]{}
\newcommand{\tocless}[2]{\bgroup\let\addcontentsline=\nocontentsline#1{#2}\egroup}

\section{Pseudocode} 


\newlength{\myskip}
\myskip=3em

\begin{algorithm}
  \caption[Dimension of the Gau\ss--Cantor sets --- master program]{Computing the first~$K$ digits of the dimension of the
  Gauss--Cantor set specified by forbidden words. \label{alg:master}}
  \begin{algorithmic}[1]
     \Require{Alphabet $\mathcal A \subset \mathbb N$, finite list of forbidden words
      $FW:=\{fw_j\}_{j=1}^k \subset \mathcal A^n$, the desired accuracy~$K \in
      \mathbb N$ digits, lower bound $T_0$ (optional, default value $T_0=0$), upper
      bound $T_1$  (optional, default value $T_1=1$) }  
    \Statex
    \Program{GaussCantorSets}{$\mathcal A$, $FW$, $K$, $T_0 = 0$, $T_1 = 1$ }
    
 \algcommentleft{Global variables specifying the IFS:}  
 
    \State $sm$: int[n,4]  \Comment{ Coefficients of the maps of the IFS written in rows
     $(a,b,c,d)\Leftrightarrow x \to \frac{ax+b}{cx+d}$.}
    \State $rc$: int[u,maxmul] \Comment{ Indices of the identical rows in the Markov matrix,
    written in rows.}
    \State $cc$: int[n] \Comment{ For the column $k$ of the Markov matrix, cc[k] is the
    smallest index in its column class.}
    \State $rmm$: bool[u,u]  \Comment{ Reduced Markov matrix. }

 \algcommentleft{Compute all of them: }

    \Let{($sm$, $rmm$, $rc$, $cc$)}{\Call{CombinatoricSetup}{$\mathcal A, FW$}}

\algcommentleft{Now compute the dimension:}

    \Let{$\dim$}{\Call{ComputeDimension}{$K$, $T_0$, $T_1$}}
    \State \Return{$\dim$}
    \EndProgram
  \end{algorithmic}
\end{algorithm}

\begin{algorithm}
  \caption[Dimension of the Gau\ss--Cantor sets --- master program]{Computing the Hausdorff 
  dimension of the Gauss--Cantor set. Combinatorics setup: matrices that define
  the IFS and Markov condition. \label{alg:combi}}
  \begin{algorithmic}[1]
     \Require{Alphabet $\mathcal A \subset \mathbb N$, finite list of forbidden words
     $FW:=\{fw_j\}_{j=1}^k \subset \mathcal A^n$. }
    \Statex
    \Function{CombinatoricSetup}{$\mathcal A$, $FW$ }
    \Let{$n$ }{ $\max|fw_j|-1$.}
    \Let{$PF$ }{the prefix tree of forbidden words}
    \Let{$SF$ }{the suffix tree of forbidden words }
    \Let{$k$}{$0$}
    
    \Let{$ws$: int[$2^n$,$n$] }{all words in alphabet~$\mathcal A$ of length $n$.} 
    \For{$w \in ws$}
    \If{ $w$ doesn't contain a path in $SF$ from root to leaf }
    \Let{$k$}{$k+1$}
    \Let{$aw[k,:]$ }{$w$} 
    \Let{$sm[k,:]$ }{matrix of $T_{w_1}\circ T_{w_2} \ldots
    T_{w_n}$}
    \EndIf
    \EndFor \Comment{$k$ is the number of allowed words} 
 
    \For{$ w \in aw$}
    \State Walking the tree~$SF$, label by~$w$ the nodes with the property: \\
    \rule{80pt}{0pt}the path to the
    root is a prefix of~$w$, but paths to the children are not prefixes
    of~$w$
    \State Walking the tree~$PF$, label by~$w$ the nodes with the property: \\ 
    \rule{80pt}{0pt}the path to the
    root is a suffix of~$w$, but paths to the children are not suffixes of~$w$.
    \EndFor
    \Let{AP, AS}{Walk the trees $PF$ and $SF$, read equivalence classes from
    labels and apply encoding to elements \\ \rule{80pt}{0pt} of equivalence
    classes as described in \S4.3.2. Order the encoding pairs lexicographically.} 

    \Let{$N$}{total number of pairs $(j,k)$ in encoding of equivalence classes}

    \Let{$N_p$}{total number of classes $AP$}
    \Let{$N_s$}{total number of classes $AS$}


    \algstore{bkbreak}
    \end{algorithmic}
    \end{algorithm}

\setcounter{algorithm}{1}
   \begin{algorithm}
   \captionsetup{list=no}    
\caption{Continued}
  \begin{algorithmic}[1]
\algrestore{bkbreak}


    \Let{$M_p$: bool[$N_p$,$N$]}{ matrix indicating if a pair is in the
    corresponding class from $AP$ }
    \Let{$M_s$: bool[$N_s$,$N$]}{ matrix indicating if a pair is in the
    corresponding class from $AS$ }

    \Let{$rmm_0$}{Comparing the rows of $M_s$ and $M_p$, obtain compatibility matrix
    for classes}

    \Let{$rmm$}{The matrix of unique rows and columns of $rmm_0$}

    \EndFunction
  \end{algorithmic}
\end{algorithm}

\begin{algorithm}
 \caption[Computation of dimension --- main code]{Computing the first~$K$ digits of the Hausdorff dimension of the
limit set of the iterated function scheme of linear fractional transformations
given by~$sm$. The maps are divided in classes in two ways: 
according to the rows of transfer matrix as specified by~$rc$ and
according to the columns of transfer matrix as specified by~$cc$. 
The reduced Markov matrix is given by~$rmm$.\label{alg:main}}
  \begin{algorithmic}[1]
    \Function{ComputeDimension}{$K$, $T_0$, $T_1$}

    \algcommentleft{Initialisation:}
    
    \Let{$prec$}{$64$} \Comment{Initial precision, balancing the speed and
      accuracy.}
      \Let{$\beta$}{$\frac{T_0+T_1}2$} \Comment{Dimension guess}
    \Let{$it$}{$0$} \Comment{Number of attempts to compute approximation to eigenfunctions}
    \Let{$m$}{$6$} \Comment{Number of Chebyshev nodes}
    \Let{$\varepsilon$}{$10^{-K-1}$} \Comment{``Number of digits
    desired'' $\to$ accuracy}
    \Let{$lnp$}{$-8$} \Comment{Logarithm of the length of intervals in the
    partition used to obtain lower and upper bounds on functions using ball
    arithmetic} 

\algcommentleft{Local Variables:}
    \State $f$: array of~$u$ polynomials of degree~$m-1$;
    \State $\lambda$: eigenvalue of the matrix~$B^\beta$ approximating~$\mathcal
    L_\beta$.
\While{($|T_1-T_0|>\varepsilon$ and $it < 8$ and $lnp > -20$)} 

\algcommentleft{Main loop, the last two conditions ensure that the computation will terminate}

    \Let{$nuls$}{Zeros of the Chebyshev polynomial $T_m$ with precision $prec$} 
    \Let{($f$, $\lambda$, $lnp$, $sw$)}{\Call{LeadingVector}{ $\beta$, $m$, $nuls$, $lnp$, $prec$}}

    \Comment{$sw$ is the success switch} 
    \While{($sw\ne$ OK and $prec<2048$ and $m\cdot u < 3500$)} 
    \If{$sw=$ precision\textunderscore low} \Comment{Power method failed --- increase precision}
    \Let{$prec$}{$prec + 128$} 
    \Let{$nuls$}{Zeros of the Chebyshev polynomial $T_m$ with precision $prec$} 
    \EndIf
    \If{$sw=$ precision\textunderscore low\textunderscore and\textunderscore
    degree\textunderscore low} 
    \Let{$m$}{$m+2$} \Comment{Non-positive polynomial, increase degree and precision}

    \Let{$prec$}{$ prec + 128$} 
    \Let{$nuls$}{Zeros of the Chebyshev polynomial $T_m$ with precision $prec$} 
    \EndIf
    \If{$sw=$ degree\textunderscore low} \Comment{To save time increase the degree now for better approximation
    otherwise {\sc Ratbounds} will give $b_1 < 1 < b_2$}
    \Let{$m$}{$m+2$}
    \Let{$nuls$}{Zeros of the Chebyshev polynomial $T_m$ with precision $prec$} 
      \EndIf
    \Let{($f$, $\lambda$, $lnp$, $sw$)}{\Call{LeadingVector}{$\beta$, $m$, $nuls$, $lnp$, $prec$}}
    \Let{$it$}{$it+1$}
    \EndWhile
    \If{$sw=$ OK} \Comment{$f$ is an array of $u$ positive polynomials of
    degree~$m$}
    \Let{$(b_1,b_2)$}{\Call{Ratbounds}{ $f$, $\beta$, $nuls$, $m$, $lnp$, $prec$}  }
    
    \Comment{Lower and upper bounds  $b_1 < \frac{\mathcal
    L_{\beta}f_j}{f_j} < b_2$, $j = 1, \ldots, u$.}

     \algstore{bkbreak}
    \end{algorithmic}
    \end{algorithm}

\setcounter{algorithm}{2}
\begin{algorithm}
    \captionsetup{list=no}    
\caption{Continued}
\begin{algorithmic}[1]
\algrestore{bkbreak}
%
%
    \If{$b_1 > 1$} \Comment{ $\dim \in (\beta,T_1)$ }
    \Let{$T_0$}{$\beta$}
    \Let{$\beta$}{$\frac{T_0+T_1}2$}
    \ElsIf{$b_2 < 1$} \Comment{ $\dim\in (T_0,\beta)$}
    \Let{$T_1$}{$\beta$}
    \Let{$\beta$}{$\frac{T_0+T_1}2$}
    \Else \Comment{The assumption of min-max Lemma doesn't hold for $f$} 
    \Let{$lnp$}{$lnp-1$} \Comment{Reduce the size of partition intervals}
    \If{$m\cdot u < 3500$} 
    \Let{$m$}{$m+2$} \Comment{Increase the degree of approximating polynomials}
    \EndIf
    \Let{$it$}{$it+1$}
    \EndIf

    \EndIf
    \EndWhile

    \If {$|T_1 - T_0| < \varepsilon$}
    \State \Return{$\frac{T_1+T_0}2$}
    \Else
    \State \Return{ $-1$ }
    \Comment{Cannot achieve the desired accuracy}
    \EndIf

    \EndFunction
  \end{algorithmic}
\end{algorithm}

\begin{algorithm}
    \caption{Computation of the polynomials, approximating the
    eigenfunctions}\label{alg:poly}
    \begin{algorithmic}[1]
        \Function{LeadingVector}{$sm$, $rmm$, $rc$, $cc$, $\beta$, $m$, $nuls$, $prec$, $lnp$}
       
        \algcommentleft{Local variables:}
        \State $B$: arb[$m \cdot u$, $m\cdot u$] the matrix~$B^\beta$, approximating the
        transfer operator; 
        \State $sw$: success switch.
        \State $\ell p_j$, $j = 1,\dots, m$: Lagrange polynomials associated to~$nuls$.
        \Let{$B$}{\Call{ApproxMatrix}{$\beta$, $m$, $nuls$, $prec$}} 
        \Let{$v$}{1}
        
        \algcommentleft{Compute the leading eigenvector and its eigenvalue; $250$ iterations to kick off}
        \Let{$(v,\lambda,err)$}{\Call{PowerMethod}{$B$, $v$, $250$, $prec$}} 

        \algcommentleft{The degree of approximating polynomials and the allowed
        approximation error depend on $|\lambda-1|$; the choice of constants is
        based on heuristic experiments}
        \Let{$d_{\min}$}{$\max(6,-1.25\lg|\lambda-1|)$}
        \Let{$\delta$}{$\exp{(-\max(16,12\lg|\lambda-1|))}$} 

        \If{$d_{\min} < m $} 
        \State $sw = $ degree\textunderscore low {\bf goto} \fbox{exit} 
        \Comment{Increase the degree for better approximation.}
        \EndIf 

        \Let{$it$}{$0$}   
        \While{($ err > \delta$ and $it < 10$)}
        \Let{$it$}{$it+1$} \Comment{Attempting more iterations}
        \Let{$\delta_0$}{ $ err $} 
        \Let{$(v,\lambda,err)$}{\Call{PowerMethod}{$B$, $v$, $50$, $prec$}}
        \If{$ err > \delta_0$} \Comment{Attempt failed:} 
        \State $sw = $ precision\textunderscore low {\bf goto} \fbox{exit} 
        \Comment{error increased --- increase precision}
        \EndIf
        \EndWhile
        \If{$ err > \delta$} 
            \Comment{Power method failed --- increase precision} 
            \State $sw = $ precision\textunderscore low {\bf goto} \fbox{exit}
       \Else \Comment{We have good approximation to the true eigenvector}
       \Let{$f_j(x)$}{$\sum\limits_{k=1}^m v[(j-1)m+k] \cdot \ell p_k(x)$, $j = 1,\dots,u$}
        \If{\Call{AllPositive}{$f$, $lnp$, $prec$}} 
        \Comment{ {\bf NB:} $lnp$ might have changed}
        \State $sw = $ OK {\bf goto} \fbox{exit} 
        \Else  
        \Comment{Non-positive polynomial, increase degree and precision}
        \State $sw = $ precision\textunderscore low\textunderscore
        and\textunderscore degree\textunderscore low {\bf goto} \fbox{exit} 
        \EndIf
        \EndIf

        \State \fbox{exit} \Return{$(f,\lambda,lnp,sw)$}
        \EndFunction

        \Function{PowerMethod}{$B$,$v$,$n$,$prec$}
        \For{$j \gets 1 \mbox{ to } n$} 
        \Comment{All computations are done with   precision $prec$}
        \Let{$v$}{$\frac{B v}{\|B v\|_1}$ } \Comment{Here we use $\ell_1$-norm
        $\|v\|_1=\sum |v[j]|$ }
        \EndFor 
        \Comment{$v$ is approximation to the eigenvector}
        \Let{$\lambda$}{$\frac{\| Bv\|_1}{\|v\|_1}$}
        \Comment{The corresponding eigenvalue}
        \State \Return{ $(v,\lambda, \|Bv - \lambda v\|)$} \Comment{Here we use
        $\sup$ norm}
        \EndFunction
    \end{algorithmic}
\end{algorithm}

\begin{algorithm}
    \caption[Calculating approximation matrix]{Calculation of the matrix
    $B^\beta$: arb[$m\cdot u$,
    $m\cdot u$] approximating the operator $\mathcal L_\beta$ }
    \label{alg:matrix}
    \begin{algorithmic}[1]
        \Function{ApproxMatrix}{$\beta$, $m$, $nuls$, $prec$}

        \algcommentleft{ $rc$: int[$u$, $maxmul$] consists of indices 
        of the identical rows in the Markov matrix, written in rows.}
  
        \algcommentleft{$F_{rc[j,k]}$ is a map with coefficients from the row $sm[rc[j,k],\cdot]$.}  

        \Let{$B$}{$0$}
        \For{$k \gets 1 \mbox{ to } u $} \Comment{For each class of rows}
        \For{$k_0 \gets 1 \mbox{ to } rc[k, 1] $} \Comment{For each word in the
        class $k$} 
        \Let{$temp$}{$0$}  \Comment{A matrix arb[$m$, $m $]. }
            \For{$k_1 \gets 1 \mbox{ to } m $} 
            \Let{$(z_1,z_1^\prime)$}{$\left(F_{rc[k,k_0]}(nuls[k_1]),F_{rc[k,k_0]}^\prime(nuls[k_1])\right)$} 
            
            \Comment{$F_{rc[k,k_0]}$ and its derivative at every node} 
                \For{$k_2 \gets 1 \mbox{ to } m $} 
                    \Let{$temp[k_1,k_2]$}{$(-1)^{k_2}|z_1^\prime|^\beta\cdot
                    T_m(2z_1-1) \cdot 
                    \frac{\sqrt{nuls[k_2] - (nuls[k_2])^2}}{m(z_1 - nuls[k_2])}$}

 \Comment{Contribution to the matrix~$B$ from~$F_{rc[k,k_0]}$}
                \EndFor \Comment{$T_m$ is the Chebyshev polynomial of the first kind}
            \EndFor
            \For{$j \gets 1 \mbox{ to } u $}   
            \If{$rmm[j,cc[k,1]]$}          
            \Let{$B\left[(j-\!1)m+\!1\!:\!jm,(k-\!1)m+\!1\!:\!km\right]$}{$B\left[(j-\!1)m+\!1\!:\!jm,(k-\!1)m+\!1\!:\!km\right]+temp$}
            \EndIf \Comment{If transition between a word of class~$j$ and a word of class~$cc[k,1]$ is allowed,  add $temp$ to a block of  $B$, that corresponds to this transition. } 
            \EndFor 
        \EndFor
        \EndFor
        \State \Return{$B$}
        \EndFunction
    \end{algorithmic}
\end{algorithm}

\begin{algorithm}
    \caption[Verifying hypothesis of the min-max Lemma]{Computation of the lower and upper bounds on $\frac{[\mathcal L_\beta
f]_j}{f_j}$, $j = 1, \ldots u$.} \label{alg:min-max}
\begin{algorithmic}[1]
\Function{Ratbounds}{ $f$, $\beta$, $nuls$, $deg$, $lnp$, $prec$} 

\algcommentleft{Local Variables:}

\State $(p, cp)$: (arb[$2^{-lnp}$],arf[$2^{-lnp}$]) partition intervals and their
centres; 
\State $n$: floor$\bigl(\frac{deg+1}{2}\bigr)$ the number of the derivatives we
calculate to get an upper bound using the Taylor series; 
\State $tops$: arb[$u$, $2^{-lnp}$, $n+1$] is an array of $u$ matrices. 
For each $j$ the first column of $tops[j,:,:]$ contains the values
$[\mathcal L_\beta f]_j(cp)$. The second column contains the numerator of
$\bigl(\frac{(\mathcal L f)_j}{f_j}\bigr)^\prime(cp)$, i.e.  
$g_j(cp):=\left((\mathcal L_\beta f)_j^\prime f_j - (\mathcal L_\beta f)_j
f_j^\prime\right)(cp)$. The next $n-2$ columns contain derivatives
$g_j^{(k)}(cp)$, $k = 1, \dots n-2$. The last column contains
the $n-1$'th derivative on the interval $g_j^{(n-1)}(p)$. 
\Statex
\Let{$pow$}{$[-2\beta - d:-1:-2\beta-d-(n-1)]$}
\Let{$tops$}{\Call{Nbn}{$f$, $p$, $cp$, $pow$, $deg$, $lnp$}} \Comment{Calculate
$g_j^{(k)}$, $j = 0, \ldots n-1$ symbolically and evaluate $g_j^{(k)}(cp)$, $k =
0,\ldots, n-2$ and $g_j^{(n-1)}(p)$ } 
\Let{$(b_1,b_2)$}{$(100,-100)$} \Comment{Lower and upper bounds}
\For{$j \gets 1 \mbox{ to } u $}
    \For{$k \gets 1 \mbox{ to } 2^{-lnp}$}
    \Let{$y_1$}{$\frac{tops[j,k,1]}{f_{j}(cp[k])}$} \Comment{$\frac{(\mathcal
    L_\beta f)_{j}(cp[k])}{f_{j}(cp[k])}=y_1$} 
    \Let{$ny_1^\prime$}{\Call{UpperBound}{$tops[j,k,2:n+1]$,  $2^{-lnp-1}$,
    $n$}}

    \Comment{ $\left|\left((\mathcal L_\beta f)_{j}^\prime f_{j} - (\mathcal
    L_\beta
    f)_{j} f_{j}^\prime\right)(p[k])\right| < ny_1^\prime$}
    \Let{$\delta$}{$\frac{ny_1^\prime}{(f_{j}(p[k]))^2}\cdot 2^{-lnp-1}$}
    \If{$ 1 \in (y_1 - \delta, y_1 + \delta)$} \Comment{hypothesis
    of min-max Lemma failed} 
    \Let{$(b_1,b_2)$}{$(y_1 - \delta, y_1 + \delta)$} {\bf goto } \fbox{exit}
    \EndIf
    \If{$y_1 + \delta > b_2$} \Comment{$y_1 + \delta$ is the new upper bound}
    \Let{$b_2$}{$y_1 + \delta$}
    \EndIf
    \If{$y_1 - \delta < b_1$} \Comment{$y_1 - \delta$ is the new lower bound}
    \Let{$b_1$}{$y_1 - \delta$}
    \EndIf
    \If{$1 \in (b_1, b_2)$} 
    \Comment{hypothesis of min-max Lemma failed} 
    \State {\bf goto } \fbox{exit} 
    \EndIf
\EndFor

\EndFor
\State \fbox{exit} \Return{$(b_1, b_2)$}
\EndFunction

\algcommentleft{Compute lower an upper bound for the
function~$f$ on the interval $(c-r,c+r)$ from the values
$fn =
[f(c),f^\prime(c),f^{\prime\prime}(c),\ldots,f^{(n-2)}(c),f^{(n-1)}(c-r,c+r)]$}

\Function{UpperBound}{$fn$, $r$, $n$}

\For{$k \gets 0 \mbox{ to } n-2$}
\Let{$x$}{$\max|fn[n-k]|$}
\Let{$fn[n-k-1]$}{$(fn[n-k-1]-r\cdot x, fn[n-k-1]+r\cdot x)$}
\EndFor
\State\Return{$\max|fn[1]|$}
\EndFunction
\end{algorithmic}
\end{algorithm}

\begin{algorithm}
    \caption{Evaluation of $(\mathcal L_\beta f)^\prime_j f_j - (\mathcal L_\beta f)_j
    f_j $ and its derivatives on intervals of partition $(p,cp)$.}
    \label{alg:taylor}
    \begin{algorithmic}[1]
        \Function{Nbn}{$f$, $p$, $cp$, $pow$, $deg$, $lnp$}
        
        \algcommentleft{$pow = (-2\beta - d, - 2 \beta - d - 1, \ldots, -2\beta-d-(n-1))$}

        \algcommentleft{Local Variables:}
        \State $rp$: int[$u$]; $rp[j] \gets rc[k,2]$, where $k = k(j)$ is such that $j \in
        rc[k,:]$. 
        \State $yp$: arb[$2^{-lnp}$]
        \State $ycp$: arb[$2^{-lnp}, n$]
        \State $g_0(\cdot), g_1(\cdot)$: polynomials of degree $deg$
        \State $h(\cdot)$: polynomial of degree $2 deg$
        \State $lin(\cdot)$: linear function (polynomial of degree~$1$)
        \Statex
        \Let{$tops$}{$0$}
        \For{$k_0 \gets 1 \mbox{ to } N$} 
        \Let{$(g_0, lin)$}{\Call{San}{$f_{k_0}$, $sm[k_0,:]$, $deg$}}
        
        \Let{$yp[k_2]$}{$(lin(p[k_2]))^{pow[n]}$, $k_2 = 1, \ldots,   2^{-lnp}$}
        \Let{$ycp[k_2,k_1]$}{$(lin(cp[k_2]))^{pow[k_1]}$, $k_2=1, \ldots, 2^{-lnp}$, $k_1 = 1, \ldots,n-1$}

        \Let{$tops[k_3,k_2,1]$}{$rmm[k_3,cc[k_0]] \cdot g_0(cp[k_2])\cdot
        ycp[k_2,1]$,  $k_2 = 1 \ldots 2^{-lnp}$, $k_3 = 1 \ldots u $ }

      
        \Let{$g_1$}{$g^\prime \cdot lin -(2\beta+d) \cdot sm[k_0,3] \cdot  g_0  $}

        \For{$k_3 \gets 1 \mbox{ to } u$}
        \If{$rmm[k_3,cc[k_0]]$}

        \Let{$h$}{$g_1\cdot f_{k_3} - f_{k_3}^\prime \cdot lin \cdot g_0$} 
        \For{$k_1 \gets 2 \mbox{ to } n-1$}
        \For{$k_2 \gets 1 \mbox{ to } 2^{-lnp}$}
            \Let{$tops[k_3,k_2,k_1]$}{$tops[k_3,k_2,k_1] + 
            h(cp[k_2])\cdot ycp[k_2,k_1]$}
       \EndFor

       \Let{$h$}{$ sm[k_0,3] \cdot pow[k_1] \cdot h + h^\prime \cdot lin $}
       \EndFor
       \For{$k_2 \gets 1 \mbox{ to } 2^{-lnp}$}
        \Let{$tops[k_3,k_2,k_1]$}{$tops[k_3,k_2,k_1] + h(p[k_2])\cdot yp[k_2]$}
       \EndFor
       \EndIf
       \EndFor
       \EndFor
       \State \Return{$tops$}
       \EndFunction
       \Function{San}{$f$, $(a,b,c,d)$, $deg$} \Comment{Returns the
       numerator of the function $f\bigl(\frac{ax +b}{cx+d}\bigr)$.  }
       \Let{$qc$[$deg+1$]}{Coefficients of $f$}
       \Let{$lin(x)$}{$cx+d$}
       \Let{$g_0(x)$}{$\sum\limits_{j=0}^{deg} qc[j]\cdot  (ax+b)^j \cdot (cx+d)^{deg-j} $}
       \smallskip
       \State\Return{$(g,lin)$}
       \EndFunction
    \end{algorithmic}
\end{algorithm}

\begin{algorithm}
    \caption[Checking that approximating polynomials are all positive]{Checking
    that approximating polynomials $f_j$, $j=1,\dots,u$ are all positive.}
    \label{alg:extra}
    \begin{algorithmic}[1]
    \Procedure{AllPositive}{$f$, $lnp$, $prec$} 
    \For{$ j \gets 1 \mbox{ to } u$}
    \If{$f_j(0)<0$}
    \Let{$f_j$}{$-f_j$}
    \EndIf
    \EndFor
    \Let{$sw$}{YES}
    \Let{$lnp$}{$lnp+1$}
    \While{$sw=$ YES and $lnp > -20$} 
    \Let{$lnp$}{$lnp-1$}
    \Let{$p$}{Partition of the interval $[0,1]$ into~$2^{-lnp}$ intervals.}
    \Let{$sw$}{NO}
    \For{$ j \gets 1 \mbox{ to } u$} 
    \For{$ k \gets 1 \mbox{ to } 2^{-lnp}$}
    \Let{$sw$}{$0 \in f_j(p[k])$?} \Comment{Relying on the ball arithmetic from
    the arb library}
    \EndFor
    \EndFor
    \EndWhile
    \State \Return{$sw$, $lnp$}
    \EndProcedure
    \end{algorithmic}
\end{algorithm}

\end{document}